\documentclass[11pt,a4paper]{amsart}
\usepackage{amsfonts,amsthm, amscd, epsfig, amsmath, amssymb,enumerate,bbm}

\usepackage{hyperref}
\hypersetup{
    colorlinks=true,
    linkcolor=blue,
    citecolor=red,
    pdfpagemode=FullScreen,
    pdftitle={Aletti,Crimaldi,Ghiglietti}
    }

\usepackage{graphicx}
\usepackage{url}
\usepackage{bbm}

\usepackage{tikz}
\foreach \a in {q,w,e,r,t,y,u,i,o,p,a,s,d,f,g,h,j,k,l,z,x,c,v,b,n,m,Q,W,E,R,T,Y,U,I,O,P,A,S,D,F,G,H,J,K,L,Z,X,C,V,B,N,M}{
  \expandafter\xdef\csname v\a\endcsname {
		{\noexpand\boldsymbol{\a}}
	}
}

\newcommand{\vone}{{\boldsymbol{1}}}
\newcommand{\vzero}{{\boldsymbol{0}}}

\newcommand{\nper}{{n_{\text{per}}}}

\theoremstyle{plain}
\newtheorem{theorem}{Theorem}[section]
\newtheorem{lemma}[theorem]{Lemma}
\newtheorem{prop}[theorem]{Proposition}
\newtheorem{corollary}[theorem]{Corollary}

\theoremstyle{remark}
\newtheorem{remark}{Remark}

\theoremstyle{assumption}
\newtheorem{assumption}{Assumption}

\theoremstyle{definition}
\newtheorem{definition}{Definition}[section]
\newtheorem{example}{Example}[section]

\allowdisplaybreaks[3]

\begin{document}

\title[]{Networks of reinforced stochastic processes:\\ a complete
  description of the first-order asymptotics}

\author[G. Aletti]{Giacomo Aletti}
\address{ADAMSS Center,
  Universit\`a degli Studi di Milano, Milan, Italy}
\email{giacomo.aletti@unimi.it}

\author[I. Crimaldi]{Irene Crimaldi}
\address{IMT School for Advanced Studies Lucca, Lucca, Italy}
\email{irene.crimaldi@imtlucca.it}

\author[A. Ghiglietti]{Andrea Ghiglietti}
\address{Universit\`a degli Studi di Milano-Bicocca, Milan, Italy}
\email{andrea.ghiglietti@unimib.it (Corresponding author)}

\begin{abstract}
We consider a finite collection of reinforced stochastic processes
with a general network-based interaction among them.  We provide {\em
  sufficient and necessary} conditions in order to have some form of
{\em almost sure asymptotic synchronization}, which could be roughly
defined as the almost sure long-run uniformization of the behavior of
interacting processes. Specifically, we detect 
a regime of {\em complete} synchronization, where 
all the processes converge toward the same random variable, 
a second regime where the system almost surely converges, 
but there exists no form of almost sure asymptotic
synchronization, and another regime where the system does not converge
with a strictly positive probability. In this latter case,
partitioning the system in cyclic classes according to the period of
the interaction matrix, we have an almost sure
asymptotic synchronization within the
cyclic classes, and, with a strictly positive probability, an 
asymptotic periodic behavior of these classes.\\

\end{abstract}

\maketitle

\section{Introduction}
\label{intro}

The random evolution of systems composed by agents who interact among
each other has always been of great interest in several scientific
fields. For example, many studies in Biology focus on the interactions
between different sub-systems. Moreover, Economics and Social sciences
deal with agents that take decisions under the influence of other
agents. In social life, opinions are partly transmitted by means of
various forms of social interaction and are driven by the tendency of
individuals to become more similar when they interact
(e.g. \cite{alb-bar, new, hof}). Sometimes, a collective behavior,
usually called ``synchronization'', reflects the result of the
interactions among different individuals (we refer to \cite{are} for a
detailed and well structured survey on this topic, rich of examples
and references). \\ 
\indent In particular, in mathematical
literature, there exists a growing interest in systems of {\em
  interacting urn models} (see the subsection below 
  devoted to a
literature review). The present work is placed in the stream of this
scientific literature, which is mainly motivated by the attempt of
understanding the role of the reinforcement mechanism in
synchronization phenomena.\\ \indent Specifically, the present work
deals with the class of the so-called {\em interacting reinforced
  stochastic processes} introduced in~\cite{ale-cri-ghi,
  cri-dai-lou-min}.  Generally speaking, by (self-)reinforcement in a
stochastic dynamics we mean any mechanism for which the probability
that a given event occurs has an increasing dependence on the number
of times that the same event occurred in the past. This {\em
  ``reinforcement mechanism''}, also known as {\em ``preferential
  attachment rule''} or {\em ``Rich get richer rule''} or {\em
  ``Matthew effect''}, is a key feature governing the dynamics of many
random phenomena in different scientific areas, such as Biology,
Economics and Social sciences (see e.g. \cite{pem} for a general
survey).  Formally, in~\cite{ale-cri-ghi}, it is given the following
definition.
\begin{definition}\label{def-single}\rm
  A Reinforced Stochastic Process (RSP) is a sequence $X=(X_n)_{n\geq 1}$ of
  random variables with values in $\{0,1\}$ such that the predictive mean 
  $$ Z_n=E[X_{n+1}\vert  Z_0, X_1,\ldots, X_n]=P(X_{n+1}=1 \vert  Z_0,
  X_1,\ldots, X_n),\quad n\geq 0,$$ satisfies the dynamics
$$
Z_{n+1}=(1-r_n)Z_n+r_nX_{n+1},\quad\mbox{where } 0<r_n<1,
$$
and $Z_{0}$ is a random variable with values in $[0,1]$. 
\end{definition}
We can image that the process $X$ describes a sequence of actions
along the time-steps and, if at time-step $n$, the ``action 1'' has
taken place, i.e. $X_n=1$, then for ``action 1'' the probability of
occurrence at time-step $(n+1)$ increases.  Therefore, the larger
$Z_{n-1}$, the higher the probability of having $X_n=1$, and so the
higher the probability of having $Z_n$ greater than $Z_{n-1}$. As a
consequence, the larger the number of times in which $X_k=1$ with
$1\leq k\leq n$, the higher the probability $Z_n$ of observing
$X_{n+1}=1$.  \\ \indent For example, if we consider the sequence
$(X_n)_n$ of extractions associated to a time-dependent P\'olya urn
(see \cite{pemantle-time-dependent, sidorova}), with $\alpha_n>0$
the number of balls (of the color equal to the extracted one) added to
the urn at each time-step, and $s_0$ the initial number of balls in
the urn, the proportion of balls of the color, say $A$, associated to
the value $1$ for $X_n$, is $Z_n= (Z_0+\sum_{k=1}^n \alpha_k
X_k)/(s_0+\sum_{k=1}^n \alpha_k)$ and so
 $$
 Z_{n+1}=(1-r_n)Z_n+r_n X_{n+1},\quad\mbox{with}\qquad
 r_n=\frac{\alpha_{n+1}}{s_0+\sum_{k=1}^{n+1} \alpha_k}\,.
 $$ (The standard Eggenberger-P\'{o}lya urn \cite{egg-pol,mah}
 corresponds to $\alpha_n=\alpha$ for each $n$).  Therefore, $(X_n)_n$
 is a RSP. But, it is also true the converse (see \cite{ale-cri-ghi-barriers} for the details). 
 Therefore, the two notions are equivalent from a purely mathematical aspect, although
the dynamics of a RSP in Definition~\ref{def-single} has the merit to highlight the key role
played by the sequence $(r_n)_n$. In the sequel, we will refer to
   $(r_n)_n$ as the {\em reinforcement sequence}, since it regulates
   the reinforcement in the dynamics of $(Z_n)_n$, that is the weight 
   of the ``new information" (i.e. $X_{n+1}$) used to define the new status of the
process (i.e. $Z_{n+1}$). As we will see, the reinforcement sequence will 
be pivotal in the obtained results.
\\ 
   
  \indent In the present work we are interested in a
 system of $N\geq 2$ reinforced stochastic processes that interact
 according to a given set of relationships among them. More precisely,
 we suppose to have a finite directed graph $G=(V,\, E)$, with
 $V=\{1,...,N\}$ as the set of vertices and $E\subseteq V\times V$ as
 the set of edges.  Each edge $(l_1,l_2)\in E$ represents the fact
 that the vertex $l_1$ has a direct influence on the vertex $l_2$. We
 also associate a weight $w_{l_1,l_2}\geq 0$ to each pair
 $(l_1,l_2)\in V\times V$ in order to quantify how much $l_1$ can
 influence $l_2$.  A weight equal to zero means that the edge is not
 present. We define the matrix $W$, called in the sequel {\em
   interaction matrix}, as $W=[w_{l_1,l_2}]_{l_1,l_2\in V\times V}$
 and we assume the weights to be normalized so that $\sum_{l_1=1}^N
 w_{l_1,l_2}=1$ for each $l_2\in V$.  Hence, $w_{l,l}$ represents how
 much the vertex $l$ is influenced by itself and $\sum_{l_1=1,l_1\neq
   l}^N w_{l_1,l}\in [0,1]$ quantifies how much the vertex $l$ is
 influenced by the other vertices of the graph.  Finally, we suppose
 to have at each vertex $l$ a reinforced stochastic process described
 by $(X_{n,l})_{n\geq 1}$ such that, for each $n\geq 0$, the random
 variables $\{X_{n+1,l}:\,l\in V\}$ take values in $\{0,1\}$ and are
 conditionally independent given ${\mathcal F}_{n}$ with
\begin{equation}\label{interacting-1-intro}
P(X_{n+1,l}=1\, \vert \, {\mathcal F}_n)=\sum_{l_1=1}^N w_{l_1,l} Z_{n,l_1},
\end{equation}
where, for each $l\in V$,
\begin{equation}\label{interacting-2-intro}
Z_{n,l}=(1-r_{n-1})Z_{n-1,l}+r_{n-1}X_{n,l},
\end{equation}
with $0\leq r_n<1$, $\{Z_{0,l}:\,l\in V\}$ random variables with
values in $[0,1]$ and ${\mathcal F}_n=\sigma(Z_{0,l}: l\in V)\vee
\sigma(X_{k,l}: 1\leq k\leq n,\,l\in V )$.\\ \indent To express the
above dynamics in a compact form, let us define the vectors
$\boldsymbol{X}_{n}=[X_{n,1},..,X_{n,N}]^{\top}$ and
$\boldsymbol{Z}_{n}=[Z_{n,1},..,Z_{n,N}]^{\top}$.  Hence, for $n\geq
0$, the dynamics described by \eqref{interacting-1-intro} and
\eqref{interacting-2-intro} can be expressed as follows:
\begin{equation}\label{eq:cond-mean}
E[\boldsymbol{X}_{n+1}\vert \mathcal{F}_{n}]=W^{\top}\,\boldsymbol{Z}_{n},
\end{equation}
and 
\begin{equation}\label{eq:dynamics}
\boldsymbol{Z}_{n+1}\ =
\left(1-r_{n}\right)\boldsymbol{Z}_{n}\ +\ r_{n}\boldsymbol{X}_{n+1}.
\end{equation}
Moreover, the assumption about the normalization of the matrix $W$ can
be written as $W^{\top}\boldsymbol{1}=\boldsymbol{1}$, where
$\boldsymbol{1}$ denotes the vector with all the entries equal to $1$.
\\ \indent In order to fix ideas, 
we can imagine that $G=(V, E)$ represents a
network of $N\geq 2$ individuals that at each time-step have to make a
choice between two possible actions $\{0,1\}$. For any $n\geq 1$, the
random variables $\{X_{n,l}:\,l\in V\}$ describe the actions adopted by the agents of the network 
 at time-step $n$; while each random variable $Z_{n,l}$ takes values in
$[0,1]$ and it can be interpreted as the ``personal inclination'' of
the agent $l$ of adopting ``action 1''. Thus, the probability that the
agent $l$ adopts ``action 1'' at time-step $(n+1)$ is given by a
convex combination of $l$'s own inclination and the inclination of the
other agents at time-step $n$, according to the ``influence-weights''
$w_{l_1,l}$ as in \eqref{interacting-1-intro}. Note that, from a
mathematical point of view, we can have $w_{l,l}\neq 0$ or
$w_{l,l}=0$. In both cases we have a reinforcement mechanism for the
personal inclinations of the agents: indeed, by
\eqref{interacting-2-intro}, whenever $X_{n,l}=1$, we have
a strictly positive increment in the personal
  inclination of the agent $l$, that is $Z_{n,l}> Z_{n-1,l}$, provided
  $Z_{n-1,l}<1$. However, only in the case $w_{l,l}>0$ (which is the
  most usual in applications), this fact results in a greater
  probability of having $X_{n+1,l}=1$ according to
  \eqref{interacting-1-intro}. Therefore, if $w_{l,l}>0$, then we have
  a ``true self-reinforcing'' mechanism; while, in the opposite case,
  we have a reinforcement property only in the own inclination of the
  single agent, but this does not affect the probability
  \eqref{interacting-1-intro} of the action taken by this
  agent.\\ 

\indent In the considered
setting, the main goals are:
\begin{itemize}
\item[(1)] to understand whether and when a (complete or partial)
  almost sure asymptotic synchronization (that could
    be roughly defined as 
    the propensity of interacting agents to uniformize their behavior) 
    can emerge;
\item[(2)] to discover and to characterize which regimes may appear
  when the complete almost sure asymptotic synchronization does not
  hold.
\end{itemize}
All the above goals are achieved by performing a detailed analysis on
the interplay between the asymptotic behavior of the system and the
properties of the interaction matrix $W$ and of the reinforcement
sequence $(r_n)_n$.\\ \indent First of all, the present paper provides 
the sufficient and necessary conditions in order to have the
almost sure asymptotic synchronization of the entire system ({\em
  complete almost sure asymptotic synchronization}), that is the
almost sure converge toward zero of all the differences
$(Z_{n,l_1}-Z_{n,l_2})_n$, with $l_1,\,l_2\in V$. Firstly, we observe
that, in the considered setting, the complete almost sure asymptotic
synchronization of the system is equivalent to the almost sure
convergence of all the RSPs $(Z_{n,l})_n$, with $l\in V$, toward a
certain common random variable $Z_\infty$. Then, under the assumption
that $W$ is irreducible (i.e.~$G=(V,E)$ is a strongly connected graph)
and $P(\vZ_0=\vzero)+P(\vZ_0=\vone)<1$\footnote{Similarly to the
  notation $\vone$ already mentioned above, the symbol $\vzero$
  denotes the vector with all the entries equal to $0$.} (in order to
exclude the trivial cases), we prove that:
\begin{itemize}
\item[(i)] when $W^{\top}$ is aperiodic, we have complete almost sure asymptotic
  synchronization if and only if $\sum_n r_n =+\infty$;
\item[(ii)] when $W^{\top}$ is periodic, we have complete almost sure asymptotic
  synchronization if and only if $\sum_n r_n(1-r_n) =+\infty$.
\end{itemize}
Therefore $\sum_n r_n=+\infty$ results to be a necessary conditions on the
reinforcement sequence $(r_n)_n$ for the complete almost sure asymptotic
synchronization. Indeed, when $\sum_n r_n<+\infty$,
all the stochastic processes 
$\{(Z_{n,l})_n:\,l\in V\}$ trivially converge almost surely,
but for any pair of distinct nodes 
we have a strictly positive probability of non-synchronization
in the limit.  More interesting is the regime when $W^{\top}$ is
periodic, $\sum_n r_n=+\infty$ and $\sum_n r_n(1-r_n)<+\infty$ 
(which actually means that $r_n$ is either stays very close to $0$ or very close to $1$). 
In this case, we
show that a {\em partial almost sure asymptotic synchronization} holds: 
indeed, the RSPs positioned at the network vertices almost surely synchronize in
the limit inside each cyclic class of $W^{\top}$, that is the
difference $(Z_{n,l_1}-Z_{n,l_2})_n$ converges almost surely to zero
whenever $l_1$ and $l_2$ belong to the same cyclic class, but the
convergence of $(Z_{n,l})_n$ is not guaranteed. Indeed, we prove that 
there exists a strictly positive probability that $(\vZ_n)_n$ does not converge, 
because of an asymptotic periodic behavior of the cyclic  classes 
(see Theorem~\ref{teo:PeriodicCycle}). Moreover, the limit set of each cyclic class 
is given by the barrier-set $\{0,1\}$, that is,  
for a large time-step, there are some cyclic classes in which the processes 
are very close to one and other classes in which the processes are very close to zero 
(see Theorem~\ref{th-limit-class-0-1}). Further, in this regime,  
we have a clockwise dynamics also for the agents' actions 
(i.e. for the random variables $X_{n,l}$) and   
their synchronization inside the cyclic classes at the clock-times (see again  Theorem~\ref{teo:PeriodicCycle}). 
The behavior we have in this last regime seems to 
well describe some social phenomena where the shift between different groups' polarizations is present 
along time.
\\
\indent  It is worthwhile to point out that these results have been achieved by means of a 
 suitable decomposition of $\vZ_n$ in terms of three components (see
Theorem~\ref{teo:sp-project}): the first related to the eigenvalue $1$
(i.e. the Perron-Frobenius or leading eigenvalue) of $W^\top$, the
second related to the cyclic classes of $W^\top$, that is to the
eigenvalues of $W^\top$, different from $1$, but with modulus equal to
one, and the third component related to the other eigenvalues of
$W^\top$, that is those with modulus strictly less than one. We study
the behavior of all the three components and how they affect the
behavior of the system (see Theorems~\ref{teo:thirdpart} and
Theorem~\ref{teo:secondpart}). The first-order asymptotic behavior of $(\vZ_ n)_n$ for the different regimes considered is completely
described and summarized in Table~\ref{tab:conv-synchro}.
\begin{table}[h]
\resizebox{\textwidth}{!}{
\begin{tabular}{|c||c|c|}
\hline
 & 
$\nper = 1$ & 
$\nper \geq 2$ \\
\hline
\hline
$\sum_n r_n < +\infty$ &
\begin{tabular}{c} $\bullet$ a.s. convergence \\ $\bullet$ non-a.s. synchro.  \\ (neither complete nor partial)\end{tabular} &
\begin{tabular}{c} $\bullet$ a.s. convergence \\ $\bullet$ non-a.s. synchro.  \\ (neither complete nor partial)\end{tabular} 
\\
\hline
\begin{tabular}{c} $\sum_n r_n(1-r_n) = +\infty$ \\ (and so $\sum_n r_n = +\infty$) \end{tabular} &
\begin{tabular}{c} $\bullet$ a.s. convergence \\ $\bullet$ complete a.s. synchro.\end{tabular} &
\begin{tabular}{c} $\bullet$ a.s. convergence \\ $\bullet$ complete a.s. synchro.\end{tabular} 
\\
\hline
\begin{tabular}{c} $\sum_n r_n = +\infty$ \\ $\sum_n r_n(1-r_n) < +\infty$ \end{tabular} &
\begin{tabular}{c} $\bullet$ a.s. convergence \\ $\bullet$ complete a.s. synchro.\end{tabular} &
\begin{tabular}{c} $\bullet$ non-a.s. convergence\\
(asymp. periodic behavior of\\ the cyclic classes) \\ $\bullet$ non-a.s. complete synchro.
  \\ $\bullet$ a.s. synchro. within \\ the cyclic classes \end{tabular} \\
\hline
\end{tabular}
}
\caption{Summary of the possible first-order 
asymptotic behavior of $(\vZ_n)_n$,
  under  the assumption $P(\vZ_0=\vzero)+P(\vZ_0=\vone)<1$. The symbol
  $\nper$ denotes the period of the matrix $W^{\top}$.}
\label{tab:conv-synchro}
\end{table}
We also discuss how the assumption of $W^\top\vone=\vone$ can be
relaxed through a natural generalization of the theory presented in
this work.  \\ 

\indent We also note that, when $(\vZ_n)_n$ almost surely converges toward a
random vector $\vZ_\infty$ (equal to $Z_\infty\vone$ in the case of
complete almost sure asymptotic synchronization), the average of times
in which the agents adopt ``action 1'', that is the vector of the
empirical means $N_{n,l}=\sum_{k=1}^nX_{k,l}/n$, with $l\in V$, almost
surely converges toward $W^\top\vZ_\infty$ (equal again to
$Z_\infty\vone$ in the case of complete almost sure asymptotic
synchronization). This result is important from an applicative point
of view, because the observable processes are typically the processes
$(X_{n,l})_n$ of the performed actions and not the processes
$(Z_{n,l})_n$ of the inclinations. Therefore, in a statistical
framework, the vector of the empirical means can be used as a strongly
consistent estimator of the random vector $W^\top \vZ_\infty$. When we
only have the almost sure asymptotic synchronization of the
inclinations $Z_{n,l}$ within the cyclic classes, we also have the
same form of asymptotic synchronization for the empirical means
$N_{n,l}$. \\ \indent Regarding the assumption of irreducibility, we
observe that when $W$ is reducible, using a suitable decomposition of
$W^\top$, we obtain a natural decomposition of the graph $G=(V,E)$ in
different sub-graphs $\{G_s;\,1\leq s\leq m\}\cup G_f$ (where $m$ is
the multiplicity of the eigenvalue $1$ of $W^\top$) such that:
\begin{itemize}
\item[(i)] the nodes in each sub-graph $G_s$ are
not influenced by the nodes of the rest of the network, and hence the
dynamics of the processes in $G_s$ can be fully established by
considering only the correspondent irreducible sub-matrix of $W^\top$ and so by
applying the results presented in the present paper to each sub-graph $G_s$;
\item[(ii)] the nodes in the sub-graph $G_{f}$ are affected by the
  behavior of the rest of the network and hence they have to be
  treated according to the specific situation considered.
\end{itemize} 
Finally, when the assumption of normalization for $W$, that
is $W^\top\vone = 1$, is not verified, the first-order asymptotic dynamics of the system 
become trivial: indeed, we have the compleyte almost sure asymptotic synchronization 
toward zero, as the result of the action of a "forcing input".


\subsection*{Literature review}
Regarding the previous literature, it is important to mention that
there are other works concerning models of interacting urns, that
consider reinforcement mechanism and/or interacting mechanisms
different from the model presented in this paper.  For instance, the
model studied in \cite{mar-val} describes a system of interacting
units, modeled by P\'olya urns, subject to perturbations and which
occasionally break down. The authors consider a system of interacting
P\'olya urns arranged on a $d$-dimensional lattice. Each urn contains
initially $b$ black balls and $1$ white ball. At each time step an urn
is selected and a ball is drawn from it: if the ball is white, a new
white ball is added to the urn; if it is black a ``fatal accident''
occurs and the urn becomes unstable and it ``topples'' coming back to
the initial configuration.  The toppling mechanism involves also the
nearby urns.\\ \indent In \cite{pag-sec} a class of discrete time
stochastic processes generated by interacting systems of reinforced
urns is introduced and its asymptotic properties analyzed. Given a
countable set of urns, at each time a ball is independently sampled
from every urn in the system and in each urn a random number of balls
of the same color of the extracted ball is added. The interaction
arises since the number of added balls depends also on the colors
generated by the other urns as well as on a common random factor.
\\ \indent In \cite{fortini} a notion of partially conditionally
identically distributed (c.i.d.) sequences has been studied as a form
of stochastic dependence, which is equivalent to partial
exchangeability in the presence of stationarity. A natural example of
partially c.i.d. construction is given by a countable collection of
stochastic processes with reinforcement, and possibly infinite state
space, with an interaction among them obtained by inserting in their
dynamics some stocastically dependent random weights. This example
contains as a special case the model in \cite{pag-sec}. \\ \indent
Interacting two-colors urns have been considered in \cite{lau2,
  lau1}. Their main results are proved when the probability of drawing
a ball of a certain color is proportional to $\rho^k$, where $\rho>1$
and $k$ is the number of balls of this color. The interaction is of
the \emph{mean-field} type. More precisely, the interacting
reinforcement mechanism is the following: at each step and for each
urn draw a ball from either all the urns combined with probability
$p$, or from the urn alone with probability $1-p$, and add a new ball
of the same color to the urn. The higher is the interacting parameter
$p$, the larger is the memory shared between the urns. The main
results can be informally stated as follows: if $p\geq 1/2$, then all
the urns fixate on the same color after a finite time, and, if
$p<1/2$, then some urns fixate on a unique color and others keep
drawing both colors.  \\ \indent In \cite{cir} the authors consider a
network of interacting urns displaced over a lattice. Every urn is
P\'olya-like and its reinforcement matrix is not only a function of
time (time contagion) but also of the behavior of the neighboring urns
(spatial contagion), and of a random component, which can represent
either simple fate or the impact of exogenous factors. In this way a
non-trivial dependence structure among the urns is built, and the
given construction is used to model different phenomena characterized
by cascading failures such as power grids and financial networks.
\\ \indent In \cite{ben, che-luc, lima} a graph-based model, with urns
at each vertex and pair-wise interactions, is considered. Given a
finite connected graph, an urn is placed at each vertex. Two urns are
called a pair if they share an edge. At discrete times, a ball is
added to each pair of urns. In a pair of urns, one of the urns gets
the ball with probability proportional to its current number of balls
raised by some fixed power $\alpha>0$. The authors characterize the
limiting behavior of the proportion of balls in the bins for different
values of the parameter $\alpha$.  \\ \indent In \cite{super-urn-1,
  super-urn-2, super-urn-3} another graph-based model, with P\'olya
urns at each vertex, is provided in order to model the diffusion of an
epidemics. Given a finite connected graph, an urn is placed at each
vertex and, in order to generate spatial infection among neighboring
nodes, instead of drawing solely from its own urn, each node draws
from a ``super urn'', whose composition is the union of the
composition of its own urn and of those of its neighbors' urns. The
stochastic properties and the asymptotic behavior of the resulting
network contagion process are analyzed.  \\ 
\indent In
\cite{cri-dai-min, dai-lou-min, sah} the authors consider interacting
urns (precisely, \cite{cri-dai-min, dai-lou-min} deal with standard
P\'olya urns and \cite{sah} regards Friedman's urns) in which the
interaction is of the mean-field type: indeed, the urns interact among
each other through the average composition of the entire system, tuned
by the interaction parameter $\alpha$, and the probability of drawing
a ball of a certain color is proportional to the number of balls of
that color. Asymptotic synchronization and central limit theorems for
the urn proportions have been proved for different values of the
tuning parameter $\alpha$, providing different convergence rates and
asymptotic variances. In \cite{cri-dai-lou-min} the same mean-field
interaction is adopted, but the analysis has been extended to the
general class of reinforced stochastic processes (in the sense of
Definition~\ref{def-single}), providing almost sure asymptotic
synchronization of the entire system and central limit theorems, also
in functional form, in the case $\lim_n n^{\gamma}r_n=c$ with $c>0$
and $1/2<\gamma\leq 1$.  Differently from these works, the model
proposed in \cite{ale-ghi} concerns with a system of generalized
Friedman's urns with irreducible mean replacement matrices based on a
general interaction matrix. Combining the information provided by the
mean replacement matrices and by the interaction matrix, first and
second-order asymptotic results of the urn proportions have been
established.  In this framework the non-synchronization is a natural
phenomenon since the mean replacement matrices are irreducible (and so
not diagonal) and are allowed to be different among the nodes, hence
the limits of the urn proportions are deterministic and possibly
different.  \\ \indent The paper \cite{ale-cri-ghi} joins the class of
reinforced stochastic processes studied in \cite{cri-dai-lou-min} with
the general interacting framework, driven by the interaction matrix,
adopted in \cite{ale-ghi}. After proving complete almost sure
asymp\-totic synchronization for an irreducible diagonalizable
interacting matrix and for $\lim_n n^{\gamma}r_n=c$ with $c>0$ and
$1/2<\gamma\leq 1$, \cite{ale-cri-ghi} provides the rates of
synchronization and the second-order asymptotic distributions, in
which the asymptotic variances have been expressed as functions of the
parameters governing the reinforced dynamics and the eigen-structure
of the interaction matrix. These results lead to the construction of
asymptotic confidence intervals for the common limit random variable
of the processes $\{(Z_{n,l})_n:\,l\in V\}$ and to the design of
statistical tests to make inference on the topology of the interaction
network given the observation of the processes $\{(Z_{n,l})_n:\, l\in
V \}$.  Finally, in \cite{ale-cri-ghi} the non-synchronization
phenomenon is discussed only as a consequence of the
non-irreducibility of the interaction matrix, which leads the system
to be decomposed in sub-systems of processes evolving with different
behaviors.  \\ 
\indent The previous quoted papers focus on the
asymptotic behavior of the stochastic processes of the personal
inclinations $\{(Z_{n,l})_n:\, l\in V \}$ of the agents, while
\cite{ale-cri-ghi-MEAN, ale-cri-ghi-WEIGHT-MEAN} study different
averages of times in which the agents adopt ``action 1'', i.e. the
stochastic processes of the empirical means and weighted empirical
means associated with the random variables $\{(X_{n,l})_n:\, l\in
V\}$.\\ 
\indent In addition, inspired by models for coordination games,
technological or opinion dynamics, the paper \cite{cri-lou-min} deals
with stochastic models where, even if a mean-field interaction among
agents is present, the absence of asymptotic synchronization may
happen due to the presence of an individual non-linear reinforcement.
\\
\indent Finally, in the recent paper \cite{KauSah21}, a model for 
interacting balanced urns is introduced and analyzed. The balanced 
reinforcement matrix is assumed the same for all the urns and the interaction 
among the urns is based on a finite directed graph (not weighted), 
in the sense that each urn positioned at a node of the graph reinforces all the urns 
in its out-neighbours. The authors show conditions on the reinforcement matrix, the topology of 
the graph and the initial configuration  in order to have the almost sure convergence of the system and 
a form of almost sure asymptotic synchronization. They also provide some fluctuation theorems.\\ 
\indent The present paper have some issues in common with
\cite{ale-cri-ghi}, although at the same time several significant
differences can be pointed out and, in particular, the intent of the
this work is different.  Indeed, here we are not only interested in
providing sufficient assumptions for the complete almost sure
asymptotic synchronization of the system and proving theoretical
results under these assumptions, but we aim at getting a complete
description of the first-order asymptotic behavior of the system,
which consists in finding sufficient and also necessary conditions on
the interaction matrix and on the reinforcement sequence for the
(complete or partial) almost sure asymptotic synchronization and also
describing the behavior of the system in the non-synchronization and
non-convergence regimes. In particular, it is very important to
underline that, with respect to \cite{ale-cri-ghi}, in the present
paper:
\begin{itemize}
\item[(1)] we eliminate the technical assumption that the interaction
  matrix is diagonalizable (indeed, this condition is difficult to be
  checked in practical applications and so it may lead to consider the
  symmetric interaction matrices as the only class of ``applicable''
  matrices);
\item[(2)] the results provided here include not only the case when
  $\lim_n n^{\gamma}r_n=c$ with $c>0$ and $1/2<\gamma\leq 1$
  considered in the previous papers mentioned above, but also the case
  with $0<\gamma\leq 1/2$;
\item[(3)] we do not limit ourselves to the case of reinforcement sequences
  $(r_n)_n$ such that $\lim_n n^{\gamma}r_n=c$ with $0<\gamma\leq 1$
  (indeed, as we will show, this assumption always implies the
  convergence of the system and so it excludes other possible behaviors
  that the system can present);
\item[(4)] we describe the behavior of the system in the regimes where
  the complete almost sure asymptotic synchronization does not hold
  (this includes the regime with almost sure convergence without any
  form of almost sure synchronization and the regime with a periodic
  dynamics and a partial almost sure synchronization, specifically
  synchronization within the cyclic classes).
\end{itemize}

\indent Regarding the methodology, we point out that there exists a vast literature 
where the asymptotics for discrete-time processes, in particular urn processes, 
are proven through ordinary differential equation (ODE) method and 
stochastic approximation theory
that may seem to be applicable to the dynamics considered in this work as well 
(see e.g \cite{benaim-sem, Sch01} and references therein). 
However, with these techniques it would be impossible to fully characterize all the 
possible regimes shown by the system.
For instance, the non-convergent periodic dynamics we present in Theorem \ref{teo:PeriodicCycle}
is not contemplated in an ODE framework,
which shows that this theory fails in our context.


\subsection*{Structure of the paper} The sequel of the paper is so structured.
In Section \ref{sec:decomp-conv-sincro} we suitably decompose the
process $(\vZ_n)_n$ and, after characterizing the (complete or
partial) almost sure asymptotic synchronization of the system in terms
of the components of this decomposition, we provide sufficient and
necessary conditions related with this phenomenon. In
Section~\ref{sec:non-sincro} we deal with the regimes where the
complete almost sure asymptotic synchronization of the system is not
guaranteed, discussing the case of almost sure convergence without any
form of almost sure asymptotic synchronization and the case of
non-almost sure convergence (nor complete synchronization) of the
system.
Section~\ref{sec:additional} contains a remark on the behavior
of the (weighted) empirical means associated to the random variables
$\{(X_{n,l})_n:\, l\in V\}$, a discussion on how to handle the case of
a reducible interaction matrix and a final comment on
  what happens when the assumption of normalization for $W$, that is
  $W^\top\vone=\vone$, is not verified. 
  Finally, Section~\ref{sec:proofs}
  contains the main ideas and the sketches of the proofs of the results stated in the present work. 
  All the details of the proofs can be found in a separate supplementary material
\cite{ale-cri-ghi-supplSPA1}, along with some recalls and auxiliary results.


\section{Almost sure asymptotic synchronization}
\label{sec:decomp-conv-sincro}

Consider a system of $N\geq 2$ Reinforced Stochastic Processes (RSPs)
with a network-based interaction as defined
in~\eqref{interacting-1-intro} and \eqref{interacting-2-intro} or,
equivalently, in~\eqref{eq:cond-mean} and \eqref{eq:dynamics}, where
the matrix $W$ has non-negative entries, is irreducible (i.e.~the
underlying graph~$G=(V,E)$ is strongly connected) and is such that
$W^{\top}\boldsymbol{1}=\boldsymbol{1}$ (where $\vone$ denotes the
vector with all the entries equal to $1$).\\

In this section, we use the assumed properties of $W^\top$ in order to
decompose the process $(\vZ_n)_n$ into three components. We will claim
that the behavior of these components depends only on the summability
of the sequences $(r_n)_n$ and $(r_n(1-r_n))_n$, according to the period of
$W^\top$.  To this end, we recall that the Perron-Frobenius
Theorem ensures that:
\begin{enumerate}
\item the eigenvalue $1$ of $W^\top$ has multiplicity one and
it is called the {\em Perron-Frobenius or leading eigenvalue},
as all the other eigenvalues have real part strictly less than 1. 
Moreover, there exists a (unique) left eigenvector
  $\vv $ associated to the leading eigenvalue of $W^\top$ with all the
  entries in $(0,+\infty)$ and such that $\vv^{\top}\boldsymbol{1}=1$;
\item if $W^\top$ is periodic with period ${\nper} \geq 2$, 
\begin{itemize}
\item[(i)] all the complex ${\nper}$-roots of the unity are eigenvalues 
  of $W^\top$ with multiplicity one;
\item[(ii)] there exists an equivalence relationship $\sim_c$ that
  separates the set of vertices $V=\{1,\ldots,N\}$ (i.e. the processes
  $\{(Z_{n,l})_n:\,l\in V\}$) into $\nper$ classes of equivalence,
  numbered from $0$ to $\nper-1$ and called \emph{cyclic classes}, such
  that $[W^\top]_{l_1,l_2}=0$ when $l_1$ belongs to the $h$-th cyclic
  class and $l_2$ does not belong to the $(h+1)$-th cyclic class
  (the numbers of the classes being defined modulus $\nper$). They
  correspond to the communicating classes of the matrix
  $(W^{\top})^{\nper}$;
\end{itemize}
\item except the complex ${\nper}$-roots of the unity, all the other
  eigenvalues have modulus strictly less than $1$.
\end{enumerate}

\indent The \textbf{first component of the decomposition} of the
process $(\vZ_n)_n$ concerns the eigenspace associated to the
Perron-Frobenius eigenvalue.  More precisely, the spectral
non-orthogonal projection of $W^\top$ corresponding to the leading
eigenvalue is the matrix $\vone \vv^\top$, that we apply to $\vZ_n$ to
obtain the process
\begin{equation}\label{def:Z_1}
{\vZ}_n^{(1)} = \vone \vv^\top \vZ_n = \widetilde{Z}_n \vone ,
\end{equation}
where $(\widetilde{Z}_n )_n$ is the bounded martingale defined as
$\widetilde{Z}_n=\vv^{\top}\vZ_n$ for each $n$ and with dynamics
\begin{equation}\label{Z-tilde-dynamics}
\widetilde{Z}_0=\boldsymbol{v}^{\top}\boldsymbol{Z}_0,\qquad
\widetilde{Z}_{n+1}=(1-r_n)\widetilde{Z}_n+r_nY_{n+1},
\end{equation}
where $Y_{n+1}=\boldsymbol{v}^\top\boldsymbol{X}_{n+1}$ takes values
in $[0,1]$ and $E[Y_{n+1}\vert {\mathcal F}_n]=\boldsymbol{v}^\top
W^{\top}\boldsymbol{Z}_n=\widetilde{Z}_n$.  This martingale plays a
central role: indeed, whenever all the stochastic processes
$\{(Z_{n,l})_n,\, l\in V\}$ converge almost surely to the same limit random
variable, say $Z_\infty$, we trivially have that $Z_\infty$ coincides
with the almost sure limit of $(\widetilde{Z}_n)_n$.
\\

\indent The \textbf{second component of the decomposition} of the
process $(\vZ_n)_n$ exists only when $\nper\geq 2$ and it concerns the
cyclic classes.
Specifically, the process ${\vZ}_n^{(2)}$ is defined as
\begin{equation}\label{def:Z_2}
{\vZ}_n^{(2)} = {\vZ}_n^{(C)} - \widetilde{Z}_n \vone,
\qquad\mbox{where}\qquad
{Z}_{n,l}^{(C)} = 
\sum_{l_1 \sim_c l} \frac{v_{l_1}}{
\sum_{l_2 \sim_c l} v_{l_2}
} Z_{n,l_1}, 
\qquad l \in V.
\end{equation}
Note that the process ${\vZ}_n^{(C)}$, and so ${\vZ}_n^{(2)}$, is
constant on each cyclic class as the terms $\sum_{l_1 \sim_c l}
v_{l_1}Z_{n,l_1}$ and $\sum_{l_2 \sim_c l} v_{l_2}$ are the same for
all the vertices $l$ in the same cyclic class. We will prove that this
part may be defined by means of the spectral non-orthogonal projection
on the eigenspaces corresponding to the complex $(\nper-1)$ roots of the 
unity different from $1$ (see Theorem~\ref{teo:sp-project}).  \\

\indent The \textbf{third component of the decomposition} of the
process $(\vZ_n)_n$ is the remaining part and it concerns the spectral
non-orthogonal projection on the eigenspaces corresponding to the
roots of the characteristic polynomial that have modulus strictly less
than $1$ (see Theorem~\ref{teo:sp-project}).  \\

\indent The full decomposition then reads
\begin{equation}\label{eq:decomposZ}
\vZ_n = 
{\vZ}_{n}^{(1)} +
{\vZ}_{n}^{(2)} +
{\vZ}_{n}^{(3)} .
\end{equation}

Analyzing separately these three components, we will obtain sufficient
and necessary conditions for the (complete or partial) almost sure
asymptotic synchronization of the system, in the sense of the
following definition:

\begin{definition}[Complete or partial almost sure asymptotic synchronization]
  \label{def-sincro}  
Given a system of $N\geq 2$ interacting RSPs as defined above, we say
that we have the {\em complete almost sure asymptotic synchronization}
of the system (or {\em the entire system almost surely synchronizes in
  the limit}) if, for each pair $(l_1,l_2)$ of indices in
$V=\{1,\dots,N\}$, we have
\begin{equation}\label{eq:def-synchro}
(Z_{n,l_2}-Z_{n,l_1})\stackrel{a.s.}\longrightarrow 0.
\end{equation}
We say that we have the {\em almost sure asymptotic synchronization}
of the system {\em within each cyclic class} (and so a {\em partial}
almost sure asymptotic synchronization of the system) if
\eqref{eq:def-synchro} holds true for each pair $(l_1,l_2)$ of indices
such that $l_1\sim_c l_2$, i.e~belonging to the same cyclic class. (Of
course this partial synchronization is interesting only in the case
$\nper\neq N$, that is when at least one cyclic class has more than
one element.)
\end{definition}

Since, by definition, the term ${\vZ}_{n}^{(3)}$ is the only one of
the three parts of the decomposition that differs between the
processes within the same cyclic class, we have the characterization
of the almost sure asymptotic synchronization of the system inside
each cyclic class as the almost sure convergence toward zero of
$(\vZ_n^{(3)})_n$. In the next theorem we state that the three
components of the decomposition of the process $(\vZ_n)_n$ are
linearly independent and so we get the characterization of the
complete almost sure asymptotic synchronization of the system as the
almost sure convergence toward zero of $(\vZ_n^{(2)})_n$ and of
$(\vZ_n^{(3)})_n$. The proofs of the following and the other results
of this section are collected more ahead in Section~\ref{sec:proofs}.

\begin{theorem}\label{teo:sp-project}
Let ${\vZ}_{n}^{(1)} ,
{\vZ}_{n}^{(2)} ,
{\vZ}_{n}^{(3)} 
$ as in
\eqref{eq:decomposZ}.
Then, we have\footnote{For the cases $\nper=1$ and $\nper=N$, we use
  the usual convention $\sum_{j=1}^0\dots=0$.}
\begin{equation*}
  \begin{split}
  {\vZ}_{n}^{(1)} &= \vone \vv^\top \vZ_n
  \in Span\{\vone\},
\qquad
{\vZ}_{n}^{(2)} = 
\sum_{j=1}^{\nper-1} \vq_{j}\vv_j^\top\vZ_n
\in Span\{\vq_1,\dots,\vq_{\nper-1}\},
\\
{\vZ}_{n}^{(3)} &= 
\sum_{j=1}^{N-\nper} \vr_{j} \vp_{j}^\top\vZ_n
\in Span\{\vr_1,\dots,\vr_{N-\nper}\},
  \end{split}
  \end{equation*}
 where $\vq_1,\dots,\vq_{\nper-1}$ (resp. $\vv_1,\dots,
 \vv_{\nper-1}$) are right (resp. left) eigenvectors of $W^\top$
 related to the eigenvalues with modulus equal to $1$, but different
 from $1$ (i.e. the complex $\nper$-roots of the unit different from
 $1$) and $\vr_1,\dots,\vr_{\nper-1}$
 (resp. $\vp_1,\dots,\vp_{N-\nper}$) are right (resp. left), possibly
 generalized, eigenvectors of $W^\top$ related to the eigenvalues with
 modulus strictly smaller than $1$.  \\ \indent As a consequence, the
 three components $\vZ_n^{(i)}$, $i=1,2,3$, are linear independent and
 so the complete almost sure asymptotic synchronization of the system
 holds if and only if both ${\vZ}_{n}^{(2)}$ and ${\vZ}_{n}^{(3)}$
 converge almost surely toward zero.
\end{theorem}

Before analyzing the behavior of the three terms ${\vZ}_n^{(i)}$,
$i=1,2,3$, in the general case, let us first consider the trivial
scenario in which initially all the processes $\{Z_{n,l}:\, l\in V\}$
start from the same barrier ($0$ or $1$) with probability one, that
is~$P(T_0)=1$ with $T_0 = \{\vZ_0=\vzero\} \cup \{\vZ_0=\vone\}$,
where, similarly to $\vone$, the symbol $\vzero$ denotes the vector
with all the entries equal to zero. Since $\{\vZ_0=\vzero\}$ implies
$\{\vZ_n=\vzero,\ \forall n\geq 0\}$ (and, analogously,
$\{\vZ_0=\vone\}$ implies $\{\vZ_n=\vone,\ \forall n\geq 0\}$), when
$P(T_0)=1$, we have ${\vZ}_{n}^{(2)}\equiv{\vZ}_{n}^{(3)}\equiv
\vzero$ and ${\vZ}_{n}^{(1)}\equiv{\vZ}_{n}\equiv\vZ_0$ with
probability one, and trivially the process $(\vZ_n)_n$ converges
almost surely and the entire system almost surely synchronizes.  For
this reason, {\em from now on we will consider the non-trivial initial
  condition $P(T_0)<1$}.  \\

\indent Consider now the first term, i.e. ${\vZ}_{n}^{(1)}$. Since
$(\widetilde{Z}_n)_n$ is a bounded martingale and so convergent almost
surely to a random variable, from Definition \ref{def-sincro} and
Theorem~\ref{teo:sp-project}, we get that the entire system almost
surely synchronizes in the limit if and only if there exists a random
variable $Z_\infty$, taking values in $[0,1]$, such that
$$
\vZ_n\stackrel{a.s.}\longrightarrow Z_\infty\vone.
$$ In other words, the complete almost sure asymptotic synchronization
of the system implies the almost sure convergence of all the
stochastic processes $\{(Z_{n,l})_n:\, l\in V\}$ to the same limit
random variable $Z_\infty$.  \\

\indent We now investigate the conditions that characterize the regime
of (complete or partial) almost sure asymptotic synchronization.  In
particular, in the next results we will show that, whenever $\sum_n
r_n<+\infty$, there is a strictly positive probability that
${\vZ}_{n}^{(2)}$ or ${\vZ}_{n}^{(3)}$ do not vanish
asymptotically. This fact, since Theorem~\ref{teo:sp-project}, means
that the system does not synchronize (completely nor partially) with
probability one.  An heuristic explanation of why this occurs when
$\sum_n r_n<+\infty$ is that, when at time $n_0$ the remaining series
$\sum_{n>n_0} r_n$ gets very low, the single processes
$\{(Z_{n,l})_n:\, l\in V\}$ cannot converge too far from their values
$\{Z_{n_0,l}:\, l\in V\}$, respectively, and this avoids them to
converge to the same limit almost surely when their distance at time
$n_0$ can be large with strictly positive probability.\\

\indent Let us now present the result regarding the third
component ${\vZ}_{n}^{(3)}$. 

\begin{theorem}\label{teo:thirdpart}
  When $\nper=N$, we have
  ${\vZ}_{n}^{(3)}\equiv \vzero$.  When $\nper<N$, we have
  ${\vZ}_{n}^{(3)}\stackrel{a.s.}\longrightarrow\vzero$ if and only if
  $\sum_n r_n = +\infty$.
\end{theorem}

As observed before, Theorem~\ref{teo:thirdpart} is providing necessary
and sufficient conditions for a {\em partial} almost sure asymptotic
synchronization of the system, that is for the {\em almost sure
  asymptotic synchronization within each cyclic class}.  This fact is
stated in the following result.

\begin{corollary}[Almost sure asymptotic synchronization within each cyclic class]
  \label{th-sincro-periodic}
Excluding the trivial case $\nper=N$ (i.e. all the cyclic classes have
only one element), we have almost sure asymptotic synchronization of
the system within each cyclic class if and only if $\sum_n r_n =
+\infty$.
\end{corollary}

Regarding the above result, it is important to underline that the
almost sure convergence of the single processes $\{(Z_{n,l})_n:\, l\in V\}$
is not guaranteed under the condition $\sum_n r_n=+\infty$, as we can
only affirm the almost sure convergence toward zero of the difference
between two processes belonging to the same cyclic class.

\begin{remark}\label{rem:cond-nec}
  Since Theorem \ref{teo:sp-project}, Theorem~\ref{teo:thirdpart}
  shows that, at least when $\nper<N$, the condition $\sum_n r_n =
  +\infty$ is necessary for the complete almost sure asymptotic
  synchronization.  More ahead we will see that it is actually
  necessary also in the case $\nper=N$ in order to let
  ${\vZ}_{n}^{(2)}$ asymptotically vanishes as well.
\end{remark}

Condition $\sum_n r_n=+\infty$ is not sufficient in order to guarantee
the complete almost sure synchronization of the system because, by
Theorem~\ref{teo:sp-project}, we have to control the second component
$\vZ_n^{(2)}$ too. 
Regarding this issue, 
the following example will result to be useful in order
to understand the sufficient and necessary conditions for the almost
sure convergence toward $\vzero$ of $\vZ_n^{(2)}$.
    
\begin{example}\label{example:no-sincro-r_n-infinite}
  Consider the case of $N=3$ and
  $$W^\top=\left(\begin{smallmatrix} 0 & 2/3 & 1/3 \\ 1 &0 & 0 \\ 1 &
    0 & 0\end{smallmatrix}\right),$$ whose eigenvalues are
    $\{1,0,-1\}$ and so the period is $\nper=2$. For
      simplicity, suppose $P(Z_{0,l}\in (0,1), \forall l\in V)>0$
      (otherwise, we have to work starting from a time-step $n_0$ such
      that $P(Z_{n_0,l}\in (0,1), \forall l\in V)>0$, which always exists by
      the irreducibility of $W$ and the assumption $P(T_0)<1$). Set
      $A_n=\{X_{n,1}=0\} \cap \{X_{n,2}=1,X_{n,3}=1\}$ for $n$ even
      and $A_n=\{X_{n,1}=1\} \cap \{X_{n,2}=0,X_{n,3}=0\}$ for $n$
      odd, so that we have
\begin{align*}
& Z_{2k,1} = (1-r_{2k})Z_{2k-1,1} \leq (1-r_{2k}),
\\
& Z_{2k+1,1} = (1-r_{2k+1})Z_{2k,1} + r_{2k+1} \geq r_{2k+1},
\\
& Z_{2k,2} = (1-r_{2k})Z_{2k-1,2} + r_{2k} \geq r_{2k},
\\
& Z_{2k+1,2} = (1-r_{2k+1})Z_{2k,2} + r_{2k+1} \leq (1-r_{2k+1}),
\\
& Z_{2k,3} = (1-r_{2k})Z_{2k-1,3} + r_{2k} \geq r_{2k},
\\
& Z_{2k+1,3} = (1-r_{2k+1})Z_{2k,3} + r_{2k+1} \leq (1-r_{2k+1})
.
\end{align*}
Then, we will prove that, when $\sum_n (1-r_n) < +\infty$, there is a
strictly positive probability that $\boldsymbol{Z}_n$ does not
converge. More specifically, we will show that the probability of the
event $A=\cap_{n=1}^{\infty} A_n$ is strictly positive if
$\sum_n (1-r_n) < +\infty$. To this end, let $\bar{A}_n =
\cap_{k=0}^{n} A_k$ and notice that
$$
P(A) = P(\cap_{n=1}^{\infty} A_n) = \prod_{n=1}^{\infty}
P(A_n \vert  \cap_{k=0}^{n-1} A_k) = \prod_{n=1}^{\infty} P(A_n \vert  \bar{A}_{n-1}).
$$
Then, take $n$ even and note that
$$
P(A_n \vert  \bar{A}_{n-1}) = 
P(X_{n,1}=0 \vert  \bar{A}_{n-1}) 
\cdot P(X_{n,2}=1 \vert  \bar{A}_{n-1})
\cdot P(X_{n,3}=1 \vert  \bar{A}_{n-1}).
$$ Now, consider $P(X_{n,1}=0 \vert  \bar{A}_{n-1})$ and notice that, since
$\{X_{n-1,2}=0,X_{n-1,3}=0\}\subset\bar{A}_{n-1}$, we have on
$\bar{A}_{n-1}$,
\begin{align*}
  Z_{n-1,2} &= (1-r_{n-1})Z_{n-2,2} + r_{n-1}X_{n-1,2} =
  (1-r_{n-1})Z_{n-2,2} \leq (1-r_{n-1}),
\\
Z_{n-1,3} &= (1-r_{n-1})Z_{n-2,3} + r_{n-1}X_{n-1,3} =
(1-r_{n-1})Z_{n-2,2} \leq (1-r_{n-1}),
\end{align*}
which implies 
$$
P(X_{n,1}=0 \vert  \bar{A}_{n-1}) = 1- \tfrac{2}{3} Z_{n-1,2} - \tfrac{1}{3} Z_{n-1,2}
\geq r_{n-1}.
$$ Analogously, consider $P(X_{n,2}=1 \vert  \bar{A}_{n-1})$ (the same is
for $X_{n,3}$) and notice that, since
$\{X_{n-1,1}=1\}\subset\bar{A}_{n-1}$, we have on $\bar{A}_{n-1}$, 
$$
Z_{n-1,1} = (1-r_{n-1})Z_{n-2,1} + r_{n-1}X_{n-1,1} \geq r_{n-1},
$$
which implies 
$$
P(X_{n,2}=1 \vert  \bar{A}_{n-1}) = Z_{n-1,2} \geq r_{n-1}.
$$
Then we have 
$P(A_n \vert  \bar{A}_{n-1})\geq r_{n-1}^3$, which implies 
$$
P(A) = \prod_{n=1}^{\infty} P(A_n \vert  \bar{A}_{n-1}) \geq 
\Big( 
\prod_{n=1}^{\infty} r_{n-1} 
\Big)^3
= 
\Big( 
\prod_{n=1}^{\infty}\left( 1 - (1-r_{n-1})\right)
\Big)^3 >0, 
$$ if $\sum_n (1-r_{n}) < +\infty$. Therefore, we have found a
non-negligible event $A$ on which $ \vZ_{2n} \to (0,1,1)^\top $ and
$ \vZ_{2n+1} \to (1,0,0)^\top $, that is on $A$ we have neither
asymptotic synchronization of the entire system nor almost sure
convergence of $(\vZ_n)_n$. Let us now discuss the asymptotic behavior
of the three components of the decomposition~\eqref{eq:decomposZ}. A
direct calculation shows that $\vv^\top = (1/2,1/3,1/6)$ is the Perron
eigenvector and
\[
\vZ^{(1)}_n = 
\frac{3 Z_{n,1} + 2 Z_{n,2} + Z_{n,3}}{6} 
\begin{pmatrix}
1
\\
1
\\
1
\end{pmatrix}
, \quad
\vZ^{(2)}_n = 
\vZ^{(C)}_n
- \vZ^{(1)}_n, 
\quad
\vZ^{(C)}_n = 
\begin{pmatrix}
Z_{n,1} 
\\
\tfrac{2 Z_{n,2} + Z_{n,3}}{3} 
\\
\tfrac{2 Z_{n,2} + Z_{n,3}}{3}  
\end{pmatrix}.
\]
Analogously, we could obtain $\vZ^{(1)}_n$, $\vZ^{(2)}_n$ and
$\vZ^{(3)}_n$ by using the spectral representation of
Theorem~\ref{teo:sp-project}, which in this case is the following:
\[
\vone \vv^\top=\left(\begin{smallmatrix} 1/2 & 1/3 & 1/6 \\ 1/2 & 1/3 & 1/6 \\ 
1/2 & 1/3 & 1/6\end{smallmatrix}\right),
\qquad
\vq_{1}\vv_1^\top=\left(\begin{smallmatrix} 1/2 & -1/3 & -1/6 \\ -1/2 & 1/3 & 1/6 \\ 
-1/2 & 1/3 & 1/6\end{smallmatrix}\right),
\]
\[
\vr_{1} \vp_{1}^\top=\left(\begin{smallmatrix}  0 &0 & 0 \\  0 & 1/3 & -1/3 \\
0 & -2/3 & 2/3\end{smallmatrix}\right).
\]
On $A$, the first term $\vZ^{(1)}_n $ almost surely converges to $
\vZ^{(1)}_\infty = \tfrac{1}{2}\vone$. The process $(\vZ_n^{(C)})_n$,
and so $(\vZ_n^{(2)})_n$, does not converge on $A$. Indeed, on $A$ we
have $\vZ^{(C)}_{2n}\to (0,1,1)^{\top}$ and $\vZ^{(C)}_{2n+1}\to
(1,0,0)^{\top}$.  Finally, it is easy to check that on $A$ the third
component $(\vZ_n^{(3)})_n$ almost surely converges to zero (in
agreement with Theorem~\ref{teo:thirdpart}, since $\sum_n
(1-r_n)<+\infty$ trivially implies $\sum_n r_n = +\infty$).  It is
interesting to note that, in this framework, we have $
\vZ_n=\vZ^{(1)}_n+\vZ_n^{(2)}+\vZ_n^{(3)} \stackrel{a.s.}\sim
\vZ^{(1)}_n+\vZ_n^{(2)}=\vZ_n^{(C)}= (Z_{n,1} ,\tfrac{2 Z_{n,2} +
  Z_{n,3}}{3} ,\tfrac{2 Z_{n,2} + Z_{n,3}}{3} )^\top, $ which is
constant on each cyclic class and, if we consider the
$\nper$-dimensional process $ \vZ_n^{(c)} = ( Z_{n,1} ,\tfrac{2
  Z_{n,1} + Z_{n,3}}{3} )^\top $ defined by taking a single entry of
$\vZ_n$ for each cyclic class, then $\vZ_n^{(c)}$ is still not
convergent on $A$ (since $\vZ_{2n}^{(c)}\to (0,1)^\top$ and
$\vZ_{2n+1}^{(c)}\to (1,0)^\top$), but its norm almost surely
converges on $A$ (we have $ \vert\vert \vZ_n^{(c)}\vert\vert ^2\to 1 $). This is a
general fact that we will prove in the sequel.  \qed
\end{example}

We will now formally provide the general theory inspired by the above
Example~\ref{example:no-sincro-r_n-infinite}. In particular, we notice
that the non-convergence of $\vZ_n$ in the example has been induced
in the case $\sum_n (1-r_n) < +\infty$ by the
component $\vZ_n^{(2)}$, which is present because of the period
$\nper\geq 2$ of $W^{\top}$.  This suggests that, for a periodic
matrix $W^\top$, a necessary condition in order to have the complete
almost sure asymptotic synchronization of the system is $\sum_n
(1-r_n)=+\infty$. However, this is not sufficient. Indeed, we need a
stronger condition as stated in the following result.

\begin{theorem}\label{teo:secondpart}
When $\nper=1$, we have
${\vZ}_{n}^{(2)}\equiv \vzero$.  When $\nper\geq 2$, we have
${\vZ}_{n}^{(2)}\stackrel{a.s.}\longrightarrow \vzero$ if and only if
$\sum_n r_n(1-r_n)=+\infty$.
\end{theorem}

Note that $\sum_n r_n(1-r_n)=+\infty$ implies both $\sum_n
r_n=+\infty$ and $\sum_n (1-r_n)=+\infty$, because $0<r_n<1$.
Therefore, this result, combined with Remark~\ref{rem:cond-nec}, makes
{\em $\sum_n r_n=+\infty$ a necessary condition for the complete
  almost sure asymptotic synchronization} of the system (regardless of
the period $\nper$ of $W^\top$).  \\

\indent Now that we have determined the sufficient and necessary
conditions in order that ${\vZ}_{n}^{(2)}$ and ${\vZ}_{n}^{(3)}$
asymptotically vanish with probability one, we can state the following
synthetic result on the complete almost sure asymptotic
synchronization of the system.

\begin{corollary}[Complete almost sure asymptotic synchronization]
  \label{th-sincro}
 When $\nper = 1$, we have ${\vZ}_n\stackrel{a.s.}\longrightarrow
 Z_\infty \vone $ if and only if $\sum_n r_n =+\infty$.  When $\nper
 \geq 2$, we have ${\vZ}_n\stackrel{a.s.}\longrightarrow Z_\infty
 \vone $ if and only if $\sum_n r_n(1-r_n) =+\infty$.
\end{corollary}

Summing up, with the results stated in this section, the question of
the complete almost sure asymptotic synchronization of the system has been
completely addressed (in \cite{ale-cri-ghi-barriers} we investigate 
the distribution of the common limit $Z_\infty$, pointing out when it can take
the extreme values, $0$ or $1$, with a strictly positive probability).  
 The next step of the present work is now
 to determine the asymptotic dynamics of the processes
  $\{(Z_{n,l})_n:\, l\in V\}$ when the complete almost sure asymptotic
  synchronization does not hold (see the following Section \ref{sec:non-sincro}).


\section{Regimes of non-almost sure asymptotic synchronization of the entire system}
\label{sec:non-sincro}

As in Section~\ref{sec:decomp-conv-sincro}, we consider a system of
$N\geq 2$ RSPs with a network-based interaction as defined in
\eqref{interacting-1-intro} and \eqref{interacting-2-intro} or,
equivalently, in~\eqref{eq:cond-mean} and \eqref{eq:dynamics}, where
the matrix $W$ has non-negative entries, is irreducible and is such
that $W^{\top}\boldsymbol{1}=\boldsymbol{1}$. Moreover, in order to
exclude trivial cases, we fix $P(T_0)<1$.  \\

\indent In this section, we analyze the asymptotic behavior of the
system when the entire process $(\vZ_n)_n$ does not almost surely
synchronize in the limit.  First, we focus on the scenario in which
there is almost sure convergence without almost sure synchronization
(neither complete nor partial), and then we will consider the case in
which we have a partial almost sure asymptotic synchronization (that
is almost sure asymptotic synchronization within the cyclic classes),
but there is a strictly positive probability that the entire process
$(\vZ_n)_n$ neither converges nor asymptotically synchronizes.

\subsection{Almost sure convergent regime}\label{subsec:convergent_no_synchro}

Let us first give the following simple result 
regarding the trivial case $\sum_n r_n<+\infty$.

\begin{prop}\label{th-r_n-finite-convergent}
When $\sum_n r_n <+\infty$, all the processes $({\vZ}_{n})_n$,
$({\vZ}_{n}^{(1)})_n$, $({\vZ}_{n}^{(2)})_n$ and $({\vZ}_{n}^{(3)})_n$
converge almost surely.
Moreover, for any pair of vertices $(l_1,l_2)\in V\times V$, with $l_1\neq l_2$, we have
$P(Z_{\infty, l_1}\neq Z_{\infty, l_2})>0$.
\end{prop}
This result shows that, whenever $\sum_n r_n<+\infty$, the processes
$\{(Z_{n,l})_n:\, l\in V\}$ at the network vertices converge with
probability one, but none of their limits
can be almost surely equal, i.e. for any pair of processes 
there exists a strictly positive
probability that they do not synchronize in the limit.  
\\

\indent An interesting consideration that can be derived from
combining Theorem~\ref{teo:thirdpart} with
Proposition~\ref{th-r_n-finite-convergent}, is that, exactly as for
$({\vZ}_{n}^{(1)})_n$, also the process $({\vZ}_{n}^{(3)})_n$ always
converges almost surely, without any additional assumption on the
matrix $W$ or on the reinforcement sequence $(r_n)_n$.  This will not be the case of
the periodic component $({\vZ}_{n}^{(2)})_n$, which will lead in some
situations to the non-convergence of the process $({\vZ}_{n})_n$, as
we will see in the following subsection.

\subsection{Non-almost sure convergent regime}
\label{subsec:periodic_no_synchro}

In the previous subsection we have shown that the first
(i.e. ${\vZ}_{n}^{(1)}$) and the third component
(i.e. ${\vZ}_{n}^{(3)}$) always converge almost surely, while
regarding the second component ${\vZ}_{n}^{(2)}$, combining
Theorem~\ref{teo:secondpart} and
Proposition~\ref{th-r_n-finite-convergent}, it only remains to consider
the set of conditions summarized in the following assumption (which is
exactly the framework considered in
Example~\ref{example:no-sincro-r_n-infinite})
\begin{assumption}\label{ass:periodic}
Assume that all the following conditions hold true: 
\begin{enumerate}
\item $\nper\geq 2$,
\item $\sum_n r_n=+\infty$,
\item $\sum_n r_n(1-r_n)<+\infty$.
\end{enumerate}
\end{assumption}
In this section, we focus on this framework and we will show that only two
events are possible, both with a strictly positive probability: either
all the processes both converge and asymptotically synchronize to the
same barrier, or they neither converge nor asymptotically synchronize
all together, following in this second case a periodic dynamics that
can be described by the structure of the cyclic classes identified by
the matrix $W^\top$.  The proofs of the following results are postponed in
Section \ref{sec:proofs}.  \\

\indent From Corollary \ref{th-sincro-periodic} and Corollary
\ref{th-sincro}, we know that, under Assumption~\ref{ass:periodic},
there is a {\em partial almost sure asymptotic synchronization}:
indeed, we have the almost sure asymptotic synchronization within each
cyclic class, but not an almost sure asymptotic synchronization across
different classes, i.e. not a complete almost sure asymptotic
synchronization of the system.  Indeed, by 
Theorem~\ref{teo:secondpart}, we have that ${\vZ}_{n}^{(2)}$ does not
converge to zero with probability one. However, we have not specified yet
whether it converges or not.  To clarify this point, let us identify
each cyclic class $h$ with the scalar process $Z_{n,h}^{(c)}$ of the
unique value assumed by $\vZ_n^{(C)}=({\vZ}_{n}^{(1)} +
{\vZ}_{n}^{(2)})$ along that class.  Formally, we can define
$\vZ_n^{(c)} = \vZ_n^{(C)}/\sim_c $ as the quotient of $\vZ_n^{(C)}$
with respect to $\sim_c$.  Alternately, component-wise, for any cyclic
class $h = 0, \ldots, \nper-1$,
\begin{equation}\label{eq:defZnB}
Z_{n,h}^{(c)} =
\sum_{l_1 \in \text{cyclic class }h} \frac{v_{l_1}}{
\sum_{l_2 \in \text{cyclic class }h} v_{l_2}
} Z_{n,l_1} .
\end{equation}
The asymptotic behavior of the process $(Z_{n,h}^{(c)})_n$ in general will
depend on the properties of the matrix $W$ and on the reinforcement sequence
$(r_n)_n$.  However, there is a very general result about the almost
sure convergence of the norm of $Z_{n,h}^{(c)}$, that always holds
without any additional assumption on $W$ or $(r_n)_n$.

\begin{theorem}[Convergence of the norm]\label{th-norm}
The sequence $(\vert\vert \vZ_n^{(c)}\vert\vert )_n$ of the norms of $(\vZ_n^{(c)})_n$
almost surely converges.
\end{theorem}

Notice that, although Theorem \ref{th-norm} holds regardless of the
behavior of $(r_n)_n$, the interpretation of $\vert\vert \vZ_n^{(c)}\vert\vert $, and in
general of $Z_n^{(c)}$, is meaningful only when the single processes
$Z_{n,l}$ within each cyclic class almost surely synchronize in the
limit, that is, by Corollary~\ref{th-sincro-periodic}, when $\sum_n
r_n = +\infty$. Indeed, under this condition, the component
$\vZ_n^{(3)}$ converges almost surely to zero and so we have
$Z_{n,l}\stackrel{a.s.}\sim Z_{n,l}^{(C)}$ for each $l$, that is
$Z_{n,l}$ is asymptotically closed to the element of $\vZ_n^{(c)}$
corresponding to its cyclic class.  Moreover, Theorem~\ref{th-norm}
becomes especially interesting when $\vZ_n^{(c)}$ does not converge 
with probability one, that is, by Corollary~\ref{th-sincro}, when
$\sum_n r_n (1-r_n)< +\infty$ and $\nper\geq 2$. Combining together
the above two considerations, in the next theorem we consider the
scenario under Assumption~\ref{ass:periodic}, and we show that each
process $(Z_{n,h}^{(c)})_n$, $h=0,\dots,\nper-1$, is asymptotically
close to the barrier-set $\{0, 1\}$ and the number of processes close
to a given barrier almost surely converges to a random variable not
concentrated in $\{0, \nper\}$.

\begin{theorem}[Limit set for the cyclic classes]
\label{th-limit-class-0-1}
Under Assumption~\ref{ass:periodic}, we have:
\begin{itemize}
\item[(a)] the limit set of each cyclic class is given by the
  barrier-set $\{0,1\}$ (in other words, we have 
  $Z_{n,h}^{(c)}(1-Z_{n,h}^{(c)})\stackrel{a.s.}\longrightarrow 0$ for
  each $h=0,\dots,\nper-1$);
\item[(b)] for any fixed $\epsilon\in(0,1/2)$, we have $card\{h:\,
  Z^{(c)}_{n,h}\geq 1-\epsilon\}\stackrel{a.s.}\longrightarrow
  N_\infty$ and $card\{h:\,Z_{n,h}^{(c)}\leq
  \epsilon\}\stackrel{a.s.}\longrightarrow \nper-N_\infty$, where
  $N_\infty$ is a random variable taking values in $\{0,1,\ldots,
  \nper\}$ with $P(N_\infty=0)+P(N_\infty=\nper)<1$.
\end{itemize}
\end{theorem}

\begin{remark}
We highlight that condition $\nper\geq 2$ is essential to have
$P(N_\infty=0)+P(N_\infty=\nper)<1$ in the above theorem, while it is
not used to get the other affirmations (that hold also for $\nper=1$).
However, $P(N_\infty=0)+P(N_\infty=\nper)<1$ is a crucial point, as
$\{N_\infty=0\}$ and $\{N_\infty=\nper\}$ are the two events on which
the asymptotic synchronization of the entire system still holds.
Note that, since Theorem~\ref{th-limit-class-0-1}(a), 
on these two events we have the almost sure 
{\em asymptotic polarization} of $(\vZ_n)_n$, 
that is in the limit the system almost surely
 synchronizes toward $0$ or $1$. 
\end{remark}

Theorem~\ref{th-limit-class-0-1}(a) states that, under the assumed
conditions for $(r_n)_n$, when $n$ is large, $Z_{n,h}^{(c)}$ can only
be close to $1$ or $0$.  Roughly speaking, this means that at a large time-step 
there are some cyclic classes in
which the agents' personal inclinations are very close to one (``\emph{on}'') and other
classes in which they are very close to zero (``\emph{off}''). In other words, 
at sufficiently large time-steps, the cyclic classes
are {\em polarized} on the status \emph{on} or \emph{off}. 
Notice that the number of cyclic classes   
that are \emph{on} at a large time-step $n$ can be well
represented by $\vert\vert {\vZ}_{n}^{(c)}\vert\vert^2$.
(Indeed, this
fact and Theorem \ref{th-norm} are the keys to prove the convergence
stated in point (b) of Theorem \ref{th-limit-class-0-1}.)  However,
when $N_\infty\in\{1,\dots,\nper-1\}$, it is still not known if each
class converges or if it can switch from \emph{on} to \emph{off} and
vice versa infinitely often. To this end, let us first define the
processes $({\vX}_{n}^{(C)})_n$ and $({\vX}_{n}^{(c)})_n$ analogously
as what has been done with $({\vZ}_{n}^{(C)})_n$ and
$({\vZ}_{n}^{(c)})_n$ in~\eqref{def:Z_2} and~\eqref{eq:defZnB},
respectively, i.e. for each element $l\in V=\{1,\dots,N\}$,
\begin{equation*}
{X}_{n,l}^{(C)} = 
\sum_{l_1 \sim_c l} \frac{v_{l_1}}{
\sum_{l_2 \sim_c l} v_{l_2}
} X_{n,l_1},
\end{equation*}
and $\vX_n^{(c)} = \vX_n^{(C)}/\sim_c $ as the quotient of $\vX_n^{(C)}$
with respect to $\sim_c$.
Then, we can present the following asymptotic result.

\begin{theorem}[Clockwise dynamics and asymptotic periodicity]
  \label{teo:PeriodicCycle}
 Under Assumption~\ref{ass:periodic}, there exists an integer-valued
increasing sequence $(\sigma_n)_n$ such that, for any cyclic class
$h$, we have for $(\vZ_n)_n$ a clockwise dynamics in $(\sigma_n)_n$,
that is
\[
Z_{\sigma_{n},h-1}^{(c)} - Z_{\sigma_{n-1},h}^{(c)} \stackrel{a.s.}\longrightarrow 0
\quad\forall h=0,\dots,\nper-1,
\]
and a stationary  dynamics between the clock-times, that is 
\[
\sup_{m_1,m_2 \in \{\sigma_{n-1},\ldots,\sigma_{n}-1\}} \vert\vert  \vZ_{m_1} - \vZ_{m_2}
 \vert\vert 
\stackrel{a.s.}\longrightarrow 0.
\]
Moreover, we have 
$$
P(X_{\sigma_n, l_1}=X_{\sigma_n,l_2}\ \text{eventually})=1\quad\forall l_1\sim_c l_2
$$
and
$$
P(X_{\sigma_{n},h-1}^{(c)} = X_{\sigma_{n-1},h}^{(c)}\ \text{eventually})=1
\quad\forall h=0,\dots,\nper-1.
$$
If, in addition $\sum_n (1-r_n) < +\infty$, then, eventually
$\sigma_{n+1}= \sigma_{n}+1$ so that  
\[
 \vZ_{n+\nper} - \vZ_{n} \stackrel{a.s.}\longrightarrow \vzero
\qquad\text{and}\qquad
P(\vX_{n+\nper} = \vX_{n}\ \text{eventually})=1.
\]
\end{theorem}

An important consequence of Theorem~\ref{th-limit-class-0-1} and
Theorem~\ref{teo:PeriodicCycle} is that Assumption~\ref{ass:periodic}
is sufficient for having a non-convergent asymptotic periodic behavior
of ${\vZ}_n$ with a strictly positive probability.  In addition, from
Theorem \ref{th-limit-class-0-1} we know that the limit set of each
cyclic class is $\{0,1\}$.  Then, when
$N_\infty\in\{1,\dots,\nper-1\}$, the asymptotic periodic dynamics of
the vector ${\vZ}_n^{(c)}$ presented in
Theorem~\ref{teo:PeriodicCycle} can be seen as a cycle over $\nper$
elements of $\{0,1\}^{\nper}$, obtained by rotating their elements
(see e.g. Example~\ref{example:no-sincro-r_n-infinite} where, on the
set $A$, we have $\vZ_{2n}^{(c)}\to (0,1)^{\top}$ and
$\vZ_{2n+1}^{(c)}\to (1,0)^{\top}$). 
Moreover, we note that the a.s. asymptotic synchronization of 
the random variables 
$X_{\sigma_n,l}$ inside the same cyclic class implies 
$X_{\sigma_n,l}\stackrel{a.s.}\sim X_{\sigma_n,l}^{(C)}$  
for each $l$, that is $X_{n,l}$ is asymptotically close to $X_{\sigma_n,h}^{(c)}$,  
where the index $h$ indicates the cyclic class of $l$. Therefore, for a large $n$, at 
time-step  $\sigma_n$, all the agents of the same cyclic class perform the same action 
($0$ or $1$) and these performed actions are periodic. Further, 
we point out that the sequence $(\sigma_n)_n$ is explicitely defined in terms of 
the values of the sequence $(r_n)_n$ as described at the beginning of the proof of Theorem \ref{teo:PeriodicCycle}.\\

In order to better understand the regime identified by Theorem \ref{teo:PeriodicCycle},
we have added Figure \ref{plot} and Figure \ref{plot-X} representing a setting 
in which the assumptions of Theorem \ref{teo:PeriodicCycle} hold
and the status of the agents' inclinations 
(blu = close to 0 = \emph{off}, red = close to 1 = \emph{on}) 
and the actions performed by the agents (blu = 0, red = 1)
inside each cyclic class  
are shown at different time-steps.
Specifically, we consider a network with $N=120$ processes 
connected by an interacting matrix $W$ (randomly) generated
to have period $\nper=6$.
In the two figures the agents belonging to the same cyclic class
have been positioned closed to each others,
in order to visualize the a.s. (partial) synchronization
within each cyclic class.
The sequence $(r_n)_n$ has been chosen such that
$\sum_{n\notin(4m)_m}r_n < +\infty$ and
$\sum_{n\in(4m)_m}(1-r_n)< +\infty$,
which means that $\sigma_{n+1}-\sigma_{n}=4$.
Therefore, as described in Theorem \ref{teo:PeriodicCycle},
we can see from Figure \ref{plot} that, for a large $n$, the agents' inclinations 
are synchronized inside the cyclic classes and 
present a clockwise periodic behavior along the subsequence $(\sigma_n)_n$,
while they remain in the same previous status
for time-steps in $\{\sigma_{n-1}+1,\ldots,\sigma_{n}-1\}$. 
Moreover, in Figure \ref{plot-X}, we can see that, for a large $n$, also the 
agents' actions at the clock-times $\sigma_n$ are  
synchronized inside the cyclic classes and  
present a clockwise periodic behavior along the subsequence $(\sigma_n)_n$.

\begin{figure}
\centering
\includegraphics[scale=0.20]{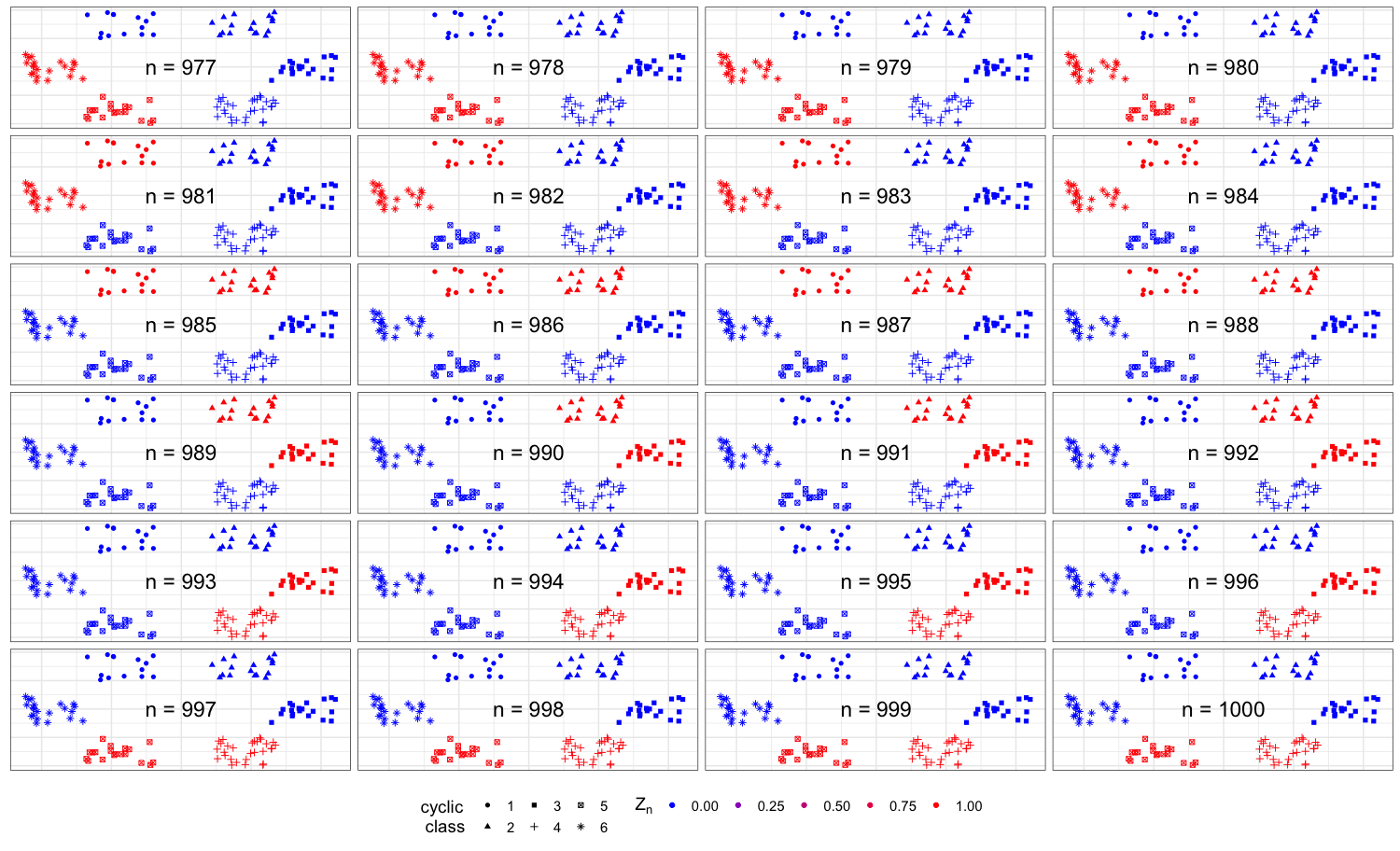}
\caption[]{Each panel represents the agents' inclinations 
at different consecutive time-steps
(large enough such that the asymptotic regime described in 
Theorem \ref{teo:PeriodicCycle} can be observed). 
Any vertex $h\in V$ is represented by a specific point in each panel, 
where its color indicates the value of $Z_{n,h}\in(0,1)$ 
and its symbol represents its cyclic class.
The sequence $(r_n)_n$ is such that $r_n=(1-cn^{-\gamma})$ when $n\in(4m)_m$ and 
$r_n=cn^{-\gamma}$ otherwise, with $c=1$ and $\gamma=3.7$.}
\label{plot}
\end{figure}

\begin{figure}
\centering
\includegraphics[scale=0.20]{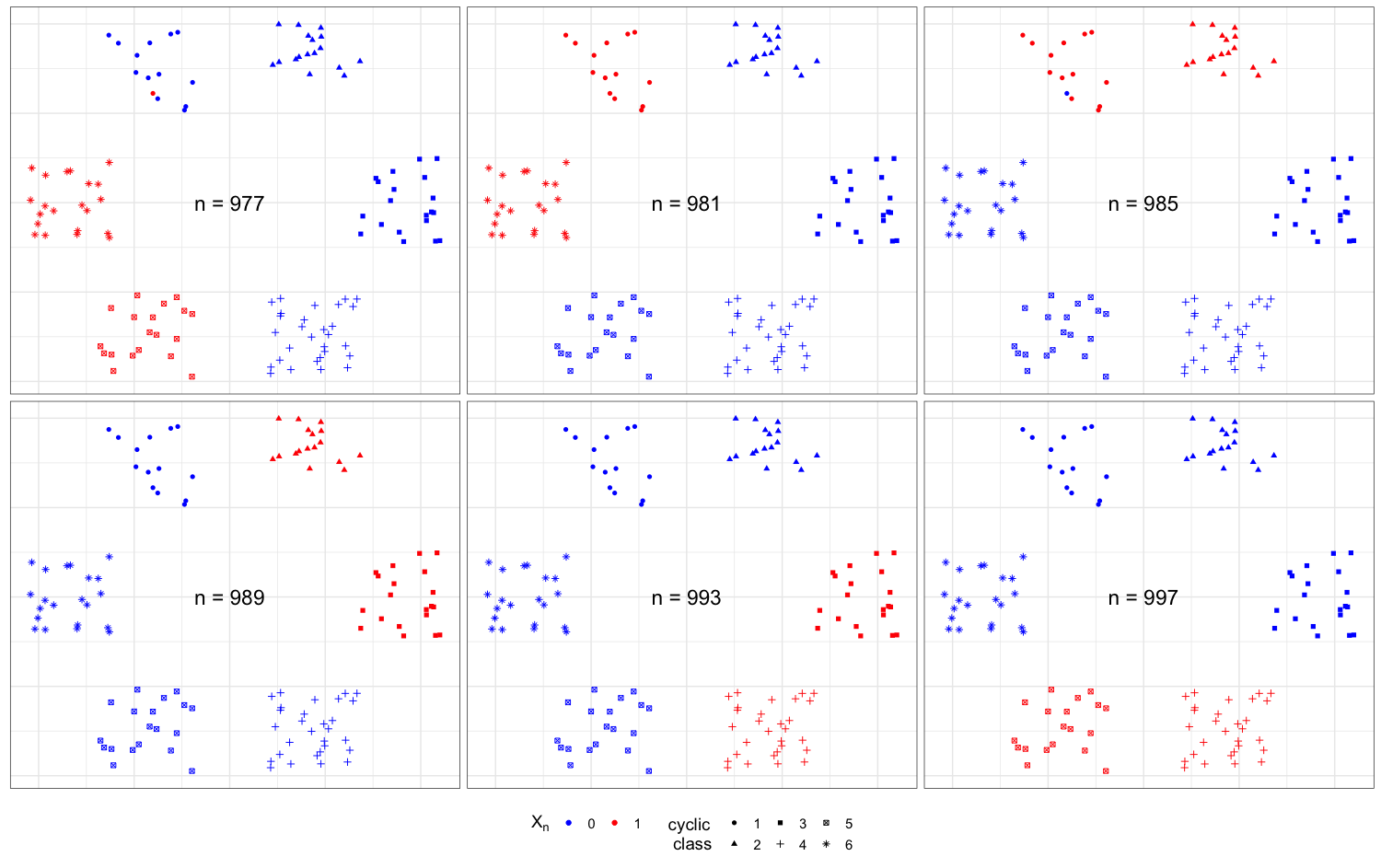}
\caption[]{Each panel represents the agents' actions 
at different clock-times $\sigma_n$ 
(with $n$ large enough such that the asymptotic regime described in 
Theorem \ref{teo:PeriodicCycle} can be observed). 
Any vertex (agent) $l\in V$ is represented by a specific point in each panel, 
where its color indicates the value of $X_{n,l}\in\{0,1\}$  (action performed by the agent) 
and its symbol represents its cyclic class.
The sequence $(r_n)_n$ is such that $r_n=(1-cn^{-\gamma})$ when $n\in(4m)_m$ and 
$r_n=cn^{-\gamma}$ otherwise, with $c=1$ and $\gamma=3.7$.}
\label{plot-X}
\end{figure}


\section{Some complements}\label{sec:additional}
We here collect some complements to the above results.

\subsection{Empirical means}
As in \cite{ale-cri-ghi-WEIGHT-MEAN}, we can consider the weighted average of
times in which the agents adopt ``action 1'', i.e. the stochastic
processes of the {\em weighted empirical means} $\{(N_{n,l})_{n}:\,
l\in V\}$ defined, for each $l\in V$, as $N_{0,l}=0$ and, for any
$n\geq 1$,
\begin{equation*}
N_{n,l}=\sum_{k=1}^{n}q_{n,k} X_{k,l}\,,\quad\hbox{
where }\quad q_{n,k}=\frac{a_k}{\sum_{\ell=1}^na_\ell},
\end{equation*}
with $(a_k)_{k\geq 1}$ a suitable sequence of strictly positive real
numbers.  In particular, when $a_k=1$ for each $k$, then the processes
$\{(N_{n,l})_n:\, l\in V\}$ simply coincide with the empirical means
associated to the processes $\{(X_{n,l})_n:\, l\in V\}$. Instead, if,
according to the principle of reinforced learning, we want to give
more ``weight'' to the current, or more recent, experience, we can
choose $(a_k)_{k\geq 1}$ increasing.  \\ \indent When the process
$(\vZ_n)_n$ almost surely converges toward a random variable, say
$\vZ_\infty$, we have
$$
E[\vX_{n+1}\vert \mathcal{F}_n]=W^\top\vZ_n \stackrel{a.s.}\longrightarrow
W^\top\vZ_\infty\,.
$$ Therefore, assuming
$q_{n,n}=\frac{q}{n^{\nu}}+O\left(\frac{1}{n^{2\nu}}\right)$, with
$q>0$ and $0<\nu\leq 1$ and applying
Lemma~B.1 in \cite{ale-cri-ghi-MEAN}, with the same computations done
in the proof of Theorem~3.1~in \cite{ale-cri-ghi-WEIGHT-MEAN}, we
obtain that
$$
\vN_{n}=[N_{n,1},\dots,N_{n,N}]^\top
\stackrel{a.s.}\longrightarrow W^{\top}\vZ_\infty\,.
$$ When the entire system almost surely synchronizes in the limit, we
have $\vZ_\infty=Z_\infty\vone$ and so, since $W^\top\vone=\vone$, we
get $W^\top\vZ_\infty=Z_\infty\vone$, that is also the weighted
empirical means of the entire system almost surely synchronizes in the
limit toward $Z_\infty$.  \\

In the case discussed in Subsection~\ref{subsec:periodic_no_synchro} under Assumption
\ref{ass:periodic}, we can easily prove
that there is the same form of partial almost sure asymptotic
synchronization for the processes of the weighted empirical means.
Indeed, if we take $l_1$ and $l_2$ in the same cyclic class $h$, we 
can apply the same arguments as above to the difference
$(N_{n,l_2}-N_{n,l_1})_n$, provided that
$E[X_{n+1,l_2}-X_{n+1,l_1}\vert \mathcal{F}_n]=
\sum_lw_{l,l_2}Z_{n,l}-\sum_lw_{l,l_1}Z_{n,l}\stackrel{a.s.}\longrightarrow
0$. In order to show this fact, we note that, for each
$l_{*}\in\{l_1,l_2\}$, the term $w_{l,l_{*}}=[W^\top]_{l_*,l}$ is not
null only if $l$ belongs to the cyclic class $(h+1)$ (recall that
$\nper\geq2$ under Assumption \ref{ass:periodic}), and hence we have
$$
\begin{aligned}
\sum_lw_{l,l_{*}}Z_{n,l} &&=&
\sum_{l \sim_c l_{*}}w_{l,l_{*}}(Z_{n,l}-Z_{n,h+1}^{(c)})+ Z_{n,h+1}^{(c)}\\
&&\leq& 
\max_{l \sim_c l_{*}}\{\vert Z_{n,l}-Z_{n,h+1}^{(c)}\vert \} + Z_{n,h+1}^{(c)}
\stackrel{a.s.}\sim Z_{n,h+1}^{(c)},
\end{aligned}$$ where the last step follows by the almost sure asymptotic
synchronization of the processes $(Z_{n,l})_{n}$ within the same
cyclic class $(h+1)$.\\

\indent Regarding the almost sure convergence of the processes
$\{(N_{n,l})_{n}:\,l\in V\}$ instead, nothing can be said in general
if we consider only the conditions in Assumption \ref{ass:periodic}.
Indeed, we may have either regimes with almost sure convergence and
regimes with non-almost sure convergence, according to the specific
sequence $(r_n)_n$ considered. The following two examples show both
these scenarios.

\begin{example}\label{empirical_means_non_convergent}
Fix $a\in(0,\frac{1}{2})$ and set the sequence $(r_n)_n$ as follows:
$r_n=(1-u_n)$ if $n\in\{[a^{-k}]:\, k\geq 1\}$ and $r_n=u_n$
otherwise, where $(u_n)_n$ is an arbitrary sequence such that $\sum_n
u_n < +\infty$.  Naturally, this implies that $\sum_n r_n=+\infty$ and
$\sum_n r_n(1-r_n)<+\infty$, so that, assuming $W$ with 
$\nper= 2$, i.e. cyclic classes $\{0,\, 1\}$, 
we are under Assumption \ref{ass:periodic}, i.e. in the framework of
Subsection~\ref{subsec:periodic_no_synchro} with non-almost sure
convergence and only almost sure asymptotic synchronization within the
cyclic classes of the processes $\{Z_{n,l}:\, l\in V\}$.  Then, we can
apply Theorem~\ref{teo:PeriodicCycle}, with $\sigma_k=[a^{-k}]$, which
ensures that there exists a set $A$ with strictly positive probability
such that, for any $l$ in one of the two cyclic classes, say $h_1$,
\[
\sup_{n \in \{\sigma_{2k-1},\ldots,\sigma_{2k}-1\}} \vert 1-Z_{n,l}\vert 
\mathop{\stackrel{a.s.}\longrightarrow}_{k\to+\infty} 0, \qquad\mbox{and}\qquad
\sup_{n \in \{\sigma_{2k},\ldots,\sigma_{2k+1}-1\}} \vert Z_{n,l}\vert 
\mathop{\stackrel{a.s.}\longrightarrow}_{k\to+\infty} 0
\]
and, for any $l$ in the other cyclic class, say $h_2$,
\[
\sup_{n \in \{\sigma_{2k-1},\ldots,\sigma_{2k}-1\}} \vert Z_{n,l}\vert 
\mathop{\stackrel{a.s.}\longrightarrow}_{k\to+\infty} 0, \qquad\mbox{and}\qquad
\sup_{n \in \{\sigma_{2k},\ldots,\sigma_{2k+1}-1\}} \vert 1-Z_{n,l}\vert 
\mathop{\stackrel{a.s.}\longrightarrow}_{k\to+\infty} 0.
\]

Then, consider the processes of the simple empirical means, i.e.
$\{N_{m,l}=\frac{1}{m}\sum_{n=1}^m X_{n,l}:\, l\in V\}$; 
for any $l$ in the cyclic class $h_2$,
we can write
the following two relations along the sequences of $(\sigma_{2k})_k$
and $(\sigma_{2k+1})_k$:
\[\begin{aligned}
\frac{1}{\sigma_{2k}}\sum_{n=1}^{\sigma_{2k}}\mathbb{E}[X_{n,l}&\vert \mathcal{F}_{n-1}] =
\frac{1}{\sigma_{2k}}\sum_{n=1}^{\sigma_{2k}}
[W\vZ_{n-1}]_l\\
&\geq \left(\frac{\sigma_{2k}-\sigma_{2k-1}}{\sigma_{2k}}\right)
\inf_{n \in \{\sigma_{2k-1},\ldots,\sigma_{2k}-1\},\ l\in h_1}Z_{n,l}
\mathop{\stackrel{a.s.}\longrightarrow}_{k\to+\infty} (1-a),
\end{aligned}\]
\[\begin{aligned}
\frac{1}{\sigma_{2k+1}}\sum_{n=1}^{\sigma_{2k+1}}\mathbb{E}[X_{n,l}&\vert \mathcal{F}_{n-1}] =
\frac{1}{\sigma_{2k+1}}\sum_{n=1}^{\sigma_{2k+1}}
[W\vZ_{n-1}]_l\\ 
&\leq
\left(\frac{\sigma_{2k+1}-\sigma_{2k}}{\sigma_{2k+1}}\right)
\sup_{n \in \{\sigma_{2k},\ldots,\sigma_{2k+1}-1\},\ l\in h_1}Z_{n,l} + 
\left(\frac{\sigma_{2k}}{\sigma_{2k+1}}\right)
\mathop{\stackrel{a.s.}\longrightarrow}_{k\to+\infty} a,
\end{aligned}\] 
which imply that, on the set $A$, 
for any $l$ in the cyclic class $h_2$, $N_{\sigma_{2k},l}\longrightarrow
(1-a)$ and $N_{\sigma_{2k+1},l}\longrightarrow a$.  Therefore, in this
example the processes $\{N_{n,l}:\, l\in V\}$ do not converge almost
surely.
\end{example}

\begin{example}
Consider the framework of Example \ref{empirical_means_non_convergent}
with a different reinforcement sequence $(r_n)_n$; indeed, here we set
$r_n=(1-u_n)$ for any $n\geq 1$, where we recall that $\sum_n u_n<
+\infty$, and hence $\sum_n (1-r_n)< +\infty$.  Then, if we apply
Theorem~\ref{teo:PeriodicCycle} we have $\sigma_k=k$ for any $k\geq 1$,
recalling that on the set $A$, for any $l$ in the cyclic class $h_2$,
we have
$$
[W\vZ_{2n}]_{l}\leq \max_{i\in h_1}\{Z_{2n-1,i}\}\stackrel{a.s.}\longrightarrow 0,
\qquad\mbox{and}\qquad
[W\vZ_{2n-1}]_{l}\geq \min_{i\in h_1}\{Z_{2n-2,i}\} \stackrel{a.s.}\longrightarrow 1,
$$
and so
 \[
\begin{aligned}
&\frac{1}{\sigma_{k}}\sum_{n=1}^{\sigma_{k}}\mathbb{E}[X_{n,l}\vert \mathcal{F}_{n-1}] =
\frac{1}{k}\sum_{n=1}^{k}[W\vZ_{n-1}]_l\\
&=\
\frac{1}{k}\sum_{n=1}^{[k/2]}([W\vZ_{2n-1}]_{l}+[W\vZ_{2n}]_{l}) +
\mathbbm{1}_{\{k\ odd\}}\frac{[W\vZ_{k-1}]_l}{k}\\
&=\
\frac{1}{k/2}\sum_{n=1}^{[k/2]}
\left(\frac{[W\vZ_{2n-1}]_{l}+[W\vZ_{2n}]_{l}}{2}\right) + o(1)\\
&\stackrel{a.s.}\sim\
\frac{1}{k/2}\sum_{n=1}^{[k/2]}\left(\frac{1}{2}\right)\ \stackrel{a.s.}\sim\
\frac{1}{2}.
\end{aligned}
\]
Therefore, in this example, the processes $\{N_{n,l}:\, l\in V\}$ converge almost surely.
\end{example}


\subsection{Reducible interaction matrix}
We here explain how to deal with the case of a {\em reducible}
interaction matrix $W$, as done in \cite{ale-cri-ghi} (and in
\cite{ale-ghi} for a similar approach to systems of interacting
generalized Friedman's urns).  \\ \indent We recall that in Network
Theory a strongly connected component of a directed graph is a maximal
sub-graph in which each node may be reached by the others in that
sub-graph.  The set of strongly connected components forms a partition
of the set of the graph nodes.  The assumption of irreducibility in
this context means that the entire directed graph is strongly
connected.  When this assumption is not verified, we can consider each
connected component (related to the graph seen as undirected) and, for
each of these components, we can analyze its condensation graph. We
remind that the condensation graph of a connected, but not strongly
connected, directed graph $\mathcal G$ is the directed acyclic graph
$C_{\mathcal G} $, where each vertex is a strongly connected component
of $\mathcal G$ and an edge in $C_{\mathcal G}$ is present when there
are edges in $\mathcal G$ between nodes of the two strongly connected
components.  In our context, the strongly connected components
corresponding to leaves vertices in $C_{\mathcal G}$ are made by
network nodes that are not influenced by the nodes of the other
strongly connected components. \\ \indent Moreover, there is an
equivalent interpretation in the Markov chain theory, where the
concept of strongly connected component is replaced by that of
communicating class.  Hence, the strongly connected components
corresponding to leaves vertices in $C_{\mathcal G}$ are the recurrent
communicating classes.  \\ \indent We employ a particular
decomposition of $W^\top$ that individuates its recurrent
communicating classes. The same decomposition has been applied to the
interaction matrix in~\cite{ale-ghi,ale-cri-ghi}.  More precisely,
denoting by $m$, with $1\leq m\leq N$, the multiplicity of the
eigenvalue $1$ of $W^\top$, the reducible matrix $W^\top$ can be
decomposed (up to a permutation of the nodes) as follows:
\begin{equation}\label{def:Wtop}
W^\top =
\begin{pmatrix}
U_1&0&\ldots&0& 0\\
0&U_2&0 &\ldots &0&\\
\ldots &\ldots &\ldots &\ldots &\ldots \\
0&0&\ldots &U_m&0\\
U_{1,f}&U_{2,f}&\ldots &U_{m,f}&U_{f}
\end{pmatrix},
\end{equation}
where
\begin{itemize}
\item[(i)] $\{U_s;\,1\leq s\leq m\}$ are irreducible $N_s\times
  N_s$-matrices with leading eigenvalue equal to $1$, that identify
  the recurrent communicating classes;
\item[(ii)] (if $\sum^{m}_{s=1} N_s\leq N-1$) $U_{f}$ is a
  $N_{f}\times N_{f}$-matrix, that contains all the transient
  communicating classes;
\item[(iii)] (if $\sum^{m}_{s=1} N_s\leq N-1$)
$\{U_{s,f};\,1\leq s\leq m\}$ are $N_s\times N_{f}$-matrices.
\end{itemize}
Obviously, when $\sum^{m}_{s=1} N_s=N$ we have $N_f=0$ and hence the
elements in $\{U_{s,f};\,1\leq s\leq m\}$ and $U_{f}$ do not exist.
This occurs when all the classes are closed and recurrent (leaves
vertices without parents) and hence the state space can be partitioned
into irreducible and disjoint sub-spaces.  In the particular case of a
$W^\top$ irreducible, considered previously in this paper, there is
only one closed and recurrent class and hence $m=1$, $N_1=N$ and
$N_f=0$.  \\ \indent Summing up, the structure of $W^\top$ given
in~\eqref{def:Wtop} leads to a natural decomposition of the graph in
different sub-graphs $\{G_s;\,1\leq s\leq m\}$ associated to the
sub-matrices $\{U_s;\,1\leq s\leq m\}$ and $G_f$ associated to $U_f$.
In addition, from~\eqref{interacting-1-intro} and~\eqref{def:Wtop}, we
can deduce that, for each $1\leq s\leq m$, the nodes in $G_s$ are not
influenced by the nodes in the rest of the network, and hence the
dynamics of the processes in $G_s$ can be fully established by
considering only the correspondent irreducible sub-matrix $U_s$.
Hence, applying the results presented in this paper to each sub-graph
$G_s$, it is possible to characterize the first-order dynamics of the
nodes in it.  \\ \indent Concerning the sub-graph $G_{f}$, we first
note that this is composed by the union of all the transient classes.
These classes are not independent of the behavior of the rest of the
network. More precisely, the condensation graph shows the conditional
dependence in the reverse order: starting from the leaves
$\{G_s;\,1\leq s\leq m\}$, whose dynamics may be computed
independently from the rest, the behavior of another strongly
connected component is always independent of its parent vertices in
the condensation graph, given the dynamics of its children vertices,
while the children vertices have an effect on it. For
  instance, if we are in the standard scenario when
  $\sum_nr_n=+\infty$ and $\sum_n r_n^2<+\infty$, then all the agents’
  inclinations in the same $G_s$ almost surely synchronize toward a
  random limit $Z_{\infty,s}$ and each of the agents' inclinations in
  $G_f$ converges almost surely toward a suitable convex combination
  of the limits $\{Z_{\infty,s};\, 1 \leq s \leq m\}$ of the processes
  related to the vertices in $\{G_s;\, 1 \leq s \leq m\}$
  (see~\cite{ale-cri-ghi}). Hence, they do not necessarily
  synchronize.

  \subsection{Non-stochastic interaction matrix
    (the case of a non-homogeneous ``forcing input'' toward zero)} 
We here assume the non-negative interaction matrix to be such that
$W^\top \vone=\vd\ (\neq \vone)$, with $d_l\leq 1$ $\forall l\in V$,
as $d_l>1$ could lead to $[W^\top \vZ_{n}]_l>1$, which is not
compatible with the original definition of the model in
\eqref{interacting-1-intro}.  However, the considerations below may
also be applied to models that only satisfy \eqref{eq:cond-mean} and
not \eqref{interacting-1-intro}, allowing the random variables
$X_{n,l}$ and $Z_{n,l}$ to take values bigger than $1$, when $[W^\top
  \vone]_l >1$.  Moreover, we assume to be in the interesting case
when $\sum_n r_n=+\infty$ and, to avoid unessential complications, we
take $W^\top$ irreducible.\\ \indent Under these assumptions, by the
Frobenius-Perron theory, the (real) leading eigenvalue $\lambda^*$ of
$W^\top$ is strictly less than $1$, and we have the associated
eigenvectors $\vu$ and $\vv$, with strictly positive entries and 
such that $\vv^\top\vone=1$ and $\vu^\top\vv=1$, that now
satisfy the relations $W^\top \vu=\lambda^*\vu$ and $\vv^\top
W^\top=\lambda^*\vv^\top $, respectively.  Then, setting as before
$\widetilde{Z_n}=\vv^{\top}\vZ_n$, in this case we get
$E[\widetilde{Z}_{n+1}\vert \mathcal{F}_n] =
(1-r_n(1-\lambda^*))\widetilde{Z}_n$ and so the process
$(\widetilde{Z}_n)_n$ is not anymore a martingale. Now, the pivotal
marginale of the system is given by the process
$(\widetilde{Z}^*_n)_n$ defined as
$$ \widetilde{Z}^*_n =
\frac{\widetilde{Z}_n}{ \prod_{k=0}^{n-1}(1-r_k(1-\lambda^*))} =
\frac{\vv^\top \vZ_n}{\prod_{k=0}^{n-1}(1-r_k(1-\lambda^*))}.
$$ This is a non-negative martingale and so it almost surely converges
toward a random variable $\widetilde{Z}_\infty^*$. This, together with
the fact that $\prod_{k=1}^{n-1}(1-r_k(1-\lambda^*))\to 0$ (because
$\sum_nr_n=+\infty$), implies that $(\widetilde{Z}_n)_n$ converges
almost surely to zero. Since $\vv$ has strictly positive entries, this
is possible only if $(\vZ_n)_n$ converges almost surely to $\vzero$.
Therefore, the synchronization result $\vZ_n \stackrel{a.s.}\sim
\widetilde{Z}_n\vone \stackrel{a.s.}\to Z_\infty\vone$ obtained before
under suitable assumptions might be replaced by the relation $\vZ_n
\stackrel{a.s.}\sim \widetilde{Z}_n\vu \stackrel{a.s.}\sim
\prod_{k=0}^{n-1}(1-r_k(1-\lambda^*)) \widetilde{Z}^*_\infty
\vu\stackrel{a.s.}\to \vzero$, with possibly $\vu \neq \vone$.
\\ \indent From an applicative point of view, the above behavior can
be explained by the presence of a forcing input that pushes agents to
action $0$. When $d_l = 0$, we simply have $P(X_{n+1,l} =
1\vert \mathcal{F}_n) = 0$ for each $n$ (and so $Z_{n,l}\stackrel{a.s.}\to
0$). Hence, we can assume $d_l>0$ for each $l$ and we can write
$$
P(X_{n+1,l}= 1\vert \mathcal{F}_n) = [W^\top \vZ_n]_l=
d_l [W_*^\top\vZ_n]_l + (1-d_l)q\qquad\mbox{with } q = 0,
$$ where, for each $l$, the column $l$ of $W_*$ is obtained taking the
column $l$ of $W$ divided by $d_l$ so that $W_*^\top\vone=\vone$, 
and $d_l<1$ for at least one $l$.  Note
that, when $\vd = d\vone$ with $d\in [0,1[$, we essentially have the
    model with forcing input $q$ equal to zero described in
    \cite{ale-cri-ghi, cri-dai-lou-min}, where the input acts
    homogeneously on all the agents of the network. Otherwise, when
    the vector $\vd$ has at least two different entries, the forcing
    input is not homogeneous, in the sense that it does not affect the
    agents in the same way.
    

\section{Proofs}\label{sec:proofs}
In this section, we describe the crucial ideas and 
the sketches of the proofs of the results presented in this work. 
All the details are collected in a separate supplementary material \cite{ale-cri-ghi-supplSPA1}.

First, let us introduce the following notation the we will adopt in this section: 
\begin{itemize}
\item[(1)] $n\in\mathbb{N}$ is the time-step index;
\item[(2)] $l,l_1,l_2,\ldots \in \{1,\ldots,N\}$ are spatial indexes
  related to the vertices of the network;
\item[(3)] the indices $j,j_1,j_2,\ldots \in \{0,\ldots,\nper-1\}$ are
  related to the $\nper$ complex roots of the unit, and they are
  always defined modules $\nper$;
\item[(4)] the indices $h,h_1,h_2,\ldots \in \{0,\ldots,\nper-1\}$ are
  related to the cyclic classes, and they are always defined modules $\nper$.
\end{itemize}
Moreover, given a complex matrix $A$, we denote by $Sp(A)$ the set of its
eigenvalues and by $A^*$ the conjugate transpose of $A$.
\\

\indent Before discussing the structures of the proofs, 
we would like to point out that
the following proofs can be considerably simplified in the
``standard'' case $\sum_n r_n = +\infty$ and $\sum_n r_n^2 < +\infty$.
However, our aim is not only to obtain sufficient conditions in order
to obtain the complete almost sure asymptotic synchronization of the
system, but to get sufficient and necessary conditions for complete or
partial almost sure asymptotic synchronization.  Therefore, we also
include in our analysis the case $\sum_n r_n<+\infty$ and the case
$\sum_n r_n^2=+\infty$, which is a completely different scenario
compared to the one considered
in~\cite{ale-cri-ghi,ale-cri-ghi-MEAN,ale-cri-ghi-WEIGHT-MEAN,cri-dai-min,dai-lou-min}. The
case $\sum_n r_n^2=+\infty$ has been considered in
\cite{cri-dai-lou-min}, but only for a very specific choice of $W$,
i.e.~the mean-field interaction (that is also included in the present
study as a special case), and with regard to the problem of asymptotic
polarization of the process $(\sum_{l=1}^NZ_{n,l}/N)_n$.  \\

\subsection{Spectral representation of the components of
  the decomposition of $(Z_n)_n$ (proof of Theorem~\ref{teo:sp-project})}
\label{sec:charSpectral}

\indent The key-point in order to prove the asymptotic results of Section
\ref{sec:decomp-conv-sincro} and Section \ref{sec:non-sincro} is the spectral
representation of the three components of the decomposition of
$\vZ_n$, that is stated in Theorem~\ref{teo:sp-project}.
To this end, let us denote with $P^{(2)}$ and $P^{(3)}$ (resp. ${P^{\dagger}}^{(2)}$ and ${P^{\dagger}}^{(3)}$) the matrix whose row (resp. column) vectors are the left (resp. right) eigenvectors of $W^\top$ associated to the eigenvalues with $|\lambda|=1$ (except $\lambda=1$) and $|\lambda|<1$, respectively.
Then, the proof of Theorem~\ref{teo:sp-project}, which is detailed in the supplementary material, can be here summarized in the following points:
\begin{itemize}
\item[(i)] By construction we imediately have that ${\vZ}_{n}^{(1)} = \vone \vv^\top \vZ_n= \widetilde{Z}_n \vone$.
\item[(ii)] Regarding ${\vZ}_{n}^{(2)}$, first we derive the analytic expressions of the $\nper$ eigenvalues and (both left and right) eigenvectors associated with $Sp(W^\top)=\{\lambda:|\lambda|=1\}$, i.e. for $j = 0, \ldots, {\nper}-1$,
\begin{itemize}
\item[(i.a)] $\lambda_{1,j}= \exp(\tfrac{2\pi i}{n_\text{per}} j)$,
\item[(i.b)] $[P^{(2)}]_{j,l_2}=[{\vv_{j}^\top}]_{l_2} = v_{l_2} \lambda_{1,j}^{-h} = e^{-i\frac{2\pi}{{\nper}}jh}  v_{l_2}$,
\item[(i.c)] $[{P^{\dagger}}^{(2)}]_{l_2,j}=[\vq_{j}]_{l_2} = \lambda_{1,j}^{h} = e^{i\frac{2\pi}{{\nper}}jh}$.
\end{itemize}
\item[(iii)] Secondly, we use the expressions derived at point (ii) to show that the element $(l_1,l_2)$ of the matrix $(\vone \vv^\top +{P^{\dagger}}^{(2)}P^{(2)})$ is equal to $v_{l_2} \nper$ if $l_1$ and $l_2$ belong to the same cyclic class, and zero otherwise. This will imply that
\[
\big[ (\vone \vv^\top 
+
{P^{\dagger}}^{(2)}P^{(2)})\vZ_n \Big]_{l}
= 
\nper \sum_{l_2 \sim_c l} v_{l_2}
Z_{n,l_2} .
\]
\item[(iv)] Then, we focus on the leading eigenvector $\vv$ and we show that the quantity $\sum_{l \in \text{cyclic class }h} v_{l}$ does not depend on the choice of the cyclic class $h$ and, in particular, $\sum_{l \in \text{cyclic class }h} v_{l}=\nper^{-1}$. This implies by definition \eqref{def:Z_2} that 
$$Z_{n,l}^{(1)}+Z_{n,l}^{(2)} = \nper \sum_{l_1 \sim_c l} v_{l_1}Z_{l_1},$$
which combined with point (iii) concludes the part on $\vZ_{n}^{(2)}$.
\item[(v)] Finally, the part on ${\vZ}_{n}^{(3)}$ simply follows by the properties of a basis of eigenvectors:
$$
 {\vZ}_{n}^{(3)} =
\vZ_n - ({\vZ}_{n}^{(1)}+{\vZ}_{n}^{(2)}) =
(I - (\vone \vv^\top 
+
{P^{\dagger}}^{(2)}P^{(2)} ) )
\vZ_n 
 =
{P^{\dagger}}^{(3)}P^{(3)} \vZ_n. 
$$
\end{itemize}

\subsection{Proofs of the asymptotic results stated in
  Section~\ref{sec:decomp-conv-sincro}} 

In this section, we employ the
matrices $P^{(2)}$, ${P^{\dagger}}^{(2)}$, $P^{(3)}$ and
${P^{\dagger}}^{(3)}$ and
the representation of the three components of the decomposition of
${\vZ_n}$ given in Theorem~\ref{teo:sp-project}. Moreover, we firstly provide 
an equation that will be used in the sequel.  By definition of $\vZ_n$
in~\eqref{eq:dynamics} we immediately obtain
\begin{equation}\label{eq:dynVect}
\vZ_{n+1} =
\vZ_{n} -r_n  (I - W^\top) \vZ_n + r_n \Delta \vM_{n+1},
\end{equation}
where $\Delta \vM_{n+1} = (\vX_{n+1}-W^\top \vZ_n)$ is the bounded
increment of a martingale.  Note that the conditional independence of
the entries of $\vX_{n+1}$ implies that
\[
E[\Delta {M}_{n+1,l_1} \Delta {M}_{n+1,l_2} \vert  \mathcal{F}_n ] = 
\begin{cases}
(1 - [W^\top\vZ_n]_{l_1}) [W^\top {\vZ}_{n}]_{l_2}
& \text{if }l_1=l_2,
\\
0 & \text{otherwise},
\end{cases}
\]
and hence, for any complex matrix $A \in \mathbb{C}^{k \times N}$,  
\begin{equation}\label{eq:boundM2}
\begin{aligned}
0 &\leq E[ \vert\vert A \Delta \vM_{n+1} \vert\vert ^2 \vert  \mathcal{F}_n] = 
E[\Delta \boldsymbol{M}_{n+1}^\top (A^*A) \Delta \boldsymbol{M}_{n+1}
  \vert  \mathcal{F}_n ] \\
&= (\boldsymbol{1}^\top - \boldsymbol{Z}_{n}^\top W) diag(A^*A)
W^\top \boldsymbol{Z}_{n} \\
&\leq \max_{l=1,\ldots,N} \{[A^*A]_{l,l}\}
(\boldsymbol{1}^\top - \boldsymbol{Z}_{n}^\top W) 
W^\top \boldsymbol{Z}_{n} 
= \max_{l=1,\ldots,N} \{[A^*A]_{l,l}\} V_n,
\end{aligned}
\end{equation}
where $[A^*A]_{l,l}\geq 0$ (by definition) and $V_n=(\boldsymbol{1} -
W^\top\boldsymbol{Z}_{n})^{\top} W^\top \boldsymbol{Z}_{n}$. To this
process $(V_n)_n$ is devoted the next lemma.

\begin{lemma}\label{lemma-tec}
Set
$V_n=
(\boldsymbol{1} - W^\top\boldsymbol{Z}_{n})^{\top} W^\top \boldsymbol{Z}_{n}
=\sum_{j=1}^N
[W^\top\boldsymbol{Z}_{n}]_j (1 - [W^\top\boldsymbol{Z}_{n}]_j)$.
Then, we have
\begin{equation}\label{eq:comp1}
  \textstyle{\sum_n} r^2_n V_n < +\infty\ \mbox{a.s.} 
  \qquad\mbox{and}\qquad
  \textstyle{\sum_n} r^2_n E[V_n] < +\infty.
\end{equation}
\end{lemma}

\begin{proof} Consider the dynamics \eqref{Z-tilde-dynamics} of the
  process $(\widetilde{Z}_n)_n$, i.e.
$$ \widetilde{Z}_0=\boldsymbol{v}^{\top}\boldsymbol{Z}_0,\qquad
  \widetilde{Z}_{n+1}-\widetilde{Z}_n=r_n(Y_{n+1}-\widetilde{Z}_n),\qquad
  Y_{n+1}=\boldsymbol{v}^\top\boldsymbol{X}_{n+1},
  $$ and denote by $\langle \widetilde{Z}\rangle=(\langle
  \widetilde{Z}\rangle_n)_n$ the predictable compensator of the
  submartingale $\widetilde{Z}^2=(\widetilde{Z}_n^2)_n$.  Since
  $\widetilde{Z}$ is a bounded martingale, we have that $ \langle
  \widetilde{Z}\rangle_n$ converges a.s. and its limit $\langle
  \widetilde{Z}\rangle_\infty$ is such that
  $E[\langle\widetilde{Z}\rangle_\infty]<+\infty$.  Then, setting
  $\mathcal{F}_n=
  \sigma(\boldsymbol{Z}_0,\boldsymbol{X}_1,\dots,\boldsymbol{X}_n)$,
  since
\begin{equation*}
     \langle \widetilde{Z}\rangle_\infty =
 \sum_n ( \langle \widetilde{Z}\rangle_{n+1}-\langle \widetilde{Z}\rangle_{n})
 = \sum_n E[( \widetilde{Z}_{n+1}-\widetilde{Z}_{n})^2\vert \mathcal{F}_n],
= \sum_n  r^2_nE[(Y_{n+1}-\widetilde{Z}_n)^2\vert \mathcal{F}_n],
  \end{equation*}
 the result \eqref{eq:comp1} follows if we show that there
 exists a constant $c>0$ such that $c
 E[(Y_{n+1}-\widetilde{Z}_n)^2\vert \mathcal{F}_n]\geq V_n$.  To this end,
 we observe that 
  \begin{align*}
E[(Y_{n+1}-&\widetilde{Z}_n)^2\vert \mathcal{F}_n]=
E[(\boldsymbol{v}^\top\boldsymbol{X}_{n+1}-\widetilde{Z}_n)^2
  \vert \mathcal{F}_n]\notag =
E[(\boldsymbol{v}^\top\boldsymbol{X}_{n+1}-\boldsymbol{v}^{\top}W^\top
  \boldsymbol{Z}_n)^2\vert \mathcal{F}_n]\notag\\ &=
E[(\langle\boldsymbol{v},\boldsymbol{X}_{n+1}-W^\top
  \boldsymbol{Z}_n\rangle)^2\vert \mathcal{F}_n]\notag =
\sum_{j=1}^Nv_j^2E[(X_{n+1,j}-[W^\top\boldsymbol{Z}_{n}]_j)^2\vert \mathcal{F}_n]
\notag\\ &=
\sum_{j=1}^N v_j^2
    [W^\top\boldsymbol{Z}_{n}]_j(1-[W^\top\boldsymbol{Z}_{n}]_j) \notag
    \geq v_{min}^2 \sum_{j=1}^N [W^\top\boldsymbol{Z}_{n}]_j (1 -
         [W^\top\boldsymbol{Z}_{n}]_j)\\ 
         &= v_{min}^2 V_n.
  \end{align*}
  where we recall that the eigenvector $\boldsymbol{v}$ of $W$
  associated to the eigenvalue $\lambda=1$ has strictly positive
  entries and so we have $v_{min}=\min_j\{v_j\}>0$. 
  This concludes the proof.
\end{proof}

Now we can focus on the main asymptotic results on ${\vZ}_{n}^{(2)}$
and ${\vZ}_{n}^{(3)}$.  In particular, we start by proving when
${\vZ}_{n}^{(3)}$ asymptotically vanishes with probability one.

\begin{proof}[Proof of Theorem~\ref{teo:thirdpart}]
Note that when $N = \nper$, ${\vZ}_{n}^{(3)}\equiv \vzero$ by
definition (see Theorem~\ref{teo:sp-project}). Then, we assume $N >
\nper$.

\noindent ($\Leftarrow$) We will prove that $\sum_n r_n = +\infty$
implies $P^{(3)} \vZ_n \to \vzero$.
To do so, let us recall that  
$$
P^{(3)} W^\top = J_{N-\nper}P^{(3)},
$$
where $J_{N-\nper}$ contains the Jordan blocks associated to the eigenvalues 
$\{\lambda\in Sp(W): |\lambda|<1\}$.
Therefore, it is sufficient
to prove that $P_\lambda \vZ_n \to \vzero$, where $P_\lambda$ is the
generalized eigenspace associated to an eigenvalue $\lambda$ of
$W^\top$ with $\vert \lambda\vert  < 1$, i.e. $P_{\lambda}W^\top= J_{W,\lambda} P_{\lambda}$, and let us introduce as $\vert\vert  A\vert\vert _{p,q}$ the $(p,q)$-operator norm of a
complex matrix $A\in \mathbb{C}^{M\times N}$, whose properties are
described in the supplementary material \cite{ale-cri-ghi-supplSPA1}.  We will not prove that
$P_{\lambda} \vZ_n$ tends to $\vzero$ directly, since the presence of
the ones on the upper diagonal of $J_{W,\lambda}$ causes that the term
$\vert\vert J_{W,\lambda} \vert\vert _{2,2}$ is not close to $\vert \lambda\vert $, and in particular
it is bigger than $1$.  For this reason, we need to modify the Jordan
space to face this issue. Techinical details on this task have been relegated in the supplementary materials.
Then, from now on, we can consider to have 
\begin{equation}\label{eq:modJW}  
  \vert\vert J_{W,\lambda}\vert\vert _{2,2}
  \leq \frac{1 + \vert \lambda\vert }{2} < 1.
\end{equation}

\indent \emph{Almost sure convergence of $P_{\lambda}
  \vZ_{n}$ to zero:}\\ Now we apply the Jordan base
$P_{\lambda}$ to $\vZ_{n}$, and we show that $\vB_{n} =
P_{\lambda} \vZ_{n} $ tends to $\vzero$.  By
\eqref{eq:dynVect} we immediately obtain that
\begin{equation*}
  \begin{split}
\vB_{n+1} &=
P_{\lambda} \vZ_{n+1} =
\vB_{n} -r_n (I -  J_{W,\lambda}) \vB_n +
r_n P_{\lambda} \Delta \vM_{n+1} 
\\
&= 
(( 1-r_n) I + r_n J_{W,\lambda})  \vB_n +
r_n P_{\lambda} \Delta \vM_{n+1} .
\end{split}
\end{equation*}
Set $\vert\vert \vB_{n+1}\vert\vert ^2 = {\vB}_{n+1}^*\vB_{n+1}$ so that,  
by \eqref{eq:dynVect}, \eqref{eq:boundM2} and \eqref{eq:modJW}, we have
\begin{equation*}
\begin{aligned}
E [ \vert\vert \vB_{n+1}\vert\vert ^2 \vert  \mathcal{F}_n] &= 
\vert\vert (( 1-r_n) I + r_n J_{W,\lambda})  \vB_n \vert\vert ^2\\
&\qquad + r_n^2 
(\boldsymbol{1}^\top - \boldsymbol{Z}_{n}^\top W)
diag(P_{\lambda}^* P_{\lambda})
W^\top \boldsymbol{Z}_{n}
\\
& 
\leq
\Big( (1 - r_n)  + r_n \vert\vert  J_{W,\lambda} \vert\vert _{2,2} \Big)^2 \vert\vert \vB_n\vert\vert ^2 \\
&\qquad +
\Big( \max_{h=1,\ldots,N}
    \{[P_{\lambda}^* P_{\lambda} ]_{h,h}\} \Big) 
r_n^2
(\boldsymbol{1}^\top - \boldsymbol{Z}_{n}^\top W) 
W^\top \boldsymbol{Z}_{n}
\\
& 
\leq
\Big(1 - r_n (1 - \tfrac{1 + \vert \lambda\vert }{2}) \Big)^2 \vert\vert \vB_n\vert\vert ^2+
\hat{q} r_n^2 V_n.
\end{aligned}
\end{equation*}
As a consequence of Lemma~\ref{lemma-tec}, $\vert\vert \vB_{n+1}\vert\vert ^2$ is a 
non-negative almost supermartingale that converges almost
surely (see \cite{rob}). We are now going to prove that its almost
sure limit is zero. To this purpose, since $(\vert\vert \vB_{n}\vert\vert ^2)_n$ is
uniformly bounded by a constant and so $E[a.s.-\lim_n\vert\vert \vB_{n}\vert\vert ^2]=\lim_n
E[\vert\vert \vB_{n}\vert\vert ^2]$, it is enough to prove that
$E[\vert\vert \vB_{n}\vert\vert ^2]$ converges to zero.  To show this last fact, we
take the expected values on both sides of the above relation and we
obtain
\[
E [ \vert\vert \vB_{n+1}\vert\vert ^2 ] \leq (1-a r_n )^2  E[\vert\vert \vB_{n}\vert\vert ^2] +
\hat{q} r_n^2 E[V_n],
\]
where $a = \tfrac{1 - \vert \lambda\vert }{2}>0$. Hence, $y_n = E [
  \vert\vert \vB_{n}\vert\vert ^2 ] $ converges to $0$ by 
Lemma~\ref{lemma-tecP}.
This means that $\vert\vert \vB_n\vert\vert ^2_2$ converges almost surely to zero and so
$\vB_{n}$ converges almost surely to $\vzero$, which implies that also
$P_{\lambda} \vZ_n$ converges almost
surely to $\vzero$ and the proof of the $\Leftarrow$ implication is
completed.
\\

\noindent  ($\Rightarrow$) This part of the proof is reported in the supplementaty material. The main steps are the following:
\begin{itemize}
\item[(i)] first we prove that, when $P(T_0)<1$, condition $\sum_n r_n
  < +\infty$ implies $P(Z_{n,l_2}-Z_{n,l_1}\not\to 0)>0$
  (i.e.~$P(\limsup_n \vert Z_{n,l_2}-Z_{n,l_1}\vert >0)>0$),
  for any pair of vertices $(l_1,l_2)\in V\times V$, with $l_1\neq
  l_2$;
\item[(ii)] then we observe that, if $N>\nper$ and we choose $l_1$ and
  $l_2$ within the same cyclic class, the result of Step (i) implies
  $P({\vZ}_{n}^{(3)}\not\to\vzero)>0$.
\end{itemize}
\end{proof}

The almost sure convergence to zero of ${\vZ}_{n}^{(3)}$ has as a
direct consequence the almost sure asymptotic synchronization within
the cyclic classes as stated in Corollary~\ref{th-sincro-periodic}.
The brief proof of this result is reported in the supplementary material.\\

We now are going to show when ${\vZ}_{n}^{(2)}$ asymptotically
vanishes with probability one.

\begin{proof}[Proof of Theorem~\ref{teo:secondpart}]
By the representation given in Theorem~\ref{teo:sp-project} we have to prove
$$
\vZ^{(2)} = {P^{\dagger}}^{(2)} P^{(2)} \vZ_n \stackrel{a.s.}\longrightarrow \vzero
$$
Then, by the representation the left eigenvectors of $W^\top$ in $P^{(2)}$, it is sufficient to show that, fixed any $j_1\in 
\{1, \ldots, \nper-1\}$,
\[
\eta_n = {\vv_{j_1}^\top} \vZ_n
\stackrel{a.s.}\longrightarrow 0,
\]
if and only if $\sum_n r_n(1-r_n)=+\infty$.
Since $\vv^\top_{j_1} W^\top = \lambda_{1,j_1} \vv^\top_{j_1}$, 
we immediately obtain by \eqref{eq:dynVect} 
\[
\begin{aligned}
\eta_{n+1} &&=&\ \vv_{j_1}^\top 
\vZ_{n+1} =
\eta_{n} -r_n (1 -  \lambda_{1,j_1} )) \eta_n + r_n 
\vv^\top_{j_1} \Delta \vM_{n+1} \\
&&=&\ (1-r_n (1 -  \lambda_{1,j_1} ) )\eta_n + r_n 
\vv^\top_{j_1} \Delta \vM_{n+1} ,
\end{aligned}
\]
and hence
\begin{equation}\label{eq:modulus}
E[\vert \eta_{n+1}\vert ^2 \vert  \mathcal{F}_n]  = a_{j_1,n}\vert \eta_{n}\vert ^2 + C_{j_1,n},
\end{equation}
where
\begin{equation*}
\begin{cases}
a_{j_1,n} = \vert 1-r_n (1 -  \lambda_{1,j_1} ) \vert ^2,\\
C_{j_1,n} = r^2_n E\big[ \vert 
\vv_{j_1}^\top
\Delta \boldsymbol{M}_{n+1}\vert ^2 \big\vert  \mathcal{F}_n \big]
= r^2_n E[\Delta \boldsymbol{M}_{n+1}^\top 
\bar{\vv}_{j_1}
\vv_{j_1}^\top
\Delta \boldsymbol{M}_{n+1} \vert  \mathcal{F}_n ].
\end{cases}
\end{equation*}
Then, we can prove that $a_{j_1,n}=1-s_{j_1,n}$, with $0<s_{j_1,n}\leq 1$, and by \eqref{eq:boundM2} we also have that $0\leq C_{j_1,n}\leq r_n^2 V_n$ (see supplementary material for these technical details). Hence, combining these results in \eqref{eq:modulus} we obtain,
\begin{equation}\label{eq:r(1-r)<infty}
E[\vert \eta_{n+1}\vert ^2] \geq (1-s_{j_1,n})E[\vert \eta_{n}\vert ^2],
\end{equation}
and
\begin{equation}\label{eq:r(1-r)=infty}
E[\vert \eta_{n+1}\vert ^2 \vert  \mathcal{F}_n] \leq
(1-s_{j_1,n})\vert \eta_{n}\vert ^2 +
C_{j_1,n}.
\end{equation}
We are now ready to conclude.\\

\noindent ($\Leftarrow$) \emph{Case $\sum_n r_n(1-r_n) = +\infty$}\\ 
  Note that, since $0<s_{j_1,n}\leq 1$ and $\sum_n C_{j_1,n}
< +\infty$ almost surely by Lemma~\ref{lemma-tec}, we have from
\eqref{eq:r(1-r)=infty} that $\vert \eta_{n}\vert ^2$ is a non-negative almost
supermartingale that converges almost surely (see \cite{rob}).
Since $(\vert \eta_{n}\vert ^2)_n$ is uniformly bounded by a constant and so
$E[a.s.-\lim_n\vert \eta_{n}\vert ^2]=\lim_n E[\vert \eta_{n}\vert ^2]$, it is enough to prove
that $E[\vert \eta_{n}\vert ^2]$ converges to zero. To prove this fact, we take the
expected values on both sides of~\eqref{eq:r(1-r)=infty}, so that we
obtain
$$
E[\vert \eta_{n+1}\vert ^2] \leq
(1-s_{j_1,n})\vert \eta_{n}\vert ^2 +
E[C_{j_1,n}],
$$ which is of the form $ y_{n+1} \leq (1-s_n)y_n + \delta_n$.  Then,
the convergence $y_n \to 0$ follows from 
Lemma~\ref{lemma-tecP} 
once
we show that the assumptions are verified.  We have already checked
that $0<s_n\leq 1$ and that $\sum_n s_n = +\infty$, while $\sum_n
\delta_n < +\infty$ follows by Lemma~\ref{lemma-tec}.  \\

\noindent ($\Rightarrow$) \emph{Case $\sum_n r_n(1-r_n) <
  +\infty$}\\ Note that~\eqref{eq:r(1-r)<infty} is of the form $
y_{n+1} \geq (1-s_n)y_n$.  Since $0<s_{n}\leq 1$ and $\sum_n s_{n} <
+\infty$, by 
Lemma~\ref{lemma-tec-appendix-1} 
we have $\liminf_n y_{n}
> 0$ whenever $y_0>0$. Then, we do not have the $L^2$-convergence of
$\eta_n$ to $0$ (which, since $(\eta_n)_n$ is uniformly bounded by a constant, is a necessary
condition for the almost sure convergence of $\eta_n$ to $0$) when
$y_0=E[\vert \eta_0\vert ^2]>0$, that is $P( \vert \eta_0\vert >0)>0$. This last condition is
verified when $P(T_0)<1$, because of the expression of the components of $\vv_{j_1}$ given in 
Lemma~\ref{lem:eigvD_per} 
and the fact that $v_l>0$ for
each $l$.
\end{proof}

The almost sure convergence to zero of both ${\vZ}_{n}^{(2)}$ (see
Theorem~\ref{teo:secondpart}) and ${\vZ}_{n}^{(3)}$ (see
Theorem~\ref{teo:thirdpart}) implies the complete almost sure asymptotic
synchronization of the system, as stated in Corollary~\ref{th-sincro}.
The short proof of this result is reported in the supplementary material.

\subsection{Proofs of the results given in Section~\ref{sec:non-sincro}}

The simple proof of Proposition~\ref{th-r_n-finite-convergent} regarding the 
case $\sum_n r_n <+\infty$ can be found in the supplementary material.
\\

\indent Let us now pointing out in the following lemma the connection of the
modulus of the process $(\vZ_n^{(c)})_n$ defined in~\eqref{eq:defZnB}
and the non-orthogonal projection of $(\vZ_n)_n$ on the periodic
eigenspaces of $W^\top$, i.e. $Span(\vone, \vq_1, \dots,
\vq_{\nper-1})$. Also the proof of this result is reported in the supplementary material.

\begin{lemma}\label{lem:modulusSpectral}
Let $\vC_n = (\begin{smallmatrix}
\vv^\top 
\\
{P^{(2)}} 
\end{smallmatrix}
) \vZ_n$ be the coefficients of the non-orthogonal projection of
$\vZ_n$ on the left eigenvectors of $W$ associated to the eigenvalues
in $D_{\nper}$, i.e. $\{\lambda\in Sp(W):\vert \lambda\vert =1\}$.  Then, $
\vZ_n^{(c)} =\nper^{\frac{3}{2}} O_\nper \vC_n $, where $O_\nper$ is
the orthogonal matrix with $(O_\nper)_{j_1j_2}= \frac{1}{\sqrt{\nper}}
\lambda_{1,j_1-1}^{j_2-1}$, that is
\[
O_\nper = 
\frac{1}{\sqrt{\nper}}
\begin{pmatrix}
\lambda_{1,0}^{0}
& 
\lambda_{1,0}^{1}
&
\cdots
&
\lambda_{1,0}^{\nper-1}
\\
\lambda_{1,1}^{0}
& 
\lambda_{1,1}^{1}
&
\cdots
&
\lambda_{1,1}^{\nper-1}
\\
\vdots
&
\vdots
&
&
\vdots
\\
\lambda_{1,\nper-1}^{0}
& 
\lambda_{1,\nper-1}^{1}
&
\cdots
&
\lambda_{1,\nper-1}^{\nper-1}
\end{pmatrix}.
\]
Therefore, we have that $\vert\vert \vZ_n^{(c)}\vert\vert  =\nper^{\frac{3}{2}}
\vert\vert \vC_n\vert\vert $, where $\vert\vert \vC_n\vert\vert ^2={\vC}_{n+1}^*\vC_{n+1}$.
\end{lemma}

We can now present the proof of the almost sure convergence of
$\vert\vert \vZ_n^{(c)}\vert\vert $.

\begin{proof}[Proof of Theorem~\ref{th-norm}]
By Lemma~\ref{lem:modulusSpectral}, we have that $\vert\vert \vZ_n^{(c)}\vert\vert ^2
=\nper^{3} \vert\vert \vC_n\vert\vert ^2$, with $\vert\vert \vC_n\vert\vert ^2= {\vC}_{n+1}^*\vC_{n+1}$.
Then, it is sufficient to prove that $\vert\vert \vC_n\vert\vert ^2$ almost surely
converges, using the fact that
\[
\begin{pmatrix}
\vv^\top 
\\
{P^{(2)}} 
\end{pmatrix} W^\top
= D_{\nper}
\begin{pmatrix}
\vv^\top 
\\
{P^{(2)}} 
\end{pmatrix} .
\]
Indeed, by \eqref{eq:dynVect}, we immediately obtain that
\begin{equation*}
\vC_{n+1} =
\begin{pmatrix}
\vv^\top 
\\
{P^{(2)}} 
\end{pmatrix}
\vZ_{n+1} =
\vC_{n} -r_n (I -  D_{\nper}) \vC_n + r_n \begin{pmatrix}
\vv^\top 
\\
{P^{(2)}} 
\end{pmatrix} \Delta \vM_{n+1} .
\end{equation*}
Since $\vert\vert \vC_{n+1}\vert\vert ^2 = {\vC}_{n+1}^*\vC_{n+1} $, $D_{{\nper}}^*
D_{{\nper}} = I$ and defining $ C_{{\nper}}=D_{{\nper}} + D_{{\nper}}^* = 2
diag(\{\cos(\frac{2\pi}{{\nper}}h), h= 0, \ldots, {\nper}-1\}) $, by
\eqref{eq:boundM2}, we obtain
\begin{align*}
&E [ \vert\vert \vC_{n+1}\vert\vert ^2 \vert  \mathcal{F}_n] 
 = 
\vC_{n}^* (( 1-r_n) I + r_n D_{{\nper}}^*) 
(( 1-r_n) I + r_n D_{{\nper}}) \vC_n 
\\
& \qquad  \qquad + r_n^2 
(\boldsymbol{1}^\top - \boldsymbol{Z}_{n}^\top W) diag((\begin{smallmatrix}
\vv^\top 
\\
{P^{(2)}} 
\end{smallmatrix}
)^* (\begin{smallmatrix}
\vv^\top 
\\
{P^{(2)}} 
\end{smallmatrix}
) )
W^\top \boldsymbol{Z}_{n}
\\
& 
= 
\vC_{n}^* 
\Big( I -2r_n (1 - r_n)  (I - C_{{\nper}}) \Big) \vC_n+
r_n^2
(\boldsymbol{1}^\top - \boldsymbol{Z}_{n}^\top W) diag((\begin{smallmatrix}
\vv^\top 
\\
{P^{(2)}} 
\end{smallmatrix}
)^* (\begin{smallmatrix}
\vv^\top 
\\
{P^{(2)}} 
\end{smallmatrix}
) )
W^\top \boldsymbol{Z}_{n}
\\
& 
\leq 
\vC_{n}^* 
\Big( I -2r_n (1 - r_n)  (I - C_{{\nper}}) \Big) \vC_n+
\Big( \max_{h=1,\ldots,N} \{((\begin{smallmatrix}
\vv^\top 
\\
{P^{(2)}} 
\end{smallmatrix}
)^* (\begin{smallmatrix}
\vv^\top 
\\
{P^{(2)}} 
\end{smallmatrix}
) )_{hh}\} \Big) 
r_n^2 
V_n.
\end{align*}
The diagonal matrix $(I - C_{{\nper}}) $ has real non-negative entries
between $0$ and $2$ and we have $r_n (1 - r_n) \leq \frac{1}{4}$, thus
\[
0 \leq
\vC_{n}^* 
\Big( I -2r_n (1 - r_n)  (I - C_{{\nper}}) \Big) \vC_n
\leq \vert\vert \vC_{n}\vert\vert ^2\,,
\]
which, together with Lemma~\ref{lemma-tec}, shows that $\vert\vert \vC_{n}\vert\vert ^2 $ 
is a non-negative almost
supermartingales that converges almost surely (see \cite{rob}).
\end{proof}

The second important result to describe the periodic behavior of the
system is Theorem~\ref{th-limit-class-0-1}, which is focused on the
limit set of the processes within the cyclic classes.  Before
presenting the proof of Theorem~\ref{th-limit-class-0-1}, we need to
show some technical results (proven in the supplementary material).

\begin{lemma}
\label{lemma:VnTo0}
If $V_n=(\boldsymbol{1} - W^\top\boldsymbol{Z}_{n})^{\top} W^\top
\boldsymbol{Z}_{n}\stackrel{a.s.}\longrightarrow 0$, then, for any
$l\in\{1,\ldots,N\}$, the limit set of the sequences $(Z_{n,l})_n$ is
$\{0,1\}$, that is $Z_{n,l}(1-Z_{n,l})\stackrel{a.s.}\longrightarrow
0$.
\end{lemma}

\begin{lemma}\label{lem:subseqRn}
  Assume $\sum_nr_n = +\infty$ and $\sum_n r_n(1-r_n)< +\infty$. For
  any $n$, set $\delta_n = \mathbbm{1}_{\{r_n > {1}/{2}\}}$ and $\tau_n =
  \inf\{\tau\geq1 \colon \sum_{i=1}^\tau \delta_i \geq n\}$ ($\tau_0 = 0$).
  Then $\sum_n \delta_n = +\infty$, or, equivalently, $\tau_n <
  +\infty$ for any $n$.  In addition, there exists a sequence
  $(\epsilon_n)_n$ such that $\sum_n \epsilon_n < +\infty$ and
  $\sup_{m \in (\tau_n,\tau_{n+1}]} \{\vert Z_{m,l_1} - Z_{\tau_{n+1},l_1}\vert \}
    < \epsilon_n $ for any $l_1\in V=\{1,\ldots,N\}$.
\end{lemma}

\begin{proof}[Proof of Theorem~\ref{th-limit-class-0-1}]
Let $(\tau_n)_n$ as in Lemma~\ref{lem:subseqRn}.  Now, since $
r_{\tau_n} > \tfrac{1}{2}$, Lemma~\ref{lemma-tec} implies that
$V_{r_{\tau_n}} \stackrel{a.s.}\longrightarrow 0$, and hence Lemma
\ref{lemma:VnTo0} entails that, for any $l_1\in V=\{1,\ldots,N\}$,
\[
\min \big\{ Z_{\tau_n,l_1} , 1- Z_{\tau_n,l_1} \big\}
\stackrel{a.s.}{\longrightarrow} 0.
\]
By Lemma~\ref{lem:subseqRn}, we have $\sup_{m \in
  (\tau_n,\tau_{n+1}]} \{\vert Z_{m,l_1} - Z_{\tau_{n+1},l_1}\vert \} < \epsilon_n
  \to 0$ and so $ \min \{ Z_{n,l_1} , 1- Z_{n,l_1} \} \to 0 $ almost surely.
  This implies that $ V_{n} \to 0 $ a.s., which means that the
  limit set of the sequences $([W^\top \vZ_{n}]_{l_2})_n$ is
  $\{0,1\}$, for any $l_2\in V$.  Lemma~\ref{lemma:VnTo0}
  implies that also the limit set of the sequences $(Z_{n,l_1})_n$
  is $\{0,1\}$, for any $l\in V$. Then, by
  Theorem~\ref{teo:thirdpart}, for any $l_1$ in the cyclic class $h$,
  we have that the limit set of
\[
Z_{n,h}^{(c)} = {Z}_{n,l_1} - {Z}_{n,l_1}^{(3)}
\] 
is $\{0,1\}$, and hence, for any $0<\epsilon < 1$,
\[
\vert\vert \vZ_{n}^{(c)}\vert\vert ^2 - \#\{ h \colon Z_{n,h}^{(c)} >1 - \epsilon \}
\stackrel{a.s.}\longrightarrow 0.
\]
Then, the second part of the proof is a consequence of
Theorem~\ref{th-norm}, once we have defined $N_{\infty} = a.s.-\lim_n
\vert\vert \vZ_{n}^{(c)}\vert\vert ^2 $.  In particular, the last statement simply
follows by the fact that $P(N_\infty=0)+P(N_\infty=\nper)=1$ would
imply the almost sure synchronization of the entire system, which is
in contradiction with Corollary~\ref{th-sincro} when $\nper\geq 2$,
as here we are assuming $\sum_n r_n(1-r_n)<+\infty$.
\end{proof}

Finally, we present the basic ideas and the structure of the 
proof concerning the asymptotic
periodic dynamics of the system. The details are reported in the supplementary material.

\begin{proof}[Proof of Theorem~\ref{teo:PeriodicCycle}]  
\noindent \emph{Notation and structure of the proof}\\ Let
$(\tau_n)_n$ as in Lemma~\ref{lem:subseqRn} and define, for any $n
\geq 0$, $\sigma_n = \tau_n+1$ and the set
$$
A_{n} 
  = \bigcap_{l=1,\ldots,N}\Big\{ X_{\sigma_n,l} =
  \mathbbm{1}_{[\tfrac{1}{2},1]} \big([W^\top \vZ_{\tau_n}]_l \big) \Big\}.
$$ Moreover, for each cyclic class $h\in\{0,\dots, \nper-1\}$ and
  binary index $g\in\{0,1\}$, let us introduce the following sets:
\begin{align*}
&
B_{X,n}^h(g) = \bigcap_{l\in \text{ cyclic class }h} \{X_{\sigma_n,l} = g\},
\\
&
B_{Z,n}^h(g) =
\bigcap_{l\in \text{ cyclic class }h}
\left\{\mathbbm{1}_{[\tfrac{1}{2},1]} \big(Z_{\tau_n, l} \big) = g\}\right\},
\\
&
B_{WZ,n}^h(g) =
\bigcap_{l\in \text{ cyclic class }h}
\left\{\mathbbm{1}_{[\tfrac{1}{2},1]} \big([W^\top \vZ_{\tau_n}]_l \big) = g\}\right\}.
\\
\end{align*}
Notice that the thesis $P(X_{\sigma_n, l_1}=X_{\sigma_n,l_2}\ \text{ev.})=1$ for all $l_1\sim_c l_2$ and  
$P(X_{\sigma_{n},h-1}^{(c)} = X_{\sigma_{n-1},h}^{(c)},\ \text{ev.})=1$ can be equivalently written
as
\begin{equation}\label{eq:thesis_B_X}
  P\Big(\bigcap_{h=0}^{\nper-1} \{B_{X,n-1}^h(0)\cap B_{X,n}^{h-1}(0)\}\cup
  \{B_{X,n-1}^h(1)\cap B_{X,n}^{h-1}(1)\},\ \text{ev.}\Big)=1.
\end{equation}
In addition, fix an integer $n_0$ and, for any $n\geq n_0$, let us define
\begin{align*}
&B_{n} = \bigcap_{h=0}^{\nper-1}B_{X,n}^h(0)\cup
B_{X,n}^h(1)=\bigcap_{(l_1,l_2)\colon l_1\sim_c l_2} \{X_{\sigma_n,l_1} = X_{\sigma_n,l_2}\}.
\\
& 
A_{n_0,n} 
= 
\bigcap_{m= n_0}^{n-1} A_{m},\qquad\text{and}\qquad
B_{n_0,n} 
= 
\bigcap_{m= n_0}^{n-1} B_{m}.
\end{align*}
The proof is structured as follows:
\begin{itemize}
\item[(1)] we prove that $\sup_{m_1,m_2 \in
  \{\sigma_{n-1},\ldots,\sigma_{n}-1\}} \{\vert\vert  \vZ_{m_1} - \vZ_{m_2}
  \vert\vert \}\stackrel{a.s.}\longrightarrow 0$, i.e. the dynamics between the
  clock times $(\sigma_n)_n$ is stationary;
\item[(2)] we show that $(\vX_{\sigma_n} -
  \vZ_{\sigma_n})\stackrel{a.s.}\longrightarrow 0$, and hence proving
  \eqref{eq:thesis_B_X} is enough to have also that
  $(Z_{\sigma_{n},h-1}^{(c)} - Z_{\sigma_{n-1},h}^{(c)})
  \stackrel{a.s.}\longrightarrow 0$;
\item[(3)] we prove that $A_n \cap B_{X,n-1}^h(g) \subseteq
  B_{X,n}^{h-1}(g)$, which implies that~\eqref{eq:thesis_B_X} follows
  once we show that $P(A_n \cap B_{n},\ \text{ev.})\to 1$; the result
  $A_n \cap B_{X,n-1}^h(g) \subseteq B_{X,n}^{h-1}(g)$ is 
  established by proving the following:
\begin{itemize}
\item[(3a)] $B_{X,n-1}^h(g) \subseteq B_{Z,n}^h(g)$;
\item[(3b)] $B_{Z,n}^h(g) \subseteq B_{WZ,n}^{h-1}(g)$;
\item[(3c)] $(B_{WZ,n}^{h-1}(g) \cap A_n) \subseteq B_{X,n}^{h-1}(g)$;
\end{itemize}
\item[(4)] we prove that $P(A_n \cap B_{X,n},\ \text{ev.})=1$: since
  this fact is equivalent to showing that the probability 
\begin{align*}
P \Big(\bigcap_{m = n_0}^{\infty} ( A_{m} \cap B_{m} )\Big) 
& = 
P( A_{n_0} \cap B_{n_0} ) 
\prod_{m=n_0+1}^{\infty} 
P( A_{m} \cap B_{m} \vert  A_{n_0,m}  \cap B_{n_0,m} ),
\end{align*}
tends to one as $n_0\to+\infty$, we show 
\begin{itemize}
\item[(4a)] $P( A_{n_0} \cap B_{n_0} )\to 1$;
\item[(4b)] $\prod_{m=n_0+1}^{\infty} 
P( A_{m} \cap B_{m} \vert  A_{n_0,m}  \cap B_{n_0,m}) \to 1$;
\end{itemize}
\item[(5)] finally, we show that, since $\sum_n(1-r_n)<+\infty$ implies $r_n\to
  1$, in this case it holds that $\sigma_{n+1}=\sigma_n+1$.
\end{itemize}
\end{proof}

\noindent{\bf Declaration}
\\
All the authors contributed equally to the present work. 
\\

\noindent{\bf Funding Sources} 
\\
Irene Crimaldi is partially supported by the Italian “Programma di Attività Integrata” (PAI),
project “TOol for Fighting FakEs” (TOFFE) funded by IMT School for Advanced Studies Lucca.

\appendix

\makeatletter
\let\oldtitle\@title
\let\@title\@empty
\makeatother

\markright{SUPPLEMENTARY MATERIAL}
\newpage
\setcounter{section}{0}
\setcounter{page}{1}
\setcounter{figure}{0}
\setcounter{equation}{0}
\renewcommand{\thesection}{S\arabic{section}}
\renewcommand{\thepage}{s\arabic{page}}
\renewcommand{\thetable}{S\arabic{table}}
\renewcommand{\thefigure}{S\arabic{figure}}
\renewcommand{\theequation}{S:\arabic{figure}}




\section*{\large\textbf{\uppercase{Supplementary material for Networks of reinforced stochastic processes: a complete  description of the first-order asymptotics}}}
\thispagestyle{empty}

\section{Details for the proofs}\label{app-proofs}
In this section, we will adopt the following notation in accordance with Section~\ref{sec:proofs}:
\begin{itemize}
\item[(1)] $n\in\mathbb{N}$ is the time-step index;
\item[(2)] $l,l_1,l_2,\ldots \in \{1,\ldots,N\}$ are spatial indexes
  related to the vertices of the network;
\item[(3)] the indices $j,j_1,j_2,\ldots \in \{0,\ldots,\nper-1\}$ are
  related to the $\nper$ complex roots of the unit, and they are
  always defined modules $\nper$;
\item[(4)] the indices $h,h_1,h_2,\ldots \in \{0,\ldots,\nper-1\}$ are
  related to the cyclic classes, and they are always defined modules $\nper$.
\end{itemize}
Moreover, given a complex matrix $A$, we denote by $Sp(A)$ the set of its
eigenvalues and by $A^*$ the conjugate transpose of $A$, and we recall
that $(A+B)^*=A^*+B^*$, $(zA)^*=\bar{z}A^*$, $(AB)^*= B^*A^*$,
$(A^*)^*=A$.\\

\subsection{Spectral representation of the components of
  the decomposition of $(Z_n)_n$ (detailed proof of Theorem~\ref{teo:sp-project})}
\label{app-charSpectral}

In this section, we present the proof of Theorem~\ref{teo:sp-project},
which characterizes the decomposition \eqref{eq:decomposZ} in terms of
the non-orthogonal projections on the spectral eigenspaces related to
$W^\top$.  Let us consider the Jordan representation of $W^{\top}$ in
terms of its left and right generalized eigenvectors given by the
Perron-Frobenious~Theorem for irreducible non-negative matrices:
\[
P W^\top = J_{W} P,
\qquad
W^\top P^{-1} = P^{-1}  J_{W} ,
\]
with
\[
J_{W} = 
\begin{pmatrix}
D_{{\nper}} & 0
\\
0 & J_{N- {\nper}}
\end{pmatrix},
\]
where $D_{{\nper}}$ is the diagonal matrix with the complex
${\nper}$-roots of the unity $\lambda_{1,j}= \exp(\tfrac{2\pi
  i}{n_\text{per}} j)$, $j = 0, \ldots, {\nper}-1$:
\[
\resizebox{\textwidth}{!}{$
D_{{\nper}} =
\begin{pmatrix}
e^{i\frac{2\pi}{{\nper}}0} = 1 & 0 & \ldots & 0 & 0  
\\
0 & e^{i\frac{2\pi}{{\nper}}1} = \lambda_{1,1} & \ldots & 0 & 0  
\\
\vdots & \vdots & \vdots & \vdots & \vdots 
\\
0 & 0 & \ldots & e^{i\frac{2\pi}{{\nper}}(-2)} = \lambda_{1,{\nper}-2} & 0  
\\
0 & 0 & \ldots & 0 & e^{i\frac{2\pi}{{\nper}}(-1)} = \lambda_{1,{\nper}-1}   
\end{pmatrix}$}
\]
and $J_{N- {\nper}}$ contains the Jordan blocks associated to the
eigenvalues $\{\lambda\in Sp(W)\colon\vert \lambda\vert <1\}$.  We already know that we
may choose the first line of $P$ as $\vv^\top$ and the first column of
$P^{-1}$ as $\vone$.  In the next lemma, we give a characterization of
the left and right eigenvectors related to the complex roots of the
unity.  To be more explicit, let us define cia
\[
P = 
\begin{pmatrix}
\vv^\top
\\
\vv_1^\top
\\
\vdots
\\
\vv_{\nper-1}^\top
\\
\vp_{1}^\top
\\
\vdots
\\
\vp_{N-\nper}^\top
\end{pmatrix}
=
\begin{pmatrix}
\vv^\top
\\
{P^{(2)}}
\\
{P^{(3)}}
\end{pmatrix},
\]
where
\[
{P^{(2)}} = 
\begin{pmatrix}
\vv_1^\top
\\
\vdots
\\
\vv_{\nper-1}^\top
\end{pmatrix}
\qquad\text{and}\qquad
{P^{(3)}} = 
\begin{pmatrix}
\vp_{1}^\top
\\
\vdots
\\
\vp_{N-\nper}^\top
\end{pmatrix},
\]
and analogously
\[
P^{-1} =
 \begin{pmatrix}
\vone
&
\vq_1
&
\cdots
&
\vq_{\nper-1}
&
\vr_{1}
&
\cdots
&
\vr_{N-\nper}
\end{pmatrix}
=
\begin{pmatrix}
\vone
&
{P^{\dagger}}^{(2)}
&
{P^{\dagger}}^{(3)}
\end{pmatrix},
\]
with
\[
{P^{\dagger}}^{(2)} =
 \begin{pmatrix}
\vq_1
&
\cdots
&
\vq_{\nper-1}
\end{pmatrix}
\qquad\text{and}\qquad
{P^{\dagger}}^{(3)} =
 \begin{pmatrix}
\vr_{1}
&
\cdots
&
\vr_{N-\nper}
\end{pmatrix}.
\]
Then, we are going to find a characterization of the left eigenvectors
$\{\vv_1^\top, \ldots, \vv_{\nper-1}^\top\}$ and of the right
eigenvectors $\{\vq_1, \ldots, \vq_{\nper-1}\}$.  Remarkable, they are
obtained by combining the entries of $\vv^\top$ and $\vone$,
respectively, with the complex $\nper$-roots of the unity, as the
following lemma points out.

\begin{lemma}[Characterization of the eigenvectors of $D_{\nper}$]
  \label{lem:eigvD_per}
  For any $j_1 = 1, \ldots, \nper-1$ and $l_2 = 1, \ldots, N$,
\[
\begin{aligned}
&[P^{(2)}]_{j_1,l_2} =
[{\vv_{j_1}^\top}]_{l_2} = v_{l_2} \lambda_{1,j_1}^{-h} =
e^{-i\frac{2\pi}{{\nper}}j_1h}  v_{l_2} ,\\
&[{P^{\dagger}}^{(2)}]_{l_2,j_1} =
[\vq_{j_1}]_{l_2} = \lambda_{1,j_1}^{h} = e^{i\frac{2\pi}{{\nper}}j_1h}  ,
\end{aligned}
\]
where $h$ is the cyclic class which the element $l_2$ belongs
to.
\end{lemma}

Notice that, if we had defined ${\vv_{j_1}^\top}$ and $\vq_{j_1}$ also
for $j_1 = 0$, we would have coherently obtained
${\vv_{0}^\top}=\vv^\top$ and $\vq_{0} = \vone$.

\begin{proof}
First, we prove that ${\vv_{j}^\top}$ and $\vq_{j}$ are eigenvectors
for $W^\top$ related to $\lambda_{1,j}$, then we prove the
orthonormalization condition.\\ \indent

First, let us recall that, by definition of cyclic classes, if $l_2$ belongs to
the $h$-th cyclic class, then
\begin{equation}\label{eq:algPeriodic}
[\vx^\top W^\top]_{l_2} =
\sum_{l \in \text{cyclic class }h-1} 
x_l [W^\top]_{l, l_2}
, \qquad
[W^\top \vy]_{l_2} =
\sum_{l \in \text{cyclic class }h+1} 
[W^\top]_{l_2,l} y_l\,.
\end{equation}
If we combine the definition of $\vv^\top_{j_1}$ and $\vq_{j_1}$ with
\eqref{eq:algPeriodic}, for $l_2$ belonging to the $h$-th cyclic
class, we obtain
\begin{align*}
[\vv^\top_{j_1} W^\top]_{l_2}  & =
\sum_{l \in \text{cyclic class }h-1} 
[\vv^\top_{j_1}]_{l} [W^\top]_{l, l_2}
 =
\sum_{l \in \text{cyclic class }h-1} 
v_{l} \lambda_{1,j_1}^{-(h-1)}  [W^\top]_{l, l_2}
\\ 
& =
\lambda_{1,j_1}\lambda_{1,j_1}^{-h} 
\Big( \sum_{l =1}^N 
v_{l} [W^\top]_{l, l_2} \Big) 
= \lambda_{1,j_1}\lambda_{1,j_1}^{-h} 
[\vv^\top W^\top]_{l_2}
 = 
\lambda_{1,j_1} [ \lambda_{1,j_1}^{-h} v_{l_2} ]
\\&= 
\lambda_{1,j_1} [\vv^\top_{j_1}]_{l_2}.
\end{align*}
The same holds for the right eigenvectors:
\begin{align*}
[W^\top \vq_{j_1} ]_{l_2}  & =
\sum_{l \in \text{cyclic class }h+1} 
[W^\top]_{l_2l} [\vq_{j_1} ]_{l} 
\\ 
& =
\sum_{l \in \text{cyclic class }h+1} 
[W^\top]_{l_2,l} 1 \lambda_{1,j_1}^{h+1} 
 =
\lambda_{1,j_1}\lambda_{1,j_1}^{h} 
[W^\top\vone]_{l_2}
 = 
\lambda_{1,j_1} [\vq_{j_1}]_{l_2}.
\end{align*}
For what concerns the orthonormalization condition, we note that
\[
\lambda_{1,j_1} {\vv_{j_1}^\top}\vq_{j_2} =
{\vv_{j_1}^\top}W^\top \vq_{j_2}  = \lambda_{1,j_2} {\vv_{j_1}^\top}\vq_{j_2},
\]
that means that $ {\vv_{j_1}^\top}\vq_{j_2} = 0$ if $j_1 \neq j_2$, as
in this case $\lambda_{1,j_1}\neq\lambda_{1,j_2}$.  Furthermore,
denoting by $h_l$ the cyclic class that element $l$ belongs to, we
have
\[
{\vv_{j_1}^\top}\vq_{j_1} =
\sum_{l=1}^N v_l \lambda_{1,j_1}^{-h_l}
\lambda_{1,j_1}^{+h_l}
=
\sum_{l=1}^N v_l = 1,
\] 
which completes the proof.
\end{proof}

The following result states that $\vv^\top$ divides the total
mass equally into each cyclic class.

\begin{lemma}\label{massOfBlocks}
The quantity $\sum_{l \in \text{cyclic class }h} v_{l}$ does not depend
on the choice of the cyclic class $h$. Then, for any $l \in \{1, \ldots , N\}$, 
\[
\sum_{l_2\sim_c l} v_{l_2} = \frac{1}{\nper}.
\]
\end{lemma}

\begin{proof}
By applying \eqref{eq:algPeriodic} multiple times, we obtain, for any
cyclic class $h$,
\begin{align*}
\sum_{l_1 \in \text{cyclic class }h} v_{l_1} & = 
\sum_{l_1 \in \text{cyclic class }h} v_{l_1} [\vone]_{l_1}
 = 
\sum_{l_1 \in \text{cyclic class }h} v_{l_1} [W^\top \vone]_{l_1}
\\
& = 
\sum_{l_1 \in \text{cyclic class }h} v_{l_1} \sum_{l_2 \in \text{cyclic class }h+1} 
[W^\top]_{l_1,l_2} [\vone]_{l_2}
\\
& = 
\sum_{l_2 \in \text{cyclic class }h+1} 
\Big( \sum_{l_1 \in \text{cyclic class }h} v_{l_1} 
[W^\top]_{l_1,l_2} \Big) [\vone]_{l_2}
\\
& = 
\sum_{l_2 \in \text{cyclic class }h+1} 
[\vv^\top W^\top]_{l_2} [\vone]_{l_2}
 = 
\sum_{l_2 \in \text{cyclic class }h+1} 
v_{l_2} [\vone]_{l_2}\\ &= 
\sum_{l_2 \in \text{cyclic class }h+1} 
v_{l_2} .
\end{align*}
The last part of the statement is a consequence of the normalizing
condition $\vv^\top \vone = 1$, because 
the number of cyclic classes is $\nper$.
\end{proof}

Finally, we are ready for presenting the proof of the main result of
this subsection, i.e. Theorem~\ref{teo:sp-project}.

\begin{proof}[Proof of Theorem~\ref{teo:sp-project}] We are going to prove
  the following equalities:
\begin{equation}\label{def:P1_P2_P3}
\begin{aligned}
&{\vZ}_{n}^{(1)} = \vone \vv^\top \vZ_n,
\qquad
{\vZ}_{n}^{(2)} = 
{P^{\dagger}}^{(2)} {P^{(2)}} \vZ_n = 
\sum_{j=1}^{\nper-1} \vq_{j}\vv_j^\top\vZ_n,
\\&\text{and}\qquad {\vZ}_{n}^{(3)} = 
{P^{\dagger}}^{(3)} {P^{(3)}} \vZ_n =
\sum_{j=1}^{N-\nper} \vr_{j} \vp_{j}^\top\vZ_n.
\end{aligned}
\end{equation}
Since ${\vZ}_{n}^{(1)} = \vone \vv^\top \vZ_n= \widetilde{Z}_n \vone$
by construction, for ${\vZ}_{n}^{(2)}$ we have to prove that
\[
(\vone \vv^\top 
+
{P^{\dagger}}^{(2)} {P^{(2)}} ) \vZ_n = 
{\vZ}_{n}^{(1)} + {\vZ}_{n}^{(2)}.
\]
Since by Lemma~\ref{massOfBlocks}
\[
[{\vZ}_{n}^{(1)} + {\vZ}_{n}^{(2)}]_l = 
\sum_{l_2 \sim_c l} \frac{v_{l_2}}{
\sum_{l_1 \sim_c l} v_{l_1}
} Z_{n,l_2} 
= 
\nper \sum_{l_2 \sim_c l} v_{l_2}
Z_{n,l_2},
\]
we have only to prove that  
\[
\big[ (\vone \vv^\top 
+
{P^{\dagger}}^{(2)} {P^{(2)}} ) \vZ_n \Big]_{l}
= 
\nper \sum_{l_2 \sim_c l} v_{l_2}
Z_{n,l_2} .
\]
Then, denoting by $h_1$ and $h_2$ the cyclic classes which $l_1$ and
$l_2$ belong to, respectively, we can use Lemma~\ref{lem:eigvD_per} to
obtain
\begin{align*}
v_{l_2} +
[{P^{\dagger}}^{(2)} {P^{(2)}}]_{l_1,l_2}
& = 
v_{l_2} +
\sum_{j = 1}^{\nper-1} \lambda_{1,j}^{h_1} 
v_{l_2} \lambda_{1,j}^{-h_2} 
 = 
v_{l_2} 
+
v_{l_2} 
\sum_{j = 1}^{\nper-1} \lambda_{1,j}^{h_1-h_2} 
\\ & = 
v_{l_2} 
+
v_{l_2} 
\sum_{j = 1}^{\nper-1} \lambda_{1,h_1-h_2}^j 
 = 
v_{l_2} 
\sum_{j = 0}^{\nper-1} \lambda_{1,h_1-h_2}^j.
\end{align*}
Now, $\lambda_{1,h_1-h_2}$
is a root of the unity, we have, by Lemma~\ref{lem:sumRoots},
\[
[\vone \vv^\top 
+
{P^{\dagger}}^{(2)} {P^{(2)}} ]_{l_1l_2} = 
\begin{cases}
v_{l_2} \nper & \text{if }l_1\sim_c l_2,
\\
0 & \text{otherwise},
\end{cases}
\]
and hence $\big[ (\vone \vv^\top 
+
{P^{\dagger}}^{(2)} {P^{(2)}} ) \vZ_n \Big]_{l}
= 
\nper \sum_{l_2 \sim_c l} v_{l_2}
Z_{n,l_2} $, as required.
The part of ${\vZ}_{n}^{(3)}$ is now obvious, since $I = P^{-1} P$, and hence
\begin{align*}
 {\vZ}_{n}^{(3)} & =
\vZ_n - ({\vZ}_{n}^{(1)}+{\vZ}_{n}^{(2)}) =
(I - (\vone \vv^\top 
+
{P^{\dagger}}^{(2)} {P^{(2)}} ) )
\vZ_n 
\\
& 
=
\bigg(\begin{pmatrix}
\vone
&
{P^{\dagger}}^{(2)}
&
{P^{\dagger}}^{(3)}
\end{pmatrix}
\begin{pmatrix}
\vv^\top
\\
{P^{(2)}}
\\
{P^{(3)}}
\end{pmatrix}
-
(\vone \vv^\top 
+
{P^{\dagger}}^{(2)} {P^{(2)}} ) \bigg)
\vZ_n 
 =
{P^{\dagger}}^{(3)} {P^{(3)}} \vZ_n. 
\end{align*}
As a consequence, we obtain that the three components $\vZ_n^{(i)}$,
with $i=1,2,3$ are linearly independent.  Finally, from Definition
\ref{def-sincro}, we have that the entire system almost surely
synchronizes in the limit if and only if there exists a process $(\vZ_n^*)_n$ of the
form $\vZ_n^*=Z_n^*\vone$ such that $(\vZ_n-\vZ_n^*)$ converges almost
surely to $\vzero$. Since ${\vZ}_{n}^{(1)} \in Span\{\vone\}$ and
the linear independence of the three
components, this fact is possible if and only if both the second and
the third component converge almost surely to zero.
\end{proof}


\subsection{Details for the proofs of the results stated in
  Section~\ref{sec:decomp-conv-sincro}} 

\begin{proof}[Proof of Theorem~\ref{teo:thirdpart}]

\noindent \emph{Modification of the Jordan space:}\\ We show how to
replace the Jordan block $J_{W,\lambda}$ with a new block
$J_{W,\beta,\lambda}$ such that $\vert\vert J_{W,\beta,\lambda}\vert\vert _{2,2} < 1$. To
this end, for any real $\beta\neq 0$ let us define $D_\beta = diag (1,
\beta, \beta^2, \ldots , \beta^{n_j-1}) $.  We have
\begin{multline*}
D_\beta  J_{W,\lambda} =
\begin{pmatrix}
1 & 0 & 0 & \ldots & 0 \\
0 & \beta & 0 & \ldots & 0 \\
\vdots & \vdots &\vdots & & \vdots \\
0 & 0 & 0 & \ldots & 0 \\
0 & 0 & 0 & \ldots & \beta^{n_j-1}
\end{pmatrix}
\begin{pmatrix}
\lambda & 1 & 0 & \ldots & 0 \\
0 & \lambda & 1 & \ldots & 0 \\
\vdots & \vdots &\vdots & & \vdots \\
0 & 0 & 0 & \ldots & 1 \\
0 & 0 & 0 & \ldots & \lambda
\end{pmatrix} 
\\
= 
\begin{pmatrix}
\lambda & \tfrac{1}{\beta} & 0 & \ldots & 0 \\
0 & \lambda & \tfrac{1}{\beta} & \ldots & 0 \\
\vdots & \vdots &\vdots & & \vdots \\
0 & 0 & 0 & \ldots & \tfrac{1}{\beta} \\
0 & 0 & 0 & \ldots & \lambda
\end{pmatrix}
\begin{pmatrix}
1 & 0 & 0 & \ldots & 0 \\
0 & \beta & 0 & \ldots & 0 \\
\vdots & \vdots &\vdots & & \vdots \\
0 & 0 & 0 & \ldots & 0 \\
0 & 0 & 0 & \ldots & \beta^{n_j-1}
\end{pmatrix}
=
J_{W,\beta,\lambda} 
D_\beta .
\end{multline*}
Then, using the above relations 
$ P_{\lambda} W^\top =J_{W,\lambda} P_{\lambda}$ and
$D_\beta  J_{W,\lambda}=J_{W,\beta,\lambda} D_\beta$, 
if we define $P_{\beta,\lambda} = D_\beta P_{\lambda} $, we have
\begin{equation}\label{eq:JWbeta}
P_{\beta,\lambda}W^\top = D_\beta P_{\lambda} W^\top = 
D_\beta  J_{W,\lambda} P_{\lambda}
= 
J_{W,\beta,\lambda} D_\beta P_{\lambda}
= J_{W,\beta,\lambda} P_{\beta,\lambda}.
\end{equation}
Roughly speaking, $P_{\beta,\lambda}W^\top = J_{W,\beta,\lambda}
P_{\beta,\lambda}$ means that $P_{\beta,\lambda}$ is a
$\tfrac{1}{\beta}$-Jordan base for the $\tfrac{1}{\beta}$-Jordan block
$J_{W,\beta,\lambda}$. Obviously, $\vert \tfrac{1}{\beta}\vert $ may be
arbitrary small if $\vert \beta\vert $ is big enough, so that we expect $\vert\vert 
J_{W,\beta,\lambda} \vert\vert _{2,2} $ so close to $\vert \lambda\vert $ to be strictly
smaller than $1$. To prove this fact, set $\bar{J} =
J_{W,\beta,\lambda} J_{W,\beta,\lambda}^*$, that can be easily
computed:
\[
\bar{J} = J_{W,\beta,\lambda} J_{W,\beta,\lambda}^*
= 
\begin{pmatrix}
  \vert \lambda\vert  + \tfrac{1}{\vert \beta\vert ^2} & \tfrac{\bar{\lambda}_j}{\beta} & 0 &
  \ldots & 0 & 0 & 0\\
  \tfrac{{\lambda}}{\bar{\beta}}  & \vert \lambda\vert  + \tfrac{1}{\vert \beta\vert ^2} &
  \tfrac{\bar{\lambda}_j}{\beta} & \ldots & 0 & 0 & 0 \\
\vdots & \vdots &\vdots & & \vdots & \vdots & \vdots \\
0 & 0 & 0 & \ldots & \tfrac{{\lambda}}{\bar{\beta}}  &
\vert \lambda\vert  + \tfrac{1}{\vert \beta\vert ^2} & \tfrac{\bar{\lambda}_j}{\beta} \\
0 & 0 & 0 & \ldots & 0 & \tfrac{{\lambda}}{\bar{\beta}}  & \vert \lambda\vert  
\end{pmatrix}.
\]
Then, we get 
\[
\vert\vert  J_{W,\beta,\lambda} \vert\vert _{2,2} \leq \max_{j= 1, \ldots, N} \Big\{\sum_{i=1}^N \vert [\bar{J}]_{ij}\vert \Big\}
= (\vert \lambda\vert  + \tfrac{1}{\vert \beta\vert })^2.
\]
  Now, let $1/ \breve{\beta}= \sqrt{\frac{1 +\vert \lambda\vert }{2}} -
  \vert \lambda\vert >0$ so that
\begin{equation*}
  \vert\vert J_{W,\breve{\beta},\lambda}\vert\vert _{2,2}
  \leq \frac{1 + \vert \lambda\vert }{2} < 1.
\end{equation*}
This concludes the step of the modification of the Jordan space.\\

\emph{Second part of the proof ($\Rightarrow$):}\\

\noindent  Here we present the technical details related to 
Step (i) and Step (ii) of the second part of the proof of 
Theorem~\ref{teo:thirdpart}.\\

\emph{Step (i)}\\
Take $n_0\geq N$ such
that the set $A_0=\{0< Z_{n_0,l_2}\leq Z_{n_0,l_1}< 1\}$ has strictly
positive probability (this is possible by the irreducibility of $W$
and the hypothesis $P(T_0)<1$). Now, take the set
$A=A_0\cap\{X_{n,l_1}=1,X_{n,l_2}=0, \forall n_0< n \leq n_1\}$ (where
$n_1$ will be determined more ahead), and notice that, on $A$, we have
$$
Z_{n_1,l_2} = Z_{n_0,l_2}\prod_{k=n_0}^{n_1-1}(1-r_k),
\qquad\text{and}\qquad
1-Z_{n_1,l_1} = (1-Z_{n_0,l_1})\prod_{k=n_0}^{n_1-1}(1-r_k).
$$ Then, denoting $\Delta Z_{n_1} = Z_{n_1,l_1} - Z_{n_1,l_2}$ and
$\Delta Z_{n_0} = Z_{n_0,l_1} - Z_{n_0,l_2}$, on $A$, we have that
$$
\Delta Z_{n_1} = 1 - (1-\Delta Z_{n_0})\prod_{k=n_0}^{n_1-1}(1-r_k)\geq
1 - \prod_{k=n_0}^{n_1-1}(1-r_k),
$$ which is strictly positive as $0<r_n <1$ by definition.  Now, set
$0<\epsilon< 1 - \prod_{k=n_0}^{\infty}(1-r_k)$ and fix $n_1$
sufficiently large so that $1 - \prod_{k=n_0}^{n_1-1}(1-r_k) >
\epsilon$ and $\sum_{n \geq n_1} r_n < \epsilon/3$. Since it can be
easily proved that $P(A)>0$ whenever $r_n<1$, what we have shown is
that $P(\Delta Z_{n_1}>\epsilon)>0$.  However, since for both $l=l_1,
l_2$ it holds that $\max_{n\geq n_1}\{\vert Z_{n_1,l}-Z_{n,l}\vert \}<\epsilon/3$,
we have that $P(\max_{n\geq n_1}\{\vert \Delta Z_{n}\vert \}>\epsilon)>0$.
Therefore, we cannot have that $a.s.-\lim_{n\to+\infty}
Z_{n,l_2}-Z_{n,l_1}=0$ on $A$.

\emph{Step (ii)}\\ Notice that, since $N>\nper$, there exist two
different states that belong to the same cyclic class, call them $l_1$
and $l_2$.  Hence, applying the result of Step (i) to the pair
($l_1,l_2$) within the same cyclic class, we have that
${\vZ}_{n}^{(3)}\not\to\vzero $ with a strictly positive probability,
since ${\vZ}_{n}^{(1)}$ and ${\vZ}_{n}^{(2)}$ are constant on the same
cyclic class.
\end{proof}

\begin{proof}[Proof of Corollary~\ref{th-sincro-periodic}]
Note that, if $l_1\sim_c l_2$, then by definition we have
$({Z}_{n,l_1}^{(1)}+{Z}_{n,l_1}^{(2)})=({Z}_{n,l_2}^{(1)}+{Z}_{n,l_2}^{(2)})$.
Hence, if $l_1\sim_c l_2$, we have ${Z}_{n,l_1}-{Z}_{n,l_2} =
{Z}_{n,l_1}^{(3)}-{Z}_{n,l_2}^{(3)} $, which tends to zero a.s. for
all these pairs $(l_1, l_2)$ with $l_1\sim_c l_2$ if and only if all
${Z}_{n,l_1}^{(3)}$, ${Z}_{n,l_2}^{(3)}$ tends to zero a.s., that by
Theorem~\ref{teo:thirdpart} occurs if and only if $\sum_n r_n=
+\infty$.  This corresponds to the almost sure asymptotic
synchronization within each cyclic class.  Recall that it is not
possible that
$a.s.-\lim_n{Z}_{n,l_1}^{(3)}=a.s.-\lim_n{Z}_{n,l_2}^{(3)}\neq 0$ for
all pairs $(l_1, l_2)$ with $l_1\sim_c l_2$ as this would mean that
$a.s.-\lim_n{\vZ}_{n}^{(3)}$ belongs to $Span\{\vone,
\vq_1,\dots,\vq_{\nper-1}\})$, which is impossible by
Theorem~\ref{teo:sp-project}.
\end{proof}

\begin{proof}[Proof of Theorem~\ref{teo:secondpart}] 
Here we present some calculations related to the 
terms $a_{j_1,n}$ and $C_{j_1,n}$ derived 
in the Proof of Theorem~\ref{teo:secondpart}.\\

\indent \emph{(i) Term $a_{j_1,n}$}\\
We have 
\begin{equation*}
\begin{aligned}
a_{j_1,n} &&=&\
1+r_n^2(1-\Re(\lambda_{1,j_1}))^2 + r_n^2\Im(\lambda_{1,j_1})^2
- 2r_n(1 - \Re(\lambda_{1,j_1}))\\
&&=&\  1 - 2r_n(1-r_n)(1 - \Re(\lambda_{1,j_1}))\\ 
&&=&\ 1-s_{j_1,n},
\end{aligned}
\end{equation*}
where, since $\vert \lambda_{1,j_1}\vert =1$, we can write
$s_{j_1,n}=2r_n(1-r_n)(1 - \cos(\tfrac{2 \pi}{{\nper}}j_1))$.  Notice
that, since we are considering ${\nper}\geq 2$ and $1\leq j_1\leq
\nper-1$, we have $-1\leq \cos(\tfrac{2 \pi}{{\nper}}j_1)<1$, which,
combined with $0<r_n(1-r_n)\leq 1/4$, implies $0<s_{j_1,n}\leq 1$.
\\

\indent \emph{(ii) Term $C_{j_1,n}$}\\ 
By \eqref{eq:boundM2} and
Lemma~\ref{lem:eigvD_per}, we have that $\max_{l=1,\ldots,N}
\{[\bar{\vv}_{j_1} \vv_{j_1}^\top ]_{l,l}\} =
\max_{l=1,\ldots,N}\{v_l^2\}\leq 1$, and hence we can bound $C_{j_1,n}$
as follows:
$$
0\ \leq\ C_{j_1,n}\leq r_n^2 V_n.
$$
\end{proof}

\begin{proof}[Proof of Corollary~\ref{th-sincro}]
The spectral representation given in Theorem~\ref{teo:sp-project}
states that any linear combination of ${\vZ}_{n}^{(2)} $ and
${\vZ}_{n}^{(3)}$ involves the base vectors given in $P^{-1}$ except
the vector $\vone$, which is related to ${\vZ}_{n}^{(1)} $.  Then, the
complete almost sure asymptotic synchronization coincides with the
asymptotic vanishing of the processes ${\vZ}_{n}^{(2)} $ and
${\vZ}_{n}^{(3)}$, and so it is a direct consequence of
Theorem~\ref{teo:thirdpart} and Theorem~\ref{teo:secondpart}.
\end{proof}


\subsection{Details for the proofs of the results stated in Section~\ref{sec:non-sincro}}

\begin{proof}[Proof of Proposition~\ref{th-r_n-finite-convergent}]
Fix $l$, the existence of the almost sure limit of $(Z_{n,l})_n$  
follows by the fact that, under the assumption $\sum_n r_n<+\infty$,
  the process $(Z_{n,l})_n$ is a non-negative almost
    super-martingale (see \cite{rob}). Indeed, we have
$$
    E[Z_{n+1,l}\vert \mathcal{F}_n] =Z_{n,l}-r_nZ_{n,l}+
    r_nE[X_{n+1,l}\vert \mathcal{F}_n]\leq Z_{n,l}+\xi_{n,l}
$$ where $\sum_n \xi_{n,l}=\sum_n r_n E[X_{n+1,l}\vert \mathcal{F}_n]\leq \sum_n
    r_n <+\infty$. Then, since $\widetilde{Z}_n$, and
so $\vZ^{(1)}_n$ (see~\eqref{def:Z_1}), converges almost surely, we
have that also ${\vZ}_{n}^{(2)}$ converges almost surely because by
definition ${\vZ}_{n}^{(2)}$ is a function of ${\vZ}_{n}$
(see~\eqref{def:Z_2}). Finally, the almost sure convergence of
${\vZ}_{n}^{(3)}$ follows by~\eqref{eq:decomposZ} as
${\vZ}_{n}^{(3)}={\vZ}_{n}-{\vZ}_{n}^{(1)}-{\vZ}_{n}^{(2)}$. 
The last statement of the result is proved in Step (i) of the second part
$(\Leftarrow)$ of the Proof of Theorem \ref{teo:thirdpart}.
\end{proof}

We also observe that, in the scenario $\sum_n r_n<+\infty$, for each $l$, the limit random 
variable $Z_{\infty,l}$ of the $(Z_{n,l})_n$ cannot touch the barriers 
when the process starts in $(0, 1)$. This fact follows from Lemma~\ref{lemma-tec-appendix-1} 
applied to $y_n=Z_{n,l}(\omega)$ and to $y_n=1-Z_{n,l}(\omega)$, with $\omega\in \{0<Z_{0,l}<1\}$. 
Indeed, we have $Z_{n+1,l}\geq (1-r_n)Z_{n,l}$ and $(1-Z_{n+1,l})\geq (1-r_n)(1-Z_{n,l})$.
\\

Now, we give the proof of Lemma~\ref{lem:modulusSpectral}, used in the proof of 
Theorem~\ref{th-norm}.

\begin{proof}[Proof of Lemma~\ref{lem:modulusSpectral}]
First, we show that $O_\nper$ is an orthogonal matrix.
To this end, notice that for any $j_1, j_2\in\{0,\dots,\nper-1\}$, we can write
\[
(O_\nper^* O_\nper)_{j_1j_2} =
\frac{1}{\nper}
\sum_{j = 0}^{\nper-1} 
\lambda_{1,j}^{-j_1} \lambda_{1,j_2}^{j}
=
\frac{1}{\nper}
\sum_{j = 0}^{\nper-1} 
\lambda_{1,-j_1}^{j} \lambda_{1,j_2}^{j}
=
\frac{1}{\nper}
\sum_{j = 0}^{\nper-1} 
\lambda_{1,j_2-j_1}^{j}.
\]
Then, by Lemma~\ref{lem:sumRoots} we have
\[
\frac{1}{\nper}
\sum_{j = 0}^{\nper-1} 
\lambda_{1,j_2-j_1}^{j} 
= 
\begin{cases}
1 & \text{if }j_1=j_2,
\\
0 & \text{otherwise},
\end{cases}
\]
which concludes the proof that $O_\nper$ is orthogonal.

Now, we focus on proving that $\vZ_n^{(c)} =\nper^{\frac{3}{2}}
O_\nper \vC_n$.  For this purpose, first notice that by
\eqref{eq:defZnB} and Lemma~\ref{massOfBlocks}, for any
$h_1\in\{0,\dots, \nper-1\}$ we have
$$
Z_{n,h_1}^{(c)} = \sum_{l_2 \in \text{cyclic class }h_1} \frac{v_{l_2}}{
\sum_{l \in \text{cyclic class }h_1} v_l} Z_{n,l_2}=
\nper \sum_{l_2 \in \text{cyclic class }h_1} v_{l_2} Z_{n,l_2}  ,
$$
and hence it is enough to prove that 
$$
\sqrt{\nper}[ O_{\nper}\vC_n]_{(h_1+1)} = 
\sum_{l_2 \in \text{cyclic class }h_1} v_{l_2} Z_{n,l_2}.
$$
Recalling that $\vC_n = (\begin{smallmatrix}
\vv^\top 
\\
{P^{(2)}} 
\end{smallmatrix}
) \vZ_n$, this is the same as proving that, for each $l_2\in\{1,\dots,N\}$,
$$
\sqrt{\nper}\Big[ O_{\nper}\begin{pmatrix}
\vv^\top 
\\
{P^{(2)}} 
\end{pmatrix}\Big]_{(h_1+1) l_2} = 
\begin{cases}
v_{l_2} & \text{if $h_2=h_1$},
\\
0 & \text{otherwise},
\end{cases}
$$
where with $h_2$ we have denoted the cyclic class which $l_2$ belongs to.
To this end, by
Lemma~\ref{lem:eigvD_per} and Lemma~\ref{lem:sumRoots} we obtain
\begin{align*}
\sqrt{\nper} \Big[ O_\nper 
\begin{pmatrix}
\vv^\top 
\\
{P^{(2)}} 
\end{pmatrix}\Big]_{(h_1+1)l_2}
& =
v_{l_2} +
\sum_{j_2=1}^{\nper-1}
[\sqrt{\nper} O_\nper]_{(h_1+1)(j_2+1)}  ({P^{(2)}})_{j_2l_2}
\\
& =
v_{l_2} +
\sum_{j_2=1}^{\nper-1}
\lambda_{1,h_1}^{j_2}
v_{l_2} \lambda_{1,j_2}^{-h_2}
\\ & =
v_{l_2} + v_{l_2} 
\sum_{j_2=1}^{\nper-1}
\lambda_{1,h_1}^{j_2}
\lambda_{1,-h_2}^{j_2}
\\
& =
v_{l_2} 
\sum_{j_2=0}^{\nper-1}
\lambda_{1,h_1-h_2}^{j_2}
 =
\begin{cases}
v_{l_2} & \text{if $h_2=h_1$},
\\
0 & \text{otherwise}.
\end{cases}
\end{align*}
This concludes the proof that $ \vZ_n^{(c)} =\nper^{\frac{3}{2}} O_\nper \vC_n$.

Finally, since
$O_\nper$ is an orthonormal matrix, we immediately have that $\vert\vert \vZ_n^{(c)}\vert\vert ^2
=\nper^{3} \vert\vert O_\nper \vC_n\vert\vert ^2 =\nper^{3} \vert\vert \vC_n\vert\vert ^2$, where $\vert\vert \vC_n\vert\vert ^2=
{\vC}_{n+1}^*\vC_{n+1}$.
\end{proof}

The following proofs concern 
Lemma \ref{lemma:VnTo0} and Lemma \ref{lem:subseqRn}, 
used in the proof of Theorem~\ref{teo:PeriodicCycle}. 

\begin{proof}[Proof of Lemma \ref{lemma:VnTo0}]
Since $V_n=(\boldsymbol{1} - W^\top\boldsymbol{Z}_{n})^{\top} W^\top
\boldsymbol{Z}_{n}=(\boldsymbol{1}^\top - \boldsymbol{Z}_{n}^\top
W)^{\top} W^\top \boldsymbol{Z}_{n}$, it is enough to prove that,
given a vector $\vx$, with $x_l\in[0,1]$, for any $\epsilon>0$, there
exists $\delta>0$ such that
$$
(\vone^\top-\vx^\top W)W^\top \vx < \delta
\Rightarrow
(\vone^\top-\vx^\top)\vx <  \epsilon.$$ 
In order to prove this fact, first notice that
$$(\vone^\top-\vx^\top W)W^\top \vx =
\sum_{l_2=1}^N(1-[\vx^\top W]_{l_2})[W^\top \vx]_{l_2} \geq
\max_{l_2=1,\dots,N}\big\{(1-[\vx^\top W]_{l_2})[W^\top \vx]_{l_2}\big\},$$
and since $t(1-t) \geq \frac{1}{2} \min\{t,1-t\}$ for $t\in [0,1]$,
$$\max_{l_2=1,\dots,N}\big\{(1-[\vx^\top W]_{l_2})[W^\top \vx]_{l_2}\big\} \geq
\frac{1}{2}\max_{l_2=1,\dots,N}\big\{\min\big\{(1-[\vx^\top W]_{l_2}),[W^\top \vx]_{l_2}
\big\}\big\},
$$
which implies that, for any $\delta>0$,
$$
(\vone^\top-\vx^\top W)W^\top \vx < \delta\qquad 
\Longrightarrow\qquad
\min \Big\{ 1-[\vx^\top W]_{l_2}, [W^\top \vx]_{l_2} \Big\} < 2\delta\; \forall l_2.
$$
Finally,
setting $w_{\min} =
\min\{[W]_{l_1,l_2}=w_{l_1,l_2}\colon w_{l_1,l_2}\neq 0\}>0$ and
$2\delta = w_{\min}\epsilon $, we need to prove that,
$$
\min \Big( 1-[\vx^\top W]_{l_2}, [W^\top \vx]_{l_2} \Big) < 2\delta\;
\forall l_2\qquad
\Longrightarrow\qquad
\min \{ 1-x_{l_1},  x_{l_1}\} < \epsilon\;\forall l_1,
$$
which follows by showing its contrapositive as follows:
\begin{equation*}
\begin{aligned}
x_{l_1} \geq \epsilon 
& \qquad 
\Longrightarrow
\qquad 
    [W^\top \vx]_{l_2}=\sum_{l_1}w_{l_1,l_2}x_{l_1} \geq w_{\min}\epsilon=2\delta,
\\
1-x_{l_1} \geq \epsilon 
& \qquad 
\Longrightarrow
\qquad 
1-[W^\top \vx]_{l_2}=\sum_{l_1}w_{l_1,l_2}(1-x_{l_1}) \geq w_{\min}\epsilon=
2\delta,
\end{aligned}
\end{equation*}
for all $l_2$ such that $ w_{l_1,l_2}>0$.
This concludes the proof.
\end{proof}

\begin{proof}[Proof of Lemma \ref{lem:subseqRn}]
First, we will prove that $\sum_n \delta_n =+\infty$.
To this end, assume by contradiction that $\sum_n \delta_n = M < +\infty$,
which implies
$r_n\leq 1/2$ for any $n > \tau_M$. Hence, we can write
\[
\sum_n r_n(1-r_n) = 
\sum_{n \leq \tau_M} r_n(1-r_n) +
\sum_{n > \tau_M} r_n(1-r_n) 
\geq 
\sum_{n \leq \tau_M} r_n(1-r_n) +
\frac{1}{2}\sum_{n > \tau_M} r_n = +\infty,
\]
which contradicts the assumption $\sum_n r_n(1-r_n)<+\infty$.
This concludes the first part of the proof. 

We can now define the sequence $(\epsilon_n)_n$ as follows:
for $n \geq 0$, set 
\[
\epsilon_n = 
\begin{cases}
0 & \text{if }\tau_{n+1} = \tau_{n} + 1,
\\
\sum_{p=\tau_{n} + 1}^{\tau_{n+1}-1}
r_p
& \text{if }\tau_{n+1} > \tau_{n} + 1,
\end{cases}
\]
which has the property that 
$$\sum_n \epsilon_n 
= \sum_n \mathbbm{1}_{\{\tau_{n+1} > \tau_{n} + 1\}}\sum_{p=\tau_{n} + 1}^{\tau_{n+1}-1}r_p 
= \sum_n r_n\mathbbm{1}_{\{\delta_n=0\}}.$$
First we will show that $\sum_n \epsilon_n < +\infty$.
Note that, when $\delta_n=0$, $r_n \leq1/2$ and hence $2(1-r_n)\geq 1$,
which can be used to get the following
\[
\sum_n \epsilon_n = \sum_n\mathbbm{1}_{\{\delta_n=0\}} r_n \leq
\sum_n \mathbbm{1}_{\{\delta_n=0\}}r_n 2(1-r_n) \leq 2 \sum_n r_n (1-r_n) < +\infty.
\]
Finally, we will show that $\sup_{m \in (\tau_n,\tau_{n+1}]}
  \{\vert Z_{m,l_1} - Z_{\tau_{n+1},l_1}\vert \} < \epsilon_n $ for any
  $l_1\in\{1,\ldots,N\}$.  To this end, by \eqref{interacting-2-intro}
  and the triangular inequality, we get for $m \in \{\tau_n
  +1,\ldots,\tau_{n+1} \}$ that
\[
\begin{aligned}
\vert Z_{m,l_1} - Z_{\tau_{n+1},l_1}\vert  
&&\leq& 
\sum_{p = m+1}^{\tau_{n+1}} 
\vert Z_{p,l_1} - Z_{p-1,l_1} \vert 
\leq 
\sum_{p = m+1}^{\tau_{n+1}} 
r_{p-1} \vert X_{p,l_1} - Z_{p-1,l_1} \vert \\
&&\leq& 
\sum_{p = \tau_n +2}^{\tau_{n+1}} 
r_{p-1} 
=
\epsilon_n. \qedhere
\end{aligned} 
\]
\end{proof}

Finally, we present the details of the proof of Theorem~\ref{teo:PeriodicCycle}, 
whose notation and structure was described in the main article. 

\begin{proof}[Proof of Theorem~\ref{teo:PeriodicCycle}]

\noindent \emph{Step (1)}\\ Here we prove the stationary dynamics of
$(\vZ_n)_n$ between the clock times $(\sigma_n)_n$, that is
\[
\sup_{m_1,m_2 \in \{\sigma_{n-1},\ldots,\sigma_{n}-1\}}
 \{\vert\vert  \vZ_{m_1} - \vZ_{m_2}
  \vert\vert \}\stackrel{a.s.}\longrightarrow 0.
\]
This simply follows by Lemma~\ref{lem:subseqRn}, as we have  
\begin{multline*}
\sup_{m_1,m_2 \in \{\sigma_{n-1},\ldots,\sigma_{n}-1\}} \{\vert\vert  \vZ_{m_1} - \vZ_{m_2} \vert\vert  \}
\\
\leq 
\sup_{m_1,m_2 \in \{\tau_{n-1}+1,\ldots,\tau_{n}\}} 
\Big\{
\vert\vert  \vZ_{m_1} - \vZ_{\tau_n} \vert\vert  
+
\vert\vert  \vZ_{m_2} - \vZ_{\tau_n} \vert\vert  
\Big\}
\leq 2 \sqrt{N} \epsilon_n,
\end{multline*}
that goes to $0$ almost surely.\\

\noindent \emph{Step (2)}\\
Now, $r_{\tau_n} \to 1 $ by
Lemma~\ref{lem:subseqRn} and this, together with 
the definition of $\vZ_n$ in \eqref{eq:dynamics}, immediately implies 
\begin{equation}\label{eq:XsimZ}
  \vX_{\sigma_n} - \vZ_{\sigma_n} =
  (1-r_{\tau_n}) (\vX_{\tau_n+1} - \vZ_{\tau_n})
  \stackrel{a.s.}\longrightarrow \vzero.
\end{equation}

\noindent \emph{Step (3a)}\\
Let $ (\epsilon_n)_n$ as in Lemma~\ref{lem:subseqRn}, so that
\begin{equation}\label{eq:Ztau2sigma}
\vert Z_{\tau_{n+1}}-Z_{\sigma_n}\vert  \leq \epsilon_n.
\end{equation}
Since $r_{\tau_n} \to 1$, take $n_*$
so that $m \geq n_*$ implies $r_{\tau_{m-1}} -\epsilon_{m-1} > 1/2$.
Then, when $m \geq n_*$,
by \eqref{eq:dynamics} and \eqref{eq:Ztau2sigma}, we have
\begin{equation}\label{eq:ZsincroX1}
\begin{aligned}
&X_{\sigma_{m-1},l_1} = 1 
&&\Longleftrightarrow
&&Z_{\sigma_{m-1},l_1} \geq r_{\tau_{m-1}} 
&&\Longleftrightarrow
&&Z_{\tau_{m},l_1} 
\geq 
r_{\tau_{m-1}} -\epsilon_{m-1},
\\
&X_{\sigma_{m-1},l_1} = 0 
&&\Longleftrightarrow
&&Z_{\sigma_{m-1},1} \leq 1- r_{\tau_{m-1}} 
&&\Longleftrightarrow
&&Z_{\tau_{m},l_1} 
\leq 
1 - r_{\tau_{m-1}} +\epsilon_{m-1} ,
\end{aligned}
\end{equation}
and hence, for each cyclic class $h$,
\[
\forall m \geq n_*,\qquad 
B_{X,m-1}^h(0) \subseteq B_{Z,m}^h(0)\qquad\text{and}\qquad
B_{X,m-1}^h(1) \subseteq B_{Z,m}^h(1).
\]

\noindent \emph{Step (3b)}\\
Note that, by \eqref{eq:algPeriodic} and \eqref{eq:ZsincroX1}, if $l$
belongs to the $(h-1)$-th cyclic class, then, on $B_{m-1}$ and $m\geq
n_*$, we have that $[W^\top \vZ_{\tau_m}]_{l}$ is equal to
\begin{equation}\label{eq:WZtau}
\sum_{l_1 \in \text{cyclic class }h} 
[W^\top]_{l,l_1} Z_{\tau_{m},l_1}
\quad
\begin{cases}
\geq 
r_{\tau_{m-1}} -\epsilon_{m-1}  & \text{if }Z_{\sigma_{m-1},l_1} \geq r_{\tau_{m-1}} ,
\\
\leq 
1 - r_{\tau_{m-1}} +\epsilon_{m-1} 
& \text{if }Z_{\sigma_{m-1},1} \leq 1- r_{\tau_{m-1}} ,
\end{cases}
\end{equation}
and hence, for each cyclic class $h$,
\[
\forall m \geq n_*,\qquad 
B_{Z,m}^h(0) \subseteq B_{WZ,m}^{h-1}(0)\qquad\text{and}\qquad
B_{Z,m}^h(1) \subseteq B_{WZ,m}^{h-1}(1).
\]

\noindent \emph{Step (3c)}\\
By the definition of $A_m$, it is immediate to see that, on $A_m$, we have
$B_{WZ,m}^{h-1}(0)\equiv B_{X,m}^{h-1}(0)$ and $B_{WZ,m}^{h-1}(1)\equiv B_{X,m}^{h-1}(1)$
for any cyclic class $h$.\\

\noindent \emph{Step (4a)}\\
The definition of $V_{\tau_{n_0}}$, together with 
the relation $2t(1-t) \geq \min\{t,1-t\}$ for $t\in [0,1]$,
implies that, for any $l= 1,\ldots,N$, we have 
\[
\min \big\{
[W^\top \vZ_{\tau_{n_0}}]_l , ( 1- [W^\top \vZ_{\tau_{n_0}}]_l ) 
\big\}
\leq 
2[W^\top \vZ_{\tau_{n_0}}]_l ( 1- [W^\top \vZ_{\tau_{n_0}}]_l ) 
\leq 2V_{\tau_{n_0}},
\]
and hence 
\begin{equation}\label{eq:maxVn}
\max \big\{
[W^\top \vZ_{n_0}]_l , ( 1- [W^\top \vZ_{n_0}]_l ) 
\big\} \geq ( 1 - 2V_{n_0}).
\end{equation}
The definition of $A_{n_0}$ reads
\[
P( A_{n_0}) = \prod_{l=1}^N 
\begin{cases}
[W^\top \vZ_{\tau_{n_0}}]_l & \text{if }[W^\top \vZ_{\tau_{n_0}}]_l \geq \tfrac{1}{2},
\\
1 -[W^\top \vZ_{\tau_{n_0}}]_l & \text{if }[W^\top \vZ_{\tau_{n_0}}]_l < \tfrac{1}{2},
\end{cases}
\]
and hence by \eqref{eq:maxVn}, $P( A_{n_0}) \geq (1-2V_{n_0})^N$
that goes to $1$ by Theorem~\ref{th-limit-class-0-1}.

Since $ P( A_{n_0} \cap B_{n_0}) = 
P( B_{n_0} \vert  A_{n_0})P( A_{n_0}) $, for what concerns the remaining term
$P( B_{n_0} \vert  A_{n_0})$,
we will prove the sufficient condition $P(  B_{n_0,\infty} ) \to 1$.
Recall the definition of $\liminf$ for a sequence of sets, that
reads $\{\liminf_{n\rightarrow \infty} B_n\} = \cup_{n_0} \cap_{n \geq n_0} B_n$, so that the thesis becomes
\[
\lim_{n_0 \to +\infty}P(  B_{n_0,\infty} )
= P( \cup_{m=1}^\infty B_{m,\infty}) =
P( \liminf_{n\rightarrow \infty} B_{n} ) = 1.
\]
Now, using \eqref{eq:XsimZ} in Step 2
and Theorem~\ref{teo:thirdpart}, we have
\[
\vX_{\sigma_n} - \vZ^{(C)}_{\sigma_n}
= \Big( \vZ_{\sigma_n} -
 (\vZ^{(1)}_{\sigma_n} + \vZ^{(2)}_{\sigma_n} ) \Big) + (\vX_{\sigma_n} - \vZ_{\sigma_n})
\stackrel{a.s.}\longrightarrow \vzero .
\]
Now, $ (\vZ^{(C)}_{\sigma_n} )_n $ is constant on each cyclic class,
then the same holds asymptotically for $(\vX_{\sigma_n})_n$, that
assumes values in $\{0,1\}$.  In other words
\[
P(\liminf_n B_{n} )= P(\{B_{n} , \text{eventually}\}) = 1.
\]

\noindent \emph{Step (4b)}\\
Notice that the relation $A_m \cap B_{X,m-1}^h(g) \subseteq B_{X,m}^{h-1}(g)$
proved in Step 3 implies also that $A_{m} \cap B_{m-1}\subseteq B_{m} $.
As a consequence, for any $m > \max\{n_*,n_0\}$, we have
\[
P( A_{m} \cap B_{m} \vert  A_{n_0,m}  \cap B_{n_0,m}  )
= 
P( A_{m} \vert  A_{n_0,m}  \cap B_{n_0,m}  ).
\]
Since $ A_{n_0,m} \cap B_{n_0,m} \subseteq B_{m-1}$, the definition of
$A_{m}$, together with \eqref{eq:ZsincroX1} and \eqref{eq:WZtau},
implies
\[
P( A_{m} \vert  A_{n_0,m}  \cap B_{n_0,m} ) \geq ( r_{\tau_{m-1}} -\epsilon_{m-1}   )^N 
= \Big( 1 - \big( (1-r_{\tau_{m-1}} )+\epsilon_{m-1}  \big) \Big)^N .
\]
Summing up, for $n_0 \geq n_*$,
\begin{align*}
\prod_{m=n_0+1}^{\infty} 
P( A_{m} \cap B_{m} \vert  A_{n_0,m}  \cap B_{n_0,m}  ) 
& 
=
\prod_{m=n_0+1}^{\infty}
P( A_{m} \vert  A_{n_0,m}  \cap B_{n_0,m}  ) 
\\
& \geq 
\Big(
\prod_{m=n_0+1}^{\infty}
\big[ 
1 - \big( (1-r_{\tau_{m-1}} )+\epsilon_{m-1}  \big) 
\big]
\Big)^N\,,
\end{align*}
that goes to $1$ when $n_0\to+\infty$, because $\sum_m (
(1-r_{\tau_{m-1}} )+ \epsilon_{m-1}) < +\infty$ by
Lemma~\ref{lem:subseqRn}.  \\

\noindent \emph{Step (5)}\\
Under the further request that $\sum_n(1-r_n)<+\infty$, we have that
$r_n \to 1$.  The definition of $(\tau_n)_n$ given in
Lemma~\ref{lem:subseqRn} implies that, eventually, $\sigma_{n+1}-1=
\tau_{n+1}= \tau_{n}+1 = \sigma_{n}$. Then, the stationary part of the
dynamics disappears and the thesis follows.
\end{proof}


\section{Some auxiliary results}
\label{app-aux}

We here collect some technical results.

\begin{lemma}\label{lemma-tec-appendix-1}
  If  $y_{n+1}\geq (1-a_n)y_n $ with $0\leq a_n<1$ and
  $\sum_n a_n < +\infty$, then $\liminf_{n} y_n > 0$, provided $y_0>0$.
\end{lemma}

\begin{proof}
If $\sum_n a_n<+\infty$, then $\prod_{k=0}^n (1-a_k)$ converges to a
constant $L\in (0,1]$. Since we have $y_n\geq
  y_0\prod_{k=0}^{n-1}(1-a_k)$, we get $\liminf_n y_n\geq y_0L>0$.
\end{proof}

%

The following result is a slight generalization of 
\cite[Lemma A.1]{cri-dai-lou-min}.

\begin{lemma}\label{lemma-tecP}
If $a_n\geq 0$, $a_n \leq 1$ for $n$ large enough, $\sum_n a_n =
+\infty$, $\delta_n \geq 0$, $\sum_n \delta_n < +\infty$, $b>0$,
$y_n\geq 0$ and $y_{n+1} \leq (1-a_n)^b y_n + \delta_n$, then $\lim_n
y_n = 0$.
  \end{lemma}

\begin{proof}
Let $l$ be such that $a_n \leq 1$ for all $n \geq
l$. It holds
$$ y_n\leq y_l \Big( \prod_{i=l}^{n-1}(1-a_i) \Big)^b + \sum_{i=l}^{n-1}
\delta_i \Big( \prod_{j=i+1}^{n-1} (1-a_j)\Big)^b .
$$ Using the fact that $\sum_n a_n = +\infty$, it follows that
$\prod_{i=l}^{n-1}(1-a_i) \longrightarrow 0$. Moreover, for every $m
\geq l$,
\begin{align*}
\sum_{i=l}^{n-1} \delta_i \Big( \prod_{j=i+1}^{n-1} (1-a_j)\Big)^b & = 
\sum_{i=l}^{m-1} \delta_i \Big( \prod_{j=i+1}^{n-1} (1-a_j) \Big)^b +
\sum_{i=m}^{n-1} \delta_i \Big( \prod_{j=i+1}^{n-1}(1-a_j) \Big)^b  
\\ 
& \leq \Big( \prod_{j=m}^{n-1}
(1-a_j) \Big)^b \sum_{i=l}^{m-1} \delta_i + \sum_{i=m}^{+\infty}
\delta_i. 
\end{align*}

Using the fact that $ \prod_{j=m}^{n-1} (1-a_j)
\longrightarrow 0$ and that $\sum_n \delta_n < +\infty$, letting first $n
\to +\infty$ and then $m \to +\infty$ in the above formula, the
conclusion follows.
\end{proof}


\section{Complex roots of the unit and norms of complex matrices}
\label{app-norm-matrices}

\begin{lemma}\label{lem:sumRoots}
For $\nper \geq 1$ and $z \in \mathbb{Z}$, let $\lambda_{z}= \exp(
\tfrac{2\pi i}{\nper} z)$ be a $\nper$-root of the unity.  Then, we have 
\[
\sum_{h=0}^{\nper-1} \lambda_{z}^h = 
\begin{cases}
\nper & \text{if }\mod(z,\nper) = 0;
\\
0 & \text{otherwise.}
\end{cases}
\]
\end{lemma}

\begin{proof}
Let $\bar{z}=\mod(z,\nper)$, so that $\bar{z}\in\{0,\dots,\nper-1\}$
and $\lambda_z=\exp(\tfrac{2\pi i}{\nper} z)=\exp(\tfrac{2\pi
  i}{\nper} \bar{z})$.  For $\bar{z}=0$, it is trivial that
$\lambda_z=1$ and so $\sum_{h=0}^{\nper-1} \lambda_{z}^h=\nper$.  For
$\bar{z}\geq 1$, since $\lambda_z\neq1$ and $\lambda_z^{\nper}=1$,
from the formula of the geometric we have that
$$
\sum_{h=0}^{\nper-1} \lambda_{z}^h = \frac{1-\lambda_{z}^\nper}{1-\lambda_{z}} = 0.
$$\qedhere
\end{proof}

Now, recall that the $(p,q)$-operator norm $\vert\vert  A\vert\vert _{p,q}$ of a complex
matrix $A\in \mathbb{C}^{M\times N}$ is defined as (note that, in the
present paper, we write $\vert\vert \cdot\vert\vert $ instead of $\vert\vert \cdot\vert\vert _2$)
\[
\vert\vert  A\vert\vert _{p,q} = 
\sup_{\vx \neq \vzero} 
\frac{\vert\vert  A \vx \vert\vert _{q} }{\vert\vert  \vx \vert\vert _{p} }.
\]
We underline the following property of $\vert\vert  A\vert\vert _{2,2}$:
\begin{itemize}
\item for any $a_1,a_2 \in \mathbb{C}$ and $A_1,A_2\in \mathbb{C}^{M\times N}$
$
\vert\vert  a_1A_1 + a_2A_2 \vert\vert _{2,2} \leq 
\vert a_1\vert \vert\vert  A_1 \vert\vert _{2,2} + \vert a_2\vert \vert\vert  A_2 \vert\vert _{2,2} 
$;
\item
by definition, $\vert\vert  A\vert\vert _{2,2}$ is the \emph{spectral norm} of $A$, that
is well known to be the square root of the largest eigenvalue of the
matrix $A A^*$ (or $A^* A$), where $A^*$ is the conjugate transpose of
$A$:
\[
\vert\vert  A\vert\vert _{2,2} = \sigma_{\max}( A)
=\sqrt{\lambda_{\max}(A A^*)} = \sqrt{\lambda_{\max}(A^*A)};
\] 
\item the H\"{o}lder's inequality for matrices reads
\[
\vert\vert  A \vert\vert _{2,2}^2 \leq \vert\vert  A \vert\vert _{1,1} \vert\vert  A \vert\vert _{\infty,\infty} =
\Big( \max_{j= 1, \ldots, N} \Big\{\sum_{i=1}^N \vert a_{ij}\vert \Big\}\Big) 
\Big( \max_{i= 1, \ldots, M} \Big\{\sum_{j=1}^M \vert a_{ij}\vert \Big\}\Big) .
\]
\end{itemize}
Note that, for a self-adjoint square matrix $ \bar{A} $, the singular
values are the absolute values of the eigenvalues.  Then, if $ \bar{A}
= A A^*$, we have
\begin{multline*}
\vert\vert  A\vert\vert _{2,2}^2 = \lambda_{\max}(A A^*) = \sigma_{\max}( \bar{A} ) = 
\vert\vert  \bar{A} \vert\vert _{2,2}  
\\
\leq \Big( \max_{j= 1, \ldots, N} \Big\{\sum_{i=1}^N \vert \bar{a}_{ij}\vert \Big\}\Big) 
\Big( \max_{i= 1, \ldots, M} \Big\{\sum_{j=1}^M \vert \bar{a}_{ij}\vert \Big\}\Big) =
\Big( \max_{j= 1, \ldots, N} \Big\{\sum_{i=1}^N \vert \bar{a}_{ij}\vert \Big\}\Big)^2,
\end{multline*}
that implies $\vert\vert  A\vert\vert _{2,2} \leq \max_{j= 1, \ldots, N} \{\sum_{i=1}^N \vert \bar{a}_{ij}\vert \}$.
\\


\begin{thebibliography}{10}

\bibitem{alb-bar}
R.~Albert and A.-L. Barab{\'a}si.
\newblock Statistical mechanics of complex networks.
\newblock {\em Rev. Modern Phys.}, 74(1):47--97, 2002.

\bibitem{ale-cri-ghi}
G.~Aletti, I.~Crimaldi, and A.~Ghiglietti.
\newblock Synchronization of reinforced stochastic processes with a
  network-based interaction.
\newblock {\em Annals of Applied Probability}, 27:3787--3844, 2017.

\bibitem{ale-cri-ghi-MEAN}
G.~Aletti, I.~Crimaldi, and A.~Ghiglietti.
\newblock Networks of reinforced stochastic processes: asymptotics for the
  empirical means.
\newblock {\em Bernoulli}, 25(4 B):3339--3378, 2019.

\bibitem{ale-cri-ghi-WEIGHT-MEAN}
G.~Aletti, I.~Crimaldi, and A.~Ghiglietti.
\newblock Interacting reinforced stochastic processes: Statistical inference
  based on the weighted empirical means.
\newblock {\em Bernoulli}, 26(2):1098--1138, 2020.

\bibitem{ale-cri-ghi-barriers}
G.~Aletti, I.~Crimaldi, and A.~Ghiglietti.
\newblock Networks of reinforced stochastic processes: estimation of the
  probability of asymptotic polarization, 2023.

\bibitem{ale-cri-ghi-supplSPA1}
G.~Aletti, I.~Crimaldi, and A.~Ghiglietti.
\newblock Supplementary material for networks of reinforced stochastic
  processes: a complete description of the first-order asymptotics, 2023.

\bibitem{ale-ghi}
G.~Aletti and A.~Ghiglietti.
\newblock Interacting generalized {F}riedman's urn systems.
\newblock {\em Stoch. Process. Appl.}, 127(8):2650--2678, 2017.

\bibitem{are}
A.~Arenas, A.~D{\'{\i}}az-Guilera, J.~Kurths, Y.~Moreno, and C.~Zhou.
\newblock Synchronization in complex networks.
\newblock {\em Phys. Rep.}, 469(3):93--153, 2008.

\bibitem{benaim-sem}
M.~Benaim.
\newblock Dynamics of stochastic approximation algorithms.
\newblock {\em S\'eminaire de probabilit\'es XXXIII, Lecture notes in
  mathematics}, 1709, 1999.

\bibitem{ben}
M.~Bena{\"{\i}}m, I.~Benjamini, J.~Chen, and Y.~Lima.
\newblock A generalized {P}\'olya's urn with graph based interactions.
\newblock {\em Random Struct. Algor.}, 46(4):614--634, 2015.

\bibitem{che-luc}
J.~Chen and C.~Lucas.
\newblock A generalized {P}\'olya's urn with graph based interactions:
  convergence at linearity.
\newblock {\em Electron. Commun. Probab.}, 19:no. 67, 13, 2014.

\bibitem{cir}
P.~Cirillo, M.~Gallegati, and J.~H{\"u}sler.
\newblock A {P}\'olya lattice model to study leverage dynamics and contagious
  financial fragility.
\newblock {\em Adv. Complex Syst.}, 15(suppl. 2):1250069, 26, 2012.

\bibitem{cri-dai-lou-min}
I.~Crimaldi, P.~Dai~Pra, P.-Y. Louis, and I.~G. Minelli.
\newblock Synchronization and functional central limit theorems for interacting
  reinforced random walks.
\newblock {\em Stoch. Process. Appl.}, 129(1):70--101, 2019.

\bibitem{cri-dai-min}
I.~Crimaldi, P.~Dai~Pra, and I.~G. Minelli.
\newblock Fluctuation theorems for synchronization of interacting {P}\'{o}lya's
  urns.
\newblock {\em Stoch. Process. Appl.}, 126(3):930--947, 2016.

\bibitem{cri-lou-min}
I.~Crimaldi, P.-Y. Louis, and I.~G. Minelli.
\newblock Interacting non-linear reinforced stochastic processes:
  synchronization or non-synchronization.
\newblock {\em Advances in Applied Probability}, 55(1):forthcoming, 2023.

\bibitem{dai-lou-min}
P.~Dai~Pra, P.-Y. Louis, and I.~G. Minelli.
\newblock Synchronization via interacting reinforcement.
\newblock {\em J. Appl. Probab.}, 51(2):556--568, 2014.

\bibitem{egg-pol}
F.~Eggenberger and G.~P\'olya.
\newblock {\"{U}}ber die statistik verketteter vorg\"ange.
\newblock {\em ZAMM - Journal of Applied Mathematics and Mechanics /
  Zeitschrift f\"ur Angewandte Mathematik und Mechanik}, 3(4):279--289, 1923.

\bibitem{fortini}
S.~Fortini, S.~Petrone, and P.~Sporysheva.
\newblock {On a notion of partially conditionally identically distributed
  sequences}.
\newblock {\em Stoch. Process. Appl.}, 128(3):819--846, 2018.

\bibitem{super-urn-3}
M.~Hayhoe, F.~Alajaji, and B.~Gharesifard.
\newblock Curing epidemics on networks using a polya contagion model.
\newblock {\em IEEE/ACM Transactions on Networking}, PP, 11 2017.

\bibitem{super-urn-2}
M.~Hayhoe, F.~Alajaji, and B.~Gharesifard.
\newblock A polya contagion model for networks.
\newblock {\em IEEE Transactions on Control of Network Systems}, PP, 05 2017.

\bibitem{super-urn-1}
M.~Hayhoe, F.~Alajaji, and B.~Gharesifard.
\newblock A polya urn-based model for epidemics on networks.
\newblock In {\em 2017 American Control Conference (ACC)}, pages 358--363,
  2017.

\bibitem{KauSah21}
G.~Kaur and N.~Sahasrabudhe.
\newblock {Interacting Urns on a Finite Directed Graph}.
\newblock {\em J. Appl. Probab.}, page Forthcoming, 2022.

\bibitem{lau2}
M.~Launay.
\newblock Interacting urn models, 2011.

\bibitem{lau1}
M.~Launay and V.~Limic.
\newblock Generalized interacting urn models, 2012.

\bibitem{lima}
Y.~Lima.
\newblock Graph-based {P}\'olya's urn: completion of the linear case.
\newblock {\em Stoch. Dyn.}, 16(2):1660007, 13, 2016.

\bibitem{mah}
H.~M. Mahmoud.
\newblock {\em {P}\'olya urn models}.
\newblock Texts in Statistical Science Series. CRC Press, Boca Raton, FL, 2009.

\bibitem{mar-val}
M.~Marsili and A.~Valleriani.
\newblock Self organization of interacting {P}\'olya urns.
\newblock {\em European Physical Journal B}, 3(4):417--420, 1998.

\bibitem{new}
M.~E.~J. Newman.
\newblock {\em Networks: An introduction}.
\newblock Oxford University Press, Oxford, 2010.

\bibitem{pag-sec}
A.~M. Paganoni and P.~Secchi.
\newblock Interacting reinforced-urn systems.
\newblock {\em Adv. in Appl. Probab.}, 36(3):791--804, 2004.

\bibitem{pemantle-time-dependent}
R.~Pemantle.
\newblock A time-dependent version of {P}\'{o}lya's urn.
\newblock {\em J. Theoret. Probab.}, 3(4):627--637, 1990.

\bibitem{pem}
R.~Pemantle.
\newblock A survey of random processes with reinforcement.
\newblock {\em Probab. Surv.}, 4:1--79, 2007.

\bibitem{rob}
H.~Robbins and D.~Siegmund.
\newblock A convergence theorem for non negative almost supermartingales and
  some applications.
\newblock In J.~S. Rustagi, editor, {\em Optimizing Methods in Statistics},
  pages 233--257. Academic Press, New York, 1971.

\bibitem{sah}
N.~Sahasrabudhe.
\newblock {Synchronization and fluctuation theorems for interacting Friedman
  urns}.
\newblock {\em J. Appl. Probab.}, 53(4):1221--1239, 2016.

\bibitem{Sch01}
S.~J. Schreiber.
\newblock Urn models, replicator processes, and random genetic drift.
\newblock {\em SIAM J. Appl. Math.}, 61(6):2148--2167, 2001.

\bibitem{sidorova}
N.~Sidorova.
\newblock Time-dependent {P}\'olya urn.
\newblock 2018.

\bibitem{hof}
R.~van~der Hofstad.
\newblock {\em Random graphs and complex networks. {V}ol. 1}.
\newblock Cambridge Series in Statistical and Probabilistic Mathematics, [43].
  Cambridge University Press, Cambridge, 2017.

\end{thebibliography}

\end{document}